\documentclass[pdflatex,sn-mathphys]{sn-jnl}% Math and Physical Sciences Reference Style
\catcode`\|=12\relax

\usepackage[referable]{threeparttablex}
 \usepackage{booktabs}
 \usepackage{float}
 \usepackage{etoolbox}
 \makeatletter
 \patchcmd{\ps@headings}
 {\hbox to \hsize{\hfill Springer Nature 2021 \LaTeX\ template\hfill}}
 {\hbox to \hsize{}}
 {}
 {}
 \patchcmd{\ps@headings}
 {\hbox to \hsize{\hfill Springer Nature 2021 \LaTeX\ template\hfill}}
 {\hbox to \hsize{}}
 {}
 {}
 \patchcmd{\ps@titlepage}
 {\hbox to \hsize{\hfill Springer Nature 2021 \LaTeX\ template\hfill}}
 {\hbox to \hsize{}}
 {}
 {}

 \newcommand{\thickhline}{%
 	\noalign {\ifnum 0=`}\fi \hrule height 1pt
 	\futurelet \reserved@a \@xhline
 }

 \makeatother
 \usepackage{amssymb,amsmath,enumerate}

\usepackage{mathtools}
\usepackage[capitalize]{cleveref}
\usepackage{xspace}
\usepackage{tikz}
\usetikzlibrary{backgrounds,arrows,shapes,fit,shadows,decorations.markings,patterns}
\usepackage{contour}
\usepackage{tabularx}
\usepackage{adjustbox}
\usepackage{eqparbox}
\usepackage{stmaryrd}
\usepackage{xparse}

\makeatletter
\NewDocumentCommand{\eqmathbox}{o O{c} m}{%
  \IfValueTF{#1}
    {\def\eqmathbox@##1##2{\eqmakebox[#1][#2]{$##1##2$}}}
    {\def\eqmathbox@##1##2{\eqmakebox{$##1##2$}}}
  \mathpalette\eqmathbox@{#3}
}
\makeatother

\usepackage{todonotes}

\newcommand{\eg}{e.\,g.,\xspace}
\newcommand{\ie}{i.\,e.,\xspace}
\newcommand{\Wlog}{W.\,l.\,o.\,g.\xspace}
\newcommand{\wrt}{w.\,r.\,t.\xspace}
\newcommand{\rewrite}[2]{\;\smash{\overset{\scriptscriptstyle\smash{#1\;}}{\underset{\raisebox{2pt}{$\scriptscriptstyle\smash{#2\;}$}}{\Longrightarrow}}}\;}
\newcommand{\IGamma}{I}

\def\u{u}

\newcommand{\PSPACE}{\ensuremath{\mathsf{PSPACE}}\xspace} %%%

\newcommand{\NP}{\ensuremath{\mathsf{NP}}\xspace} %
\newcommand{\coNP}{\ensuremath{\mathsf{coNP}}\xspace}
\newcommand{\LOGCFL}{\ensuremath{\mathsf{LOGCFL}}\xspace} %
\newcommand{\DLOGTIME}{\ensuremath{\mathsf{DLOGTIME}}\xspace} %
\renewcommand{\L}{\ensuremath{\mathsf{LOGSPACE}}\xspace} %
\newcommand{\LOGSPACE}{\L} %
\renewcommand{\L}{\ensuremath{\mathsf{L}}\xspace} %
\newcommand{\NL}{\ensuremath{\mathsf{NL}}\xspace} %

\newcommand{\Tc}[1]{\ensuremath{\mathsf{uTC}^{#1}}\xspace}
\newcommand{\Ac}[1]{\ensuremath{\mathsf{AC}^{#1}}\xspace}
\newcommand{\uAc}[1]{\ensuremath{\mathsf{uAC}^{#1}}\xspace}
\newcommand{\Nc}[1]{\ensuremath{\mathsf{uNC}^{#1}}\xspace}
\newcommand{\AC}{\ensuremath{\mathsf{AC}}\xspace}
\newcommand{\TC}{\ensuremath{\mathsf{uTC}}\xspace}

\newcommand{\uACz}{\ensuremath{\mathsf{uAC}^0}\xspace}
\newcommand{\uTCz}{\ensuremath{\mathsf{uTC}^0}\xspace}

\newcommand{\uAC}{\ensuremath{\mathsf{uAC}}\xspace}
\newcommand{\uTC}{\ensuremath{\mathsf{uTC}}\xspace}
\newcommand{\uNC}{\ensuremath{\mathsf{uNC}}\xspace}
\renewcommand{\P}{\ensuremath{\mathsf{P}}\xspace}
\newcommand{\GapL}{\ensuremath{\mathsf{GapL}}\xspace}
\newcommand{\CeqL}{\ensuremath{\mathsf{C}_=\mathsf{L}}\xspace}

\DeclareMathOperator{\sgn}{sgn}
\DeclareMathOperator{\link}{link}

\DeclareMathOperator{\nf}{nf}
\DeclareMathOperator{\dist}{dist}

\DeclareMathOperator{\alphabet}{alph}
\DeclareMathOperator{\IRR}{IRR}
\DeclareMathOperator{\GP}{GP}
\DeclareMathOperator{\UWP}{UWP}
\DeclareMathOperator{\PowWP}{PowWP}

\DeclareMathOperator{\UPowWP}{UPowWP}
\DeclareMathOperator{\SPowWP}{SPowWP}
\DeclareMathOperator{\USPowWP}{USPowWP}

\DeclareMathOperator{\GSPowWP}{GSPowWP}

\newcommand{\hookdoubleheadrightarrow}{%
  \hookrightarrow\mathrel{\mspace{-15mu}}\rightarrow
}

\newcommand{\sse}{\subseteq}

\newcommand{\Sig}{\Sigma}

\newcommand{\compproblem}[3][]{%
	\par\vspace{0.125cm plus 0.1cm minus 0.05cm}\adjustbox{valign=t}{\begin{tabularx}{\linewidth-2\parindent}{@{}lX}%
			\if\relax\detokenize{#1}\relax%
			\else%
			\textnormal{\textsf{Constant:}}&#1\\%
			\fi%
			\textnormal{\textsf{Input:}}&#2\\%
			\textnormal{$\mathrlap{\textsf{Output:}}\hphantom{\textsf{Question:}}$\!\!}&#3\\%
	\end{tabularx}}\vspace{0.125cm plus 0.1cm minus 0.05cm}\par%
}

\newcommand{\ynproblem}[3][]{%
	\par\vspace{0.125cm plus 0.1cm minus 0.05cm}\adjustbox{valign=t}{\begin{tabularx}{\linewidth-2\parindent}{@{}lX}%
			\if\relax\detokenize{#1}\relax%
			\else%
			\textnormal{\textsf{Constant:}}&#1\\%
			\fi%
			\textnormal{\textsf{Input:}}&#2\\%
			\textnormal{\textsf{Question:}}&#3\\%
	\end{tabularx}}\vspace{0.125cm plus 0.1cm minus 0.05cm}\par%
}

\renewcommand{\set}[2]{\left\{\, \mathinner{#1}\vphantom{#2}\: \left|\: \vphantom{#1}\mathinner{#2} \right.\,\right\}}

\newcommand{\abs}[1]{\left|\mathinner{#1}\right|}

\newcommand{\genr}[2]{\left< \, \mathinner{#1}\; \middle|\;\mathinner{#2} \, \right>}

\newcommand{\cA}{\mathcal{A}}

\newcommand{\cL}{\mathcal{L}}

\newcommand{\cC}{\mathcal{C}}

\newcommand{\N}{\ensuremath{\mathbb{N}}}
\renewcommand{\Z}{\ensuremath{\mathbb{Z}}}

\newcolumntype{"}{@{\hskip\tabcolsep\vrule width 1pt\hskip\tabcolsep}}
\makeatother

\jyear{2022}%

\theoremstyle{thmstyleone}%
\newtheorem{theorem}{Theorem}%  meant for continuous numbers
\newtheorem{lemma}[theorem]{Lemma}%  meant for continuous numbers
\newtheorem{proposition}[theorem]{Proposition}%
 \newtheorem{corollary}[theorem]{Corollary}

\theoremstyle{thmstyletwo}%
\newtheorem{example}[theorem]{Example}%
\newtheorem{remark}[theorem]{Remark}%

\theoremstyle{thmstylethree}%
\newtheorem{definition}[theorem]{Definition}%

\begin{document}

  \title[The Power Word Problem]{The Power Word Problem in Graph Products}

%%=============================================================%%
%% Prefix	-> \pfx{Dr}
%% GivenName	-> \fnm{Joergen W.}
%% Particle	-> \spfx{van der} -> surname prefix
%% FamilyName	-> \sur{Ploeg}
%% Suffix	-> \sfx{IV}
%% NatureName	-> \tanm{Poet Laureate} -> Title after name
%% Degrees	-> \dgr{MSc, PhD}
%% \author*[1,2]{\pfx{Dr} \fnm{Joergen W.} \spfx{van der} \sur{Ploeg} \sfx{IV} \tanm{Poet Laureate}
%%                 \dgr{MSc, PhD}}\email{iauthor@gmail.com}
%%=============================================================%%

  \author*[1]{\fnm{Markus} \sur{Lohrey}}\email{\normalsize lohrey@eti.uni-siegen.de}

  \author*[2]{\fnm{Florian} \sur{Stober}}\email{\normalsize florian.stober@fmi.uni-stuttgart.de}
%  \equalcont{These authors contributed equally to this work.}

  \author*[2]{\fnm{Armin} \sur{Wei\ss}}\email{\normalsize armin.weiss@fmi.uni-stuttgart.de}
%  \equalcont{These authors contributed equally to this work.}

  \affil[1]{
  	%\orgdiv{Department},
  	 \orgname{\normalsize Universit{\"a}t Siegen},
  	  %\orgaddress{\street{Street}, \city{City}, \postcode{100190}, \state{State},
  \country{\normalsize Germany}}

  \affil[2]{
  	%\orgdiv{Department},
  	\orgname{\normalsize Universit{\"a}t Stuttgart},
  	%\orgaddress{\street{Street}, \city{City}, \postcode{10587}, \state{State},
  		\country{\normalsize Germany}}

  \abstract{\footnotesize The power word problem for a group $G$ asks whether an expression $u_1^{x_1} \cdots u_n^{x_n}$, where the $u_i$ are words over a finite set of generators of $G$ and the $x_i$ binary encoded integers, is equal to the identity of $G$.
  	 It is a restriction of the compressed word problem, where the input word is represented by
  	a straight-line program (i.e., an algebraic circuit over $G$).
    We start by showing some easy results concerning the power word problem. In particular, the power word problem for a group $G$ is $\Nc1$-many-one reducible to the
    power word problem for a finite-index subgroup of $G$.

    For our main result, we consider graph products of groups that do not have elements of order two.
  	We show that the power word problem in a fixed such graph product is $\AC^0$-Turing-reducible to the word problem for the free group $F_2$ and the power word problems of the base groups.
  	Furthermore, we look into the uniform power word problem in a graph product, where the dependence graph and the base groups are part of the input.
  	Given a class of finitely generated groups $\mathcal{C}$ without order two elements, the uniform power word problem in a graph product can be solved in $\AC^0[\CeqL^{\UPowWP(\mathcal{C})}]$, where $\UPowWP(\mathcal{C})$ denotes the uniform
   power word problem for groups from the class $\mathcal{C}$.
  	As a consequence of our results, the uniform knapsack problem in right-angled Artin groups is \NP-complete.
  	The present paper is a combination of the two conference papers \cite{LohreyW19,StoberW22}.

   In \cite{StoberW22} and previous iterations of this paper our results on graph products were wrongly stated without the additional assumption that the base groups do not have elements of order two.  In the present work we correct this mistake. While we strongly conjecture that the result as stated in \cite{StoberW22} is true, our proof relies on this additional assumption.
  }

  \keywords{word problem, power word problem, compressed word problem, right-angled Artin groups, nilpotent groups, Grigorchuk group, finite index subgroups}

%%\pacs[JEL Classification]{D8, H51}

%%\pacs[MSC Classification]{35A01, 65L10, 65L12, 65L20, 65L70}

  \maketitle

  {\small\bmhead{\small Acknowledgments}

\small  Markus Lohrey has been funded by DFG (Deutsche Forschunggemeinschaft) project LO 748/12-2.
 Armin Weiß has been funded by DFG project DI 435/7-1 and DFG project WE 6835/1-2.
}
\newpage
\tableofcontents

\section{Introduction}

Algorithmic problems in group theory have a long tradition, going back to the work of Dehn from 1911 \cite{dehn11}.
One of the fundamental group theoretic decision problems introduced by Dehn is the {\em word problem}
for a finitely generated group $G$ (with a fixed finite generating set $\Sigma$): does a given word $w \in \Sigma^*$
evaluate to the group identity? Novikov \cite{nov55} and Boone \cite{boone59} independently proved in the 1950's the existence
of finitely presented groups with undecidable word problem. On the positive side, in many important classes
of groups the word problem is decidable, and in many cases also the computational complexity is quite low.
Famous examples are finitely generated linear groups, where the word problem belongs to deterministic logarithmic
space ($\L$ for short) \cite{lz77} and hyperbolic groups where the word problem can be solved in linear time \cite{Ho} as well as in \LOGCFL
\cite{Lo05ijfcs}.

In recent years, also compressed versions of group theoretical decision problems,
where input words are represented in a succinct form, have attracted attention. One such succinct representation
are so-called straight-line programs, which are context-free grammars that produce exactly one word. The size
of such a grammar can be much smaller than the word it produces.
For instance, the word $a^n$ can be produced by a straight-line program of size $\mathcal{O}(\log n)$.
For the {\em compressed word problem} for the group $G$ the input consists of a straight-line program that produces a word $w$
over the generators of $G$ and it is asked whether $w$ evaluates to the identity element of $G$.
This problem is a reformulation of the circuit evaluation problem for $G$. The compressed word problem
naturally appears when one tries to solve the word problem in automorphism groups or semidirect
products \cite[Section~4.2]{LohreySpringerbook2014}.
For the following classes of groups, the compressed word problem is known
to be solvable in polynomial time: finite groups (where the compressed word problem is either \P-complete or in \Nc{2}
\cite{BeMcPeTh97}), finitely generated nilpotent groups \cite{KoenigL17} (where the complexity is even in \Nc{2}),
hyperbolic groups \cite{HoLoSchl19} (in particular, free groups), and  virtually special groups
(i.e, finite extensions of subgroups of right-angled Artin groups) \cite{LohreySpringerbook2014}.
The latter class covers for instance Coxeter groups, one-relator groups with torsion, fully residually free groups and
fundamental groups of hyperbolic 3-manifolds. For finitely generated linear groups there is still a randomized polynomial
time algorithm for the compressed word problem \cite{LohreyS07,LohreySpringerbook2014}.
Simple examples of groups where the compressed word problem is intractable are wreath products $G \wr \mathbb{Z}$ with $G$ finite non-solvable: for every such group the
compressed word problem is \PSPACE-complete \cite{BartholdiFLW20}, whereas as the (ordinary) word problem for $G \wr \mathbb{Z}$ is in $\Nc{1}$ \cite{Waack90}.

In this paper, we study a natural restriction of the compressed word problem called the {\em power word problem}.
An input for the power word problem for the group $G$ is a tuple $(u_1, x_1, u_2, x_2, \ldots, u_n, x_n)$ where every $u_i$
is a word over the group generators and every $x_i$ is a binary encoded integer (such a tuple is called a {\em power word});
the question is whether $u_1^{x_1} u_2^{x_2}\cdots u_n^{x_n}$ evaluates to the group identity of $G$. This problem naturally
arises in the context of the so-called knapsack problem; we will explain more about this later.

From a power word $(u_1, x_1, u_2, x_2, \ldots, u_n, x_n)$
one can easily (\eg\ by an $\uAc{0}$-reduction) compute a straight-line program for the word
$u_1^{x_1} u_2^{x_2}\cdots u_n^{x_n}$. In this sense, the power word problem is at most as difficult as the compressed
word problem. On the other hand, both power words and straight-line programs achieve exponential compression in the best case; so the additional difficulty of the the compressed word problem does not come from a higher compression rate but rather because straight-line programs can generate more ``complex'' words.

Our main results for the power word problem are the following; in each case we compare our results with
the corresponding results for the compressed word problem:\footnote{All circuit complexity classes are assumed to be uniform in this paper, see
Section~\ref{sec-complexity} for more details.}
\begin{itemize}
	\item The power word problem for every finitely generated nilpotent group is in $\Tc{0}$
	and hence has the same complexity as the word problem (or the problem of multiplying binary encoded integers).
	The proof is a straightforward adaption
	of a proof from \cite{MyasnikovW17}. There, the special case, where all words $u_i$ in the input power word
	are single generators, was shown to be in $\Tc{0}$.
	The compressed word problem for every finitely generated nilpotent group belongs to the
	class $\mathsf{DET} \subseteq \Nc{2}$ and is hard for the counting class $\mathsf{C}_{=}\mathsf{L}$ in case
	of a torsion-free nilpotent group \cite{KoenigL17}.

	\medskip

	\item The power word problem for the Grigorchuk group is $\uAC^0$-many-one-reducible to its word problem.
	Since the word problem for the Grigorchuk group is in $\LOGSPACE$ \cite{GZ,BartholdiFLW20}, also the power word problem is in
	$\LOGSPACE$. Moreover, in \cite{BartholdiFLW20}, it is shown that the compressed word problem for the Grigorchuk
	group is $\PSPACE$-complete. Hence, the Grigorchuk group is an example of a group for which the compressed word problem is provably
	more difficult than the power word problem.

	\medskip

	\item  The power word problem for a finitely generated group $G$ is \Nc{1}-many-one-reducible to the power word problem for any finite index subgroup of $G$.
	An analogous result holds for the compressed word problem as well \cite{KoenigL17}.

	\medskip

	\item If $G$ is a graph product of finitely generated groups $G_1, \ldots, G_n$ (the so-called base groups) not containing any elements of order two, then the power word problem in $G$ can be decided in $\uAC^0$ with oracle gates for (i) the word problem for the free group $F_2$ and (ii)
	 the power word problems for the base groups $G_i$.
	 In order to define a graph product of groups $G_1, \ldots, G_n$, one needs a graph with vertices $1, \ldots, n$. The corresponding graph product is
	 obtained as the quotient of the
	 free product of $G_1, \ldots, G_n$ modulo the commutation relation that allows elements of $G_i$ to commute with elements of $G_j$ iff $i$ and $j$ are adjacent in the graph.
	 Graph products were introduced by Green in 1990~\cite{green1990graph}.  The compressed word problem for a graph product is polynomial time Turing-reducible to the compressed word problems for the the base groups~\cite{haubold2012compressed}.

	 \medskip

	 \item A right-angled Artin group (RAAG) can be defined as a graph product of copies of $\mathbb{Z}$.
	 As a corollary of our transfer theorem for graph products, it follows that the power word problem for a RAAG can be decided
	 in $\uAC^0$ with oracle gates for the word problem for the free group $F_2$. The same upper complexity bound was shown before by Kausch \cite{kausch2017parallel}
	 for the ordinary word problem for a RAAG and in \cite{LohreyW19} for the power word problem for a finitely generated free group.
	 As a consequence of our new result, the power word problem for a RAAG is in \LOGSPACE (for the ordinary word problem this follows
	 from the well-known fact that RAAGs are linear groups together with the above mentioned result of Lipton and Zalcstein \cite{lz77}).
	  The compressed word problem for every RAAG is in \P (polynomial time) and \P-complete if the RAAG is non-abelian~\cite{LohreySpringerbook2014}.

	\end{itemize}
In all the above mentioned results, the group is fixed, i.e., not part of the input. In general, it makes no sense to input an arbitrary finitely generated
group, since there are uncountably many such groups. On the other hand,
if we restrict to finitely generated groups with a finitary description, one may also consider a uniform version of the word problem/power word problem/compressed word problem,
where the group is part of the input. We will consider the uniform power word problem for graph products for a fixed countable class $\mathcal{C}$ of finitely generated groups. We assume
that the groups in $\mathcal{C}$ have a finitary description.\footnote{We assume that the description of a group $G \in \mathcal{C}$ contains a finite generating set.
A typical example might be the class $\mathcal{C}$ of finitely generated matrix groups over the field $\mathbb{Q}$. In this case, the description
of a group $G$ would consist of an integer $d \ge 1$ (the dimension) and a list of matrices $A_1, \ldots, A_n \in \mathsf{GL}_d(\mathbb{Q})$ (the generators
of the matrix group). Other examples are classes of finitely presented groups given by particular finite presentations, \eg{} hyperbolic groups given as a Dehn presentation.  The precise detail of the description of groups will be not important for us.}
 Then a graph product is given by a list $G_1, \ldots G_n$ of base groups from $\mathcal{C}$ together with
an undirected graph on the indices $1, \ldots, n$. For this setting Kausch \cite{kausch2017parallel} proved that the uniform word problem for graph products
belongs to $\CeqL^{\UWP(\mathcal{C})}$, i.e., the counting logspace class $\CeqL$ with an oracle for the uniform word problem for the class $\mathcal{C}$
(we write $\UWP(\mathcal{C})$ for the latter). We extend this result to the power word problem under the additional assumption that no group in $\cC$ contains an element of order two.
More precisely, we show that the uniform power word problem for graph products over that class $\mathcal{C}$ of base groups belongs to the closure of
$\CeqL^{\UPowWP(\mathcal{C})}$ under $\uAC^0$-Turing-reductions, where $\UPowWP(\mathcal{C})$ denotes the uniform power word problem for
the class $\mathcal{C}$. Analogous results for the uniform compressed word problem are not known. Indeed,
whether the uniform compressed word problem for RAAGs is solvable in polynomial time is posed as an open problem in~\cite{lohrey2007efficient}.

Our result for the uniform power word problem for graph products implies that the uniform power word problem for RAAGs can be solved
in polynomial time. We can apply this result to the knapsack problem for RAAGs.
 The knapsack problem is a classical optimization problem that originally has been formulated for the integers.
 Myasnikov et al.\ introduced the decision variant of the knapsack problem for an arbitrary finitely generated group $G$:
  Given $g_1, \dots, g_n, g \in G$, decide whether there are $x_1, \dots, x_n \in \mathbb{N}$ such that $g_1^{x_1} \cdots g_n^{x_n} = g$ holds in the group $G$~\cite{myasnikov2015knapsack}, see also \cite{FigeliusGLZ20,ganardiklz18,KoenigLohreyZetzsche2015a,LohreyZ18} for further work.
 For many groups $G$ one can show that,
if such $x_1, \ldots, x_n \in \mathbb{N}$ exist, then there exist such numbers of size $2^{\text{poly}(N)}$, where
$N$ it the total length of all words representing the group elements $g_1, \dots, g_n, g$.
This holds for instance for RAAGs. In this case, one nondeterministically guesses the binary encodings of numbers $x_1, \ldots, x_n$
and then verifies, using an algorithm for the power word problem, whether $g_1^{x_1} \cdots g_n^{x_n} g^{-1} = 1$
holds. In this way, it was shown in \cite{LohreyZ18} that for every RAAG the knapsack problem belongs to \NP
(using the fact that the compressed word problem and hence the power word problem for a fixed RAAG
belongs to \P). Moreover, if the commutation graph of the RAAG $G$ contains an induced subgraph $C_4$ (cycle on 4 nodes) or $P_4$ (path on 4 nodes),
then the knapsack problem for $G$ is $\NP$-complete~\cite{LohreyZ18}.
However, membership of the uniform version of the knapsack problem for RAAGs in $\NP$ remained open.
Our polynomial time algorithm for the uniform power word problem for RAAGs yields the missing piece:
the uniform knapsack problem for RAAGs is indeed $\NP$-complete.

\paragraph*{Related work}

Implicitly, (variants of) the power word problem have been studied long before.
In the commutative setting, Ge \cite{Ge93} has shown that one can verify in polynomial time an identity
$\alpha_1^{x_1} \alpha_2^{x_2}\cdots \alpha_n^{x_n} = 1$, where the $\alpha_i$ are elements of an algebraic
number field and the $x_i$ are binary encoded integers.

In \cite{GurevichS07}, Gurevich and Schupp present a polynomial time
algorithm for a compressed form of the subgroup membership problem for a free group $F$
where group elements are represented in the form $a_1^{x_1} a_2^{x_2}\cdots a_n^{x_n}$ with binary encoded integers $x_i$.
The $a_i$ must be, however, standard generators of the free group $F$. This is the same input representation as in
\cite{MyasnikovW17} (for nilpotent groups) and is more restrictive then our setting, where
we allow powers of the form $w^x$ for $w$ an arbitrary word over the group generators (on the other hand,
Gurevich and Schupp consider the subgroup membership problem, which is more general than the word problem).

Recently, the power word problem has been investigated in \cite{FigeliusGLZ20}.
In \cite{FigeliusGLZ20} it is shown that the power word problem for a wreath product of the form $G \wr \mathbb{Z}$ with $G$ finitely
generated nilpotent belongs to $\TC^0$. Moreover, the power word problem for iterated wreath products of the form
$\Z^r \wr (\Z^r \wr (\Z^r \cdots ))$ belongs to $\TC^0$. By a famous embedding theorem of Magnus \cite{Mag39}, it follows
that the power word problem for a free solvable groups is in $\TC^0$. Finally, in \cite{LohreyZ20} Zetzsche and the first author of this work showed that the power word
problem for a solvable Baumslag-Solitar group $\mathsf{BS}(1,q)$ belongs to $\TC^0$.

The present paper is a combination of the two conference papers \cite{LohreyW19} (by the first and third author) and \cite{StoberW22} (by the second and third author).
Here we also correct a mistake that occurred in \cite{StoberW22} and version 2 of this paper (see \cite{StoberW22arxivOld}): there, our results on graph products were stated without the additional assumption that the base groups do not have elements of order two. While we strongly conjecture this result to be true, our proof only works with this additional assumption. The key lies in the proof of \cref{lem:gp-rewrite_bounds_exponents} (which corresponds to Lemma 15 in \cite{StoberW22,StoberW22arxivOld})~-- indeed, the only place where we need this additional assumption. We give more technical details in \cref{rem:mistake_preprocessing} and \cref{rem:mistake}.

%Unfortunately, we only realized this during the preparation of the present manuscript.

\section{Preliminaries}\label{sec:prelims} %\section{Notation}\label{sec:notation}

For integers $a \leq b$ we write $[a,b]$ for the interval
$\set{x \in \Z}{a \leq x \leq b}$.
For an integer $z \in \mathbb{Z}$ let
us define $\llbracket x \rrbracket = [0,x]$ if $x \geq 0$
and $\llbracket x \rrbracket = [x,0]$ if $x < 0$.

\subsection{Words}

An \emph{alphabet} is a (finite or infinite) set $\Sig$; an element $a \in \Sig$ is called a \emph{letter}.
The free monoid over $\Sig$ is denoted by $\Sig^*$; its elements
are called {\em words}.
The multiplication of the free monoid is concatenation of words. The identity element is the empty word $1$.

Consider a word $w = a_1 \cdots a_n$ with $a_i \in \Sigma$.
For $A \subseteq \Sigma$ we write $\left\lvert w\right\rvert_A$ for the number
of $i \in [1,n]$ with $a_i \in A$ and we set $|w| = |w|_\Sigma$ (the length of $w$) and $|w|_a = |w|_{\{a\}}$ for $a \in \Sigma$.
A word $w$ has \emph{period} $k$ if $a_i= a_{i+k}$ for all $i$ with $i,i+k \in [1,n]$.

\subsection{Monoids}
  Let $M$ be an arbitrary monoid.
  Later, we will consider finitely generated monoids $M$, where elements of $M$ are described by words over an alphabet of monoid generators.
  To distinguish equality as words from equality as elements of $M$, we also write $x=_My$ (or $x=y$ in $M$) to indicate equality in $M$ (as opposed to equality as words). Let $x =_M uvw$ for some $x, u, v, w \in M$.
  We say $u$ is a \textit{prefix} of $x$, $v$ is a \textit{factor} of $x$, and $w$ is a \textit{suffix} of $x$.
  We call $u$ a \textit{proper prefix} if $u \neq x$.
  Similarly, $v$ is a \textit{proper factor} if $v \neq x$ and $w$ is a \textit{proper suffix} if $w \neq x$.

  An element $u\in M$ is \emph{primitive} if $u\neq_M v^k$ for any $v \in M$ and $k >1$.
  Two elements $u, v \in M$ are \emph{transposed} if there are $x, y \in M$ such that $u =_M xy$ and $v =_M yx$.
  We call $u$ and $v$ \emph{conjugate} if there is an element $t \in M$ such that $ut = tv$ (note that this is also sometimes called left-conjugate in the literature).
  For a free monoid $\Sigma^*$, two words $u,v$ are transposed if and only if they are conjugate. In this case, we also say that the word $u$ is a \emph{cyclic permutation} of the word $v$.

\newcommand\SGr{\Sig^*}

\subsection{Rewriting systems over monoids} \label{sec-rewriting}

A \emph{rewriting system} over the monoid $M$ is a subset $S \subseteq M \times M$. We write $\ell \to r$ if $(\ell, r) \in S$. The corresponding \emph{rewriting relation} $\rewrite{}{S}$ over $M$ is
defined by: $u \rewrite{}{S} v$ if and only if there exist $\ell\to r\in S$ and $s,t \in M$ such that
$u =_M s\ell t$ and $v =_M sr t$. We also say that $u$ can be rewritten to $v$ in one step.
Let $\rewrite{+}{S}$ be the transitive closure of $\rewrite{}{S}$ and $\rewrite{*}{S}$ the reflexive and transitive closure of $\rewrite{}{S}$.
 We write $u \rewrite{\le k}{S} v$ to denote that $u$ can be rewritten to $v$ using at most $k$ steps.
We say that $w \in M$ is \textit{irreducible} with respect to $S$ if there is no $v \in M$ with $w \rewrite{}{S} v$.
  The set of irreducible monoid elements is denoted as
  $
  \IRR(S) = \left\{ w \in M \mid w \text{ is irreducible} \right\}
  $.
  A rewriting system $S$ is called \emph{confluent} if, whenever $x\rewrite{*}{S}y$ and $x\rewrite{*}{S}z$, then there is some $w$ with $y\rewrite{*}{S}w$ and $z\rewrite{*}{S}w$. Note that if $S$ is confluent, then for each $v$ there is at most one $w \in \IRR(S)$ with $v\rewrite{*}{S}w$.
   A rewriting system $S$ is called \emph{terminating} if there is no infinite chain
\[x_0 \rewrite{}{S} x_1 \rewrite{}{S} \cdots x_{i-1} \rewrite{}{S} x_i \rewrite{}{S} \cdots.\]
  We write $M/S$ for
  the quotient monoid $M/\equiv_S$, where $\equiv_S$ is the smallest congruence
  relation on $M$ that contains $S$.

  The above notion of a rewriting system over a monoid $M$ is a generalization
  of the notion of a {\em string rewriting system}, which is a rewriting system over a free monoid $\Sigma^*$.  For further details on rewriting systems we refer to \cite{bo93springer,jan88eatcs}.

\subsection{Partially commutative monoids} \label{sec-traces}

In this subsection, we introduce a few basic notations concerning
partially commutative monoids. More information can be found in \cite{DieRoz95}.

  Let $\Sigma$ be an alphabet of symbols.
  We do not require $\Sigma$ to be finite.
  Let $I \subseteq \Sigma \times \Sigma$ be a symmetric and irreflexive relation.
  The {\em partially commutative monoid} defined by $(\Sigma, I)$ is the quotient monoid
  \[ M(\Sigma, I) = \Sigma^*/\{(ab,ba)\mid (a,b)\in I\}.\]
  Thus, the relation $I$ describes which generators commute; it is called the \emph{commutation relation} or \emph{independence relation}.
  The relation $D = (\Sigma \times \Sigma) \setminus I$ is called \emph{dependence relation} and $(\Sigma, D)$  is called a \emph{dependence graph}.
The monoid $M(\Sigma, I)$ is also called a {\em trace monoid} and its elements are called \emph{traces} or \emph{partially commutative words}.
Note that for words $u,v \in \Sigma^*$ with $u =_{M(\Sigma, I)} v$ we have $\abs{u}_a=\abs{v}_a$ for every $a \in \Sigma$.
Hence, the length $\abs{w}$ and $\abs{w}_a$ for a trace $w \in M(\Sigma, I)$  is well-defined and we use this notation henceforth.

A letter $a$ is called a {\em minimal letter} of $w \in M(\Sigma,I)$ if
$w =_{M(\Sigma, I)} au$ for some $u\in M(\Sigma,I)$.
Likewise a letter $a$ is called a {\em maximal letter} of $w$ if $w =_{M(\Sigma, I)} ua$ for some $u\in M(\Sigma,I)$.
When we say that $a$ is minimal (maximal)
in $w \in \Sigma^*$, we mean that $a$ is minimal (maximal) in the trace represented by $w$. Note that if both $a$ and $b \neq a$ are minimal (maximal) letters of $w$, then
$(a,b) \in I$.
A \emph{trace rewriting system} is simply a rewriting system over a trace monoid
$M(\Sigma, I)$ in the sense of \cref{sec-rewriting}. If $\Delta \sse \Sigma$ is a subset, we write $M(\Delta, I)$ for the submonoid of $M(\Sigma, I)$ generated by $\Delta$.

  Elements of a partially commutative monoid can represented by directed acyclic graphs: Let $w = a_1 \cdots a_n$ with $a_i \in \Sigma$. We define the \emph{dependence graph} of $w$ as follows:
  The node set is $[1,n]$ and there is an edge $i \to j$ if and only if $i < j$ and $(a_i, a_j) \in D$. Then, for two words $u,v \in \Sigma^*$ we have
  $u =_{M(\Sigma,I)} v$  if and only if the dependence graphs of $u$ and $v$ are isomorphic (as labeled directed graphs).
  The dependence graph of a trace $v \in M(\Sigma, I)$ is the dependence graph of (any) word representing $v$.
  The trace $v$ is said to be \emph{connected} if its dependence graph is weakly connected, or, equivalently, if the induced subgraph of
   $(\Sigma,D)$ consisting only of the letters occurring in $v$ is connected.
   The \emph{connected components} of the trace $v$ are the weakly connected components of the dependence graph of $v$.

\subsubsection{Levi's lemma}

 As a consequence of the representation of traces
by dependence graphs,
one obtains Levi's lemma for traces (see e.g.
\cite[p.~74]{DieRoz95}), which is one of the fundamental
facts in trace theory. The formal statement is as follows.

\begin{lemma}[Levi's lemma] \label{lemma-levi}
Let $M = M(\Sigma,I)$ be a trace monoid and
$u_1, \ldots, u_m, v_1, \ldots, v_n \in M $. Then
\[ u_1u_2 \cdots u_m =_M v_1 v_2 \cdots  v_n\] if and only if
there exist $w_{i,j} \in M(\Sigma,I)$ (for $i \in [1,m]$, $j \in [1,n]$) such that
\begin{itemize}
\item $u_i =_M w_{i,1}w_{i,2}\cdots w_{i,n}$ for every $i \in [1,m]$,
\item $v_j =_M w_{1,j}w_{2,j}\cdots w_{m,j}$ for every $j \in [1,n]$, and
\item $(w_{i,j}, w_{k,\ell}) \in I$ if $1 \leq i < k \leq m$ and $n \geq j > \ell \geq 1$.
\end{itemize}
\end{lemma}
\noindent
The situation in the lemma will be visualized by a
diagram of the following kind. The $i$--th column
corresponds to $u_i$, the $j$--th row (read from bottom to top)
corresponds to $v_j$, and the intersection of the $i$--th
column and the $j$--th row represents $w_{i,j}$.
Furthermore, $w_{i,j}$ and $w_{k,\ell}$ are independent
if one of them is left-above the other one. So, for instance,
all $w_{i,j}$ in the red part are independent from all $w_{k,\ell}$ in the blue part.

\medskip

\begin{center}
  \begin{tabular}{c"c|c|c|c|c|}\hline
  $v_n$  & \textcolor{red}{$w_{1,n}$} & \textcolor{red}{$w_{2,n}$} & $w_{3,n}$ & \dots  & $w_{m,n}$ \\ \hline
  \vdots & \vdots    & \vdots    & \vdots    & \vdots & \vdots    \\ \hline
  $v_3$  & \textcolor{red}{$w_{1,3}$} & \textcolor{red}{$w_{2,3}$} & $w_{3,3}$ & \dots  & $w_{m,3}$ \\ \hline
  $v_2$  & $w_{1,2}$ & $w_{2,2}$ & \textcolor{blue}{$w_{3,2}$} & \dots  & \textcolor{blue}{$w_{m,2}$} \\ \hline
  $v_1$  & $w_{1,1}$ & $w_{2,1}$ & \textcolor{blue}{$w_{3,1}$} & \dots  & \textcolor{blue}{$w_{m,1}$} \\ \thickhline
         & $u_1$     & $u_2$     & $u_3$     & \dots  & $u_m$
  \end{tabular}
  \end{center}

\medskip\noindent
Usually, Levi's lemma is formulated for the case that the alphabet
$\Sigma$ is finite. But the case that $\Sigma$ is finite already implies
the general case with $\Sigma$ possibly infinite: simply replace the trace
monoid $M(\Sigma,I)$ by $M(\Sigma',I')$, where $\Sigma'$ contains all symbols
occurring in one of the traces $u_i, v_j$ and $I'$ is the restriction of $I$
to $\Sigma'$.

A consequence of Levi's Lemma is that
trace monoids are cancellative, i.e., $usv=utv$ implies $s=t$ for all
traces $s,t,u,v\in M$.

 \subsubsection{Projections to free monoids}
  \label{sec:pojection-to-free-monoids}

  It is a well-known result
\cite{DUBOC19851,duboc1986some,wrathall1988word} that every trace monoid can be embedded into a direct product of free monoids. In this section we recall the corresponding results.

 Consider a trace monoid $M = M(\Sigma, I)$ with the property that there exist
 finitely many sets
$A_i \subseteq \Sigma$ ($i \in [1,k]$ for some $k\in \mathbb{N}$) fulfilling the following property:
  \begin{equation*}
      (a, b) \in D \text{ if and only if } \exists i \in [1,k] : a,b \in A_i.
  \end{equation*}
 Since $D$ is reflexive this implies that for every $a\in\Sigma$ there is an $i$ such that $a\in A_i$.
  %Note that we omit the requirement $a \in A_i \implies a^{-1} \in A_i$ given in %\cite{wrathall1988word}, because the embedding of a trace monoid into a RAAG does %not produce any $a^{-1}$.
  All trace monoids $M(\Sigma,I)$ that will appear in this paper have the above
 property if one takes for the $A_i$ the maximal cliques in the dependence graph $(\Sigma,D)$~\cite{duboc1986some}. If $\Sigma$ is finite, one can take for the $A_i$
 also all sets $\{a, b\}$ with $(a, b) \in D$ together with all singletons $\{a\}$ with $a$ an isolated vertex in $(\Sigma,D)$~\cite{DUBOC19851}.

  Let $\pi_i: \Sigma^* \to A_i^*$ be the projection to the free monoid $A_i^*$ defined by $\pi_i(a) = a$ for $a \in A_i$ and $\pi_i(a) = 1$ otherwise.
We define a projection $\Pi: \Sigma^* \to A_1^*\times\cdots\times A_k^*$
to a direct product of free monoids by $\Pi(w) = (\pi_1(w), \dots, \pi_k(w))$.
It is straightforward to see that, if $u =_M v$, then also  $\Pi(u) = \Pi(v)$. Hence, we can consider $\Pi$ also as a monoid morphism $\Pi : M \to A_1^*\times\cdots\times A_k^*$ (which from now on we denote by the same letter $\Pi$).
  We will make use of the following two lemmata presented in~\cite{duboc1986some}.

    \begin{lemma}
    [\mbox{\cite[Lemma~1]{wrathall1988word}}, \mbox{\cite[Proposition~1.2]{duboc1986some}}]
    \label{lem:projection_eq_gp}
    Let $M = M(\Sigma, I)$.
    For $u, v \in \Sigma^*$ we have $u =_M v$ if and only if $\Pi(u) = \Pi(v)$.
  \end{lemma}
Thus,  $\Pi$ is an \emph{injective} monoid morphism $\Pi : M \to A_1^*\times\cdots\times A_k^*$.
  \begin{lemma}
    [\mbox{\cite[Proposition~1.7]{duboc1986some}}]
    \label{lem:power_gp}
    Let $M = M(\Sigma, I)$, $w \in \Sigma^*$ and $t > 1$.
    Then, there is $u \in \Sigma^*$ with $w =_M u^t$ if and only if there is a tuple
    $\vec{v} \in \prod_{i=1}^k A_i^*$ with $\Pi(w) = \vec{v}^t$.
  \end{lemma}
  In \cite{duboc1986some} these lemmata are only proved for the case that
  $\Sigma$ is finite, but as for \hyperref[lemma-levi]{Levi's Lemma} one obtains the general case
  by restricting $(\Sigma,I)$ to those letters that appear in the traces involved.

Projections onto free monoids were used in~\cite{duboc1986some}  in order
  to show the following lemmata.

\begin{lemma}[\mbox{\cite[Corollary~3.13]{duboc1986some}}]\label{lemm:conjugacy_equivalent}
   Let $M = M(\Sigma, I)$
  and $u,v \in M$. Then there is some $x \in M$ with $xu =_M vx$ if and only if $u$ and $v$ are related by a sequence of transpositions, \ie{} there are $y_1, \dots, y_k$ such that $u = y_1$, $v = y_k$ and $y_{i+1}$ is a transposition of $y_i$.
\end{lemma}
 \cref{lemm:conjugacy_equivalent} gives us a tool for checking conjugacy in  $ M(\Sigma, I)$; indeed, from now on, we will  most of the time use that conjugate elements are related by a sequence of transpositions.

 \begin{lemma}[\mbox{\cite[Proposition~3.5]{duboc1986some}}]\label{lemm:conjugacy_primitive}
   Let $M = M(\Sigma, I)$
  and $u,v,p,q \in M$ such that $u = p^k$ and $v = q^\ell$ with
$p$ and $q$ primitive and $k,\ell \geq 1$. Then $u$ and $v$ are conjugate
if and only if $k = \ell$ and $p$ and $q$ are conjugate.
\end{lemma}
Note that \cref{lemm:conjugacy_primitive} implies that
if $u$ is conjugate to a primitive trace, then $u$ must be primitive as well.

\subsection{Trace monoids defined by finite graphs}\label{sec:big_trace_monoid_prelims}

\newcommand{\cLelA}{\zeta}
\newcommand{\cLelB}{\xi}
\newcommand{\cLelC}{\nu}
\newcommand{\cLelD}{\kappa}

As a first step towards graph products let us consider trace monoids of a special form:  Let $\cL$ be a finite set of size  $\sigma = \lvert\cL\rvert$ and \(I \subseteq \cL \times \cL\) be irreflexive and symmetric (\ie $(\cL,I)$ is a finite undirected simple graph). Moreover, assume that for each $\cLelA \in \cL$ we are given a (possibly infinite) alphabet $\Gamma_\cLelA$ such that $\Gamma_\cLelA \cap \Gamma_\cLelB = \emptyset$ for $\cLelA \neq \cLelB$.
By setting $\Gamma = \bigcup_{\cLelA\in \cL}\Gamma_\cLelA$ and $I_\Gamma = \set{(a,b)}{(\cLelA,\cLelB) \in I, a\in \Gamma_\cLelA, b\in \Gamma_\cLelB}$, we obtain a trace monoid $M = M(\Gamma, I_\Gamma)$. Henceforth, we simply write $I$ for $I_\Gamma$.  For $a \in \Gamma$ we define \(\alphabet(a) = \cLelA \) if \(a \in \Gamma_\cLelA\).
  For $\u = a_1\cdots a_k \in \Gamma^*$ we define $\alphabet(\u) = \{\alphabet(a_1), \dots, \alphabet(a_k)\}$.

The following lemma  characterizes the shape of a prefix, suffix or factor of a power in the above trace monoid $M$.

	\begin{lemma}
		\label{lem:factor_shape}
		Let $p \in M$ be connected and $k \in \mathbb{N}$.
		Then we have:
		\begin{enumerate}%[label=(\roman*)]
			\item If $p^k =_M uw$ for traces $u,w \in M(\Gamma, I)$, then
            there exist $s < \sigma$, $\ell,m \in \mathbb{N}$ and factorizations $p = u_i w_i$ for $i \in [1,s]$ such that
            \begin{itemize}
                \item $k = \ell+s+m$,
                \item $u_i \neq 1 \neq w_i$ for all $i \in [1,s]$ and
                $(w_i,u_j) \in I$ for $i < j$,
                \item $u =_M p^{\ell}u_1\cdots u_s$ and $w =_M w_1\cdots w_s p^m$.
            \end{itemize}
			\item Given a factor $v$ of $p^k$ at least one of the following is true.
			\begin{itemize}
				\item $v = u_1\cdots u_a v_1 \cdots v_b w_1 \cdots w_c$ where $a, b, c \in \mathbb{N}$, $a + b + c \le 2\sigma - 2$, $u_i$ is a proper suffix of $p$ for $i \in [1,a]$, $v_i$ is a proper factor of $p$ for $i \in [1,b]$ and $w_i$ is a proper prefix of $p$ for $i \in [1,c]$.
				\item $v = u_1\cdots u_a p^b w_1 \cdots w_c$ where $a, b, c \in \mathbb{N}$, $a, c < \sigma$, $u_i$ is a proper suffix of $p$ for $i \in [1,a]$ and $w_i$ is a proper prefix of $p$ for $i \in [1,c]$.
			\end{itemize}
		\end{enumerate}
	\end{lemma}
\noindent
	\Cref{fig:prefix_power} illustrates case (i) of \cref{lem:factor_shape}.

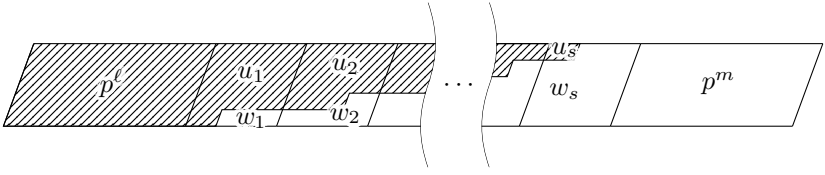
\begin{figure}

		\contourlength{0.15em}
		\centering
		\begin{tikzpicture}[yscale=0.55,xscale=0.4]
			\begin{scope}
				\clip (11,-1) to [bend left=25]  (11, 1) to [bend right=25]  (11, 3) -- (-3.1, 3) -- (-3.1, -1);

				\draw (-3, 0) -- (-2, 2);
				\draw (-3, 0) -- (3, 0);
				\draw (-2, 2) -- (4, 2);

				\draw (3, 0) -- (4, 2);
				\draw (3, 0) -- (6, 0);
				\draw (4, 2) -- (7, 2);

				\draw (6, 0) -- (7, 2);
				\draw (6, 0) -- (9, 0);
				\draw (7, 2) -- (10, 2);

				\draw (9, 0) -- (10, 2);
				\draw (9, 0) -- (12, 0);
				\draw (10, 2) -- (13, 2);

				\draw[pattern=north east lines] (-2, 2) -- (-3,0) -- (4, 0) -- (4.2, 0.4) -- (8.2, 0.4) -- (8.4, 0.8) -- (12, 0.8) -- (12, 2);

				\draw  (11,-1) to [bend left=25]  (11, 1) to [bend right=25]  (11, 3);
			\end{scope}

			\node at (12, 1) {$\,\cdots$};

			\begin{scope}

				\clip (13,-1) to [bend left=25]  (13, 1) to [bend right=25]  (13, 3) -- (24.1, 3) -- (24.1, -1);

				\draw (11, 0) -- (14, 0);
				\draw (12, 2) -- (15, 2);
				\draw (14, 0) -- (15, 2);

				\draw (14, 0) -- (17, 0);
				\draw (15, 2) -- (18, 2);
				\draw (17, 0) -- (18, 2);

				\draw (17, 0) -- (23, 0);
				\draw (18, 2) -- (24, 2);
				\draw (23, 0) -- (24, 2);

				\draw[pattern=north east lines] (12, 1.2) -- (13.6, 1.2) -- (13.8, 1.6) -- (15.8, 1.6) -- (16, 2) -- (12, 2);

				\draw (13,-1) to [bend left=25]  (13, 1) to [bend right=25]  (13, 3);
			\end{scope}
			\node at (0.5, 1) {\contour{white}{$\,p^\ell$}};
			\node at (5.1, 1.2) {\contour{white}{$\,u_1$}};
			\node at (8.2, 1.4) {\contour{white}{$\,u_2$}};
			\node at (15.5, 1.8) {\contour{white}{$u_s$}};
			\node at (5.1, 0.1) {\contour{white}{$\,w_1$}};
			\node at (8.2, 0.2) {\contour{white}{$\,w_2$}};
			\node at (15.5, 0.8) {\contour{white}{$w_s$}};
			\node at (20.5, 1) {\contour{white}{$\,p^m$}};
		\end{tikzpicture}
		\caption{A factorization of $p^k$ as in case 1 of \cref{lem:factor_shape}.}
		\label{fig:prefix_power}
	\end{figure}

 \begin{proof}
 Let us start with the first statement.
We apply \hyperref[lemma-levi]{Levi's Lemma} to the identity $p^k =_M uw$ and obtain the following diagram:
\begin{center}
  \begin{tabular}{c"c|c|c|c|c|c|c|}\hline
  $w$  & $w_{1}$ & $w_{2}$ & $w_{3}$  & $w_{4}$ & $\cdots$  &  $w_{k-1}$  & $w_{k}$ \\ \hline
  $u$  & $u_{1}$ & $u_{2}$ & $u_{3}$  & $u_{4}$ & $\cdots$  &  $u_{k-1}$  & $u_{k}$ \\ \thickhline
          & $p$          & $p$          & $p$            & $p$          & $\cdots$  & $p$ & $p$
  \end{tabular}
  \end{center}
  We have $(w_i, u_{i+1}) \in I$ and hence $\alphabet(w_i) \cap \alphabet(u_{i+1}) = \emptyset$ for all $i \in [1,k-1]$.
  Since $\alphabet(p) = \alphabet(u_j) \cup \alphabet(w_j)$ for all $j \in [1,k]$,
  this implies $\alphabet(u_{i+1}) \subseteq \alphabet(u_i)$. Now assume that
  $i \in [1,k-1]$ is such that $u_i \neq 1 \neq w_i$. We have
  $(u_{i+1}, w_i) \in I$. Since
  we cannot have
  $(u_i, w_i) \in I$ ($p$ is connected), we cannot have
  $\alphabet(u_i) \subseteq \alphabet(u_{i+1})$. Therefore, $\alphabet(u_{i+1}) \subsetneq \alphabet(u_{i})$ whenever
  $u_i \neq 1 \neq w_i$. It follows that there are $\ell,m \geq 0$ and $s < \sigma$
such that $k = \ell+s+m$ and
  \begin{itemize}
      \item $u_i = p$, $w_i = 1$ for $i \in [1,x]$,
      \item $u_i \neq 1 \neq w_i$, $u_i w_i = p$ for $i \in [x+1,x+s]$, and
      \item $u_i = 1$, $w_i = p$ for $i \in [x+s+1,k]$.
  \end{itemize}
  By renaming $u_{x+i}$ and $w_{x+i}$ into $u_i$ and $w_i$, respectively, for
  $i \in [1,s]$ we obtain factorizations
  $u =_M p^\ell u_1 \cdots u_s$ and  $w =_M w_1 \cdots w_s p^m$ for some $s < \sigma$
  and traces $u_i, w_i \in M\setminus \{1\}$ with $p =_M u_i w_i$.
  This yields statement 1.

  To derive statement 2, consider the factorization $p^k =_M u (vw)$. Applying the final conclusion of the previous paragraph, we obtain factorizations
  $u =_M p^x u_1 \cdots u_s$ and $vw =_M x_1 \cdots x_s p^\ell$ where $s < \sigma$,
  the $u_i$ are proper prefixes of $p$, the $x_i$ are proper suffixes of $p$ and
  $k = x + s + \ell$.

  We then consider two cases: if $\ell = 0$, then $vw =_M x_1 \cdots x_s$.
  Applying \hyperref[lemma-levi]{Levi's Lemma} to this factorization yields the following diagram:
  \begin{center}
  \begin{tabular}{c"c|c|c|c|c|c|}\hline
  $w$  & $w_{1}$ & $w_{2}$ & $w_{3}$   & $\cdots$  &  $w_{s-1}$  & $w_{s}$ \\ \hline
  $v$  & $v_{1}$ & $v_{2}$ & $v_{3}$  & $\cdots$  &  $v_{s-1}$  & $v_{s}$ \\ \thickhline
       & $x_1$          & $x_2$    & $x_3$    & $\cdots$  & $x_{s-1}$ & $x_s$
  \end{tabular}
  \end{center}
  Hence, $v =_M v_1 v_2 \cdots v_s$, where every $v_i$ is a prefix of the proper suffix $x_i$
  of $p$. Therefore, $v_i$ is a proper factor of $p$.

  Now assume that $\ell > 0$. Applying \hyperref[lemma-levi]{Levi's Lemma} to $vw =_M x_1 \cdots x_s p^\ell$ yields
  a diagram of the following form:
  \begin{center}
  \begin{tabular}{c"c|c|c|c|c|c|c|c|c|c}\hline
  $w$  & $w_{1}$ & $w_{2}$    & $\cdots$  &  $w_{s-1}$  & $w_{s}$ & $w_{s+1}$ & $w_{s+2}$ & $\cdots$ & $w_{s+\ell-1}$ & $w_{s+\ell}$\\ \hline
  $v$  & $v_{1}$ & $v_{2}$    & $\cdots$  &  $v_{s-1}$  & $v_{s}$ & $v_{s+1}$ & $v_{s+2}$  & $\cdots$ & $v_{s+\ell-1}$ & $v_{s+\ell}$\\ \thickhline
     & $x_1$          & $x_2$       & $\cdots$  & $x_{s-1}$ & $x_s$
            & $p$ & $p$ & $\cdots$ & $p$ & $p$
  \end{tabular}
  \end{center}
  To the factorizations $p =_M v_{s+i} w_{s+i}$ ($i \in [1,\ell]$) we apply the arguments
  used for the proof of statements 1 and 2. There are $y,z \geq 0$ and $t < \sigma$
such that $\ell = y+t+z$ and
  \begin{itemize}
      \item $v_{s+i} = p$, $w_{s+i} = 1$ for $i \in [1,y]$,
      \item $v_{s+i} \neq 1 \neq w_{s+i}$, $v_{s+i} w_{s+i} =_M p$ for
           $i \in [y+1,y+t]$, and
      \item $v_{s+i} = 1$, $w_{s+i} = p$ for $i \in [y+t+1,\ell]$.
  \end{itemize}
  We obtain $v =_M v_1 \cdots v_s p^y v_{s+y+1} \cdots v_{s+y+t}$
  where every $v_{s+y+i}$ ($i \in [1,t]$) is a proper prefix of $p$.
  If $y > 0$ then $(v_{s+1}, w_1 \cdots w_s) \in I$ implies $w_1 \cdots w_s = 1$.
  Hence, every $v_i$ ($i \in [1,s]$) is proper suffix of $p$
  (by a symmetric argument, we could write $v$ also as a concatenation of
  $s < \sigma$ many proper suffixes of $p$ followed by $t < \sigma$ many
  proper factors of $p$).

  Finally, assume that $y=0$. We get
  $v =_M v_1 \cdots v_s v_{s+y+1} \cdots v_{s+y+t}$ with every
  $v_i$ ($i \in [1,s]$) a proper factor of $p$ and every
  $v_{s+y+i}$ ($i \in [1,t]$) a proper prefix of $p$.
 \end{proof}

For a trace $u\in M = M(\Gamma,I)$ and  $\cLelA \in \cL$, we write $\abs{u}_\cLelA = \abs{u}_{\Gamma_\cLelA} = \sum_{a\in \Gamma_\cLelA} \abs{u}_a$. Note that, while the sum might be infinite, only finitely many summands are non-zero.

\begin{lemma}\label{lem:lettercount}
    Let $r,s,t,u\in M$ with $rs =_M tu$ and, for all $\cLelA \in \cL$, $\abs{s}_\cLelA \geq  \abs{u}_\cLelA$
    or, equivalently, $\abs{r}_\cLelA \leq  \abs{t}_\cLelA$. Then, as elements of $M$, $u$ is a suffix of $s$ and $r$ is a prefix of $t$.
    In particular, if for all $\cLelA \in \cL$ we have $\abs{s}_\cLelA =  \abs{u}_\cLelA$, then $s =_M u$ and $r =_M t$.
\end{lemma}

\begin{proof}
By \hyperref[lemma-levi]{Levi's Lemma}, there are $p,q,x,y \in M$ with $(x,y) \in I$ and $r=px$, $t = py$, $s =yq$, and $u = xq$. Because of the condition $\abs{s}_\cLelA \geq  \abs{u}_\cLelA$ for all $\cLelA \in \cL$, $x$ must be the empty trace.

    The second part of the lemma follows by using the first part for the two inequalities $\abs{s}_\cLelA \geq  \abs{u}_\cLelA$ and $\abs{u}_\cLelA \geq  \abs{s}_\cLelA$.
\end{proof}

\begin{lemma}\label{lem:p_factor}
    Let $p^\sigma u=_Mvp^\sigma$ for some primitive and connected trace $p$ and let $u \in M$ be a prefix of $p^k$ for some $k\in \N$. Then we have $u=v=p^\ell$ for some $\ell \in [0,k]$.
\end{lemma}

\begin{proof}
If $u$ is the empty prefix, we are done. Hence, from now on, we can assume that $u$ is non-empty.
%    We start with the observation that also $v$ is a prefix of $p^{\sigma + k}$.
First consider the case that $p$ is  a prefix of $u$. Then, $p^{\sigma+1}$ is a prefix of $vp^\sigma$. Hence, there is a trace $q$ with $vp^\sigma =_M pq$, where
$p^\sigma$ is a factor of $q$. Then \cref{lem:lettercount} implies that
$p$ is a prefix of $v$.

If $v =_M pv'$ and $u =_M  p u'$, we obtain $p^{\sigma+1} u' =_M p v' p^{\sigma}$.
Cancelling $p$ yields $p^\sigma u' =_M v' p^\sigma$. Since $u'$ is a prefix
of $p^{k-1}$ we can replace $u$ and $v$ by $u'$ and $v'$, respectively.
    Therefore, we can assume that $p$ is not a prefix of $u$. Since $u$ is a prefix of some $p^k$, \cref{lem:factor_shape} implies that $u$ is already a prefix of $p^\sigma$.

    Let us next show that $u=v$. To do so, we write $p^\sigma =_M uw$. Then we have
    \[uwu =_M p^\sigma u =_M vp^\sigma =_M vuw.\]
    Since $\abs{wu}_a =  \abs{uw}_a$ for all $a \in \Gamma$, \cref{lem:lettercount}
    implies $u=v$.

    Now, we have $p^\sigma u=_Mup^\sigma$.
Since $p$ is connected, \cite[Proposition 3.1]{DUBOC19851} implies that there are $i,j \in \N$ with $p^{\sigma \cdot i} =_M u^j$. Then, by \cite[Theorem 1.5]{DUBOC19851} it follows that there are $t \in M$ and $\ell,m \in \N$ with $p=_Mt^m$ and $u=_Mt^\ell$. As $p$ is primitive, we have $m = 1$ and $t=p$
and hence $u =_M = p^\ell$.
Since $u$ is a prefix of $p^k$, we have $\ell \leq k$.
\end{proof}

  To prove the next lemma, we want to apply \cref{lem:projection_eq_gp}.
 To do so, we use the following projections suitable for our use case.
  Let
\begin{equation} \label{clique-covering-A}
 \mathcal{A} = \left\{ \Gamma_\cLelA \cup \Gamma_\cLelB \mid (\cLelA, \cLelB) \in D,  \cLelA \neq \cLelB\right\} \cup \{ \Gamma_\cLelA \mid \cLelA \text{ is isolated} \},
\end{equation}
   where $\cLelA$ is isolated if
  there is no $\cLelB \neq \cLelA$ with  $(\cLelA, \cLelB) \in D$ (and $D= \cL\times\cL \setminus I$). Notice that even though $\Gamma$ might be infinite, $\cA$ is finite in any case (because $\cL$ is finite).
  Let us write  $\mathcal{A}=\{A_1, \ldots, A_k\}$ and $\pi_i$ for the projection $M(\Gamma,I) \to A_i^*$.

\begin{lemma}\label{lem:lettercount_conjugacy}
    Let $u,v,p,q \in M$ and $k\in \N$ with $uqv =_M p^k$ and $\abs{p}_\cLelA =  \abs{q}_\cLelA$  for all $\cLelA \in \cL$. Then $p$ and $q$ are conjugate in $M$.
\end{lemma}
\begin{proof}
First, we are going to show that the transposition $qvu $ of $uqv$ is equal to $q^k$ in $M$. Consider projections $\pi_i$
onto cliques.
    By the assumption $\abs{p}_\cLelA =  \abs{q}_\cLelA$  for all $\cLelA \in \cL$, it follows that $\abs{\pi_i(p)} = \abs{\pi_i(q)}$. As $\pi_i(p^k)$ has a period $\abs{\pi_i(p)}$, so has its cyclic permutation   $\pi_i(qvu)$. As its first $\abs{\pi_i(p)}$ letters are exactly $\pi_i(q)$, it follows that $\pi_i(qvu) = \pi_i(q^k)$.  Since this holds for all $i$, it follows by \cref{lem:projection_eq_gp} that $qvu =_M q^k$.

Now, observe that $q^k =_M qvu$ and $p^k =_M uqv$ are conjugate in $M$.
    Hence, it remains to apply \cref{lemm:conjugacy_primitive} to conclude that $p$ and $q$ are conjugate: we write $p=\tilde p^i$ and $q=\tilde q^j$ for primitive traces $\tilde p, \tilde q$. Then \cref{lemm:conjugacy_primitive} tells us that $i=j$ and $\tilde p$ and $\tilde q$ are conjugate. Hence, also $p$ and $q$ are conjugate.
\end{proof}

\subsection{Groups}\label{sec:group_prelims}

 If $G$ is a group, then $u, v \in G$ are conjugate if and only if there is a $g \in G$ such that $u =_G g^{-1}vg$ (note that this agrees with the above definition for monoids).

 \subsubsection{Free groups} \label{sec-free-groups}

\newcommand{\freeRS}{S_{\mathrm{free}}}
\newcommand{\traceRS}{T}
\newcommand{\proj}{\eta}
\newcommand{\finv}[1]{\overline{#1}}

Let $X$ be a set and $\finv{X} = \set{\finv{a}}{a \in X}$ be a disjoint copy of $X$.
We extend the mapping $a \mapsto \finv{a}$
to an involution without fixed points on $\Sig = X \cup \finv{X}$ by $\finv{\finv{a}} = a$
and finally to an involution on $\Sig^*$ by $\finv{a_1 a_2 \cdots a_n} =
\finv{a_n} \cdots \finv{a_2} \;\finv{a_1}$. The only fixed point of the latter involution is the empty word $1$.
The string rewriting system
\[\freeRS= \set{a \finv{a} \to 1}{a \in \Sig}\]
is strongly confluent and terminating meaning that for every word $w \in \Sig^*$ there exists
a unique word $\hat{w} \in \IRR(\freeRS)$ with $w \rewrite{*}{\freeRS} \hat{w}$.
Words from $\IRR(\freeRS)$ are called \emph{freely reduced}.
The system $\freeRS$ defines the free group $F(X) = \Sig^*/ \freeRS$ with basis $X$.
Let $\proj: \Sig^* \to F(X)$ denote the canonical monoid homomorphism. Then we have $\proj(w)^{-1} = \proj(\finv{w})$ for all words $w \in \Sig^*$.
If $|X|=2$, then we write $F_2$ for $F(X)$. It is known that for every countable set $X$,
$F_2$ contains an isomorphic copy of $F(X)$.

\subsubsection{Finitely generated groups and the word problem} \label{f.g-groups}

A group $G$ is called {\em finitely generated} (f.g.) if there exists a finite set $X$ and
a surjective group homomorphism $h : F(X) \to G$. In this situation, the set
$\Sig = X \cup \finv{X}$ is called a finite (symmetric) generating set for $G$. Usually, we write $X^{-1}$ instead of $\finv{X}$ and $a^{-1}$ instead of $\finv{a}$ for $a \in \Sig$.
Thus, for an integer $z < 0$ and $w \in \Sig^*$ we
write $w^z$ for $(\finv{w})^{-z}$.

In many cases we can think of $\Sigma$ as a subset of $G$, but, in general,
we can also have more than one letter for the same group element.
The group identity of $G$ is denoted with $1$ as well
(this fits to our notation $1$ for the empty word which is the identity of $F(X)$).

For words $u,v \in \Sigma^*$  we usually say that $u = v$ in $G$ or $u =_G v$ in case $h(\proj(u)) = h(\proj(v))$ and we do not write $\proj$ nor $h$ from now on.
The {\em word problem} for the finitely generated group $G$, $\WP(G)$ for short, is defined as follows:
\ynproblem{ a word $w \in \Sigma^*$.}{Does $w=_G 1$ hold?}

\subsubsection{The power word problem}

A \emph{power word} (over $\Sigma$) is a tuple $(u_1,x_1,u_2,x_2,\ldots,u_n,x_n)$ where
$u_1, \dots, u_n \in \Sigma^*$ are words over the group generators
and $x_1, \dots, x_n\in \Z$ are integers that are given in binary notation. Such a power word represents the
word $u_1^{x_1} u_2^{x_2}\cdots u_n^{x_n}$. Quite often, we will identify the power word $(u_1,x_1,u_2,x_2,\ldots,u_n,x_n)$
with the word $u_1^{x_1} u_2^{x_2}\cdots u_n^{x_n}$. Moreover, if $x_i=1$, then we usually omit the exponent $1$ in a power word.
The \emph{power word problem} for the finitely generated group $G$, $\PowWP(G)$ for short, is defined as follows:
\ynproblem{ a power word $(u_1,x_1,u_2,x_2,\ldots,u_n,x_n)$.}{Does $u_1^{x_1} u_2^{x_2}\cdots u_n^{x_n}=_G 1$ hold?}

\noindent Due to the binary encoded exponents, a power word can be seen as a succinct description
of an ordinary word. Hence, a priori, the power word problem for a group $G$ could be computationally
more difficult than the word problem. An example, where this happens
(under standard assumptions from complexity theory) is the wreath product
$S_5 \wr \mathbb{Z}$ (where $S_5$ is the symmetric group on 5 elements).
The word problem for this group can be easily solved in logspace, whereas
the power word problem for $S_5 \wr \mathbb{Z}$ is $\coNP$-complete \cite{LohreyW19}.

  Let $\mathcal{C}$ be a countable class of groups, where every group has a finite description. We also assume that
  the description of $G \in \mathcal{C}$ contains a generating set for $G$.
  We write $\UPowWP(\mathcal{C})$ for the {\em uniform power word problem}:
  \ynproblem{ a group $G \in \mathcal{C}$ and a power word $(u_1,x_1,u_2,x_2,\ldots,u_n,x_n)$ over the generating set of $G$.}{Does $u_1^{x_1} u_2^{x_2}\cdots u_n^{x_n}=_G 1$ hold?}

\subsubsection{Right-angled Artin groups} \label{sec-RAAG}

Right-angled Artin groups
are defined similarly to partially commutative monoids.
Again we have a symmetric and irreflexive commutation
relation $I \subseteq X \times X$.
%We require that $(a, b) \in I$ if and only if $(a^{-1}, b) \in I$.
Then $G(X,I) = F(X)/\{ab=ba\mid (a,b)\in I\}$ is the corresponding \emph{right-angled Artin group} (RAAG), also known as a \emph{graph group} or \emph{free partially commutative group}.
The name graph group is due to the commutation relation being commonly visualized as an undirected graph.
Note that we have $M(X,I) \sse G(X,I)$.

We can view $G(X,I)$ also as follows: let $\Sigma = X \cup \finv{X}$ where $\finv{X}$ is a disjoint copy of $X$ and $\finv{\finv{a}} = a$ for $a \in \Sigma$ (like for free groups). Extend $I$ to $\Sigma \times \Sigma$ by requiring that $(a, b) \in I$ if and only if $(\finv{a}, b) \in I$ for $a,b \in \Sigma$.
Then $G(X,I)$ is the quotient of $M(\Sigma,I)$ defined by the relations $a\finv{a}=1$ for $a \in \Sigma$. A trace $w \in M(\Sigma,I)$ is called \emph{reduced} if it does not contain a factor $a\finv{a}$ for $a \in \Sigma$. For every trace $u \in M(\Sigma,I)$ there is a unique
reduced trace $v$ (the reduced normal form of $u$) with $u = v$ in
$G(X,I)$. Like for free groups, it can be computed using the confluent and terminating trace rewriting system $\{ a \finv{a} \to 1 \mid a \in \Sigma\}$.

\subsubsection{Graph products}\label{sec:graph_prod_prelims}
  Let $\left(G_\cLelA\right)_{\cLelA \in \cL}$ be a family of so-called \emph{base groups} and \(I \subseteq \cL \times \cL\) be an irreflexive and symmetric relation (the  \emph{independence relation}).
   As before, we assume that $\cL$ is always finite and
  we write $\sigma = \abs{\cL}$.
  The graph product $ \GP(\cL, I, \left(G_\cLelA\right)_{\cLelA \in \cL})$ is defined as the free product of the $G_\cLelA$ modulo the relations expressing that elements from $G_\cLelA$ and $G_\cLelB$ commute whenever $(\cLelA,\cLelB) \in I$. Below, we define this group by a group presentation.

Let $\Gamma_\cLelA = G_\cLelA \setminus \{ 1 \}$ be the set of non-trivial elements of the group $G_\cLelA$ for $\cLelA \in \cL$. We assume w.l.o.g.~that the sets $\Gamma_\cLelA$ are pairwise disjoint.
We then define $\Gamma$ and $I_{\Gamma}$ as in \cref{sec:big_trace_monoid_prelims}:
$\Gamma = \bigcup_{\cLelA \in \cL} \Gamma_\cLelA$  (note that typically, $\Gamma$ will be infinite)
and $I_\Gamma = \{ (a, b) \in \Gamma\times\Gamma \mid (\alphabet(a), \alphabet(b)) \in I\}$.
As in \cref{sec:big_trace_monoid_prelims} we write $I$ instead of $I_\Gamma$.
  For $a,b \in G_\cLelA$ we write $[ab]$ for the element of $G_\cLelA$ obtained by multiplying $ab$ in $G_{\cLelA}$  (whereas $ab$ denotes a two-letter word in $\Gamma^*$). Here, we identify $1\in G_\cLelA$ with the empty word $1$.
 The relation $I$ is extended to
  $\Gamma^*$ by $I = \{ (u,v) \in \Gamma^* \times \Gamma^* \mid \alphabet(u) \times \alphabet(v) \subseteq I \}$ (where $\alphabet(u) \sse \cL$  is defined as in \cref{sec:big_trace_monoid_prelims}).
  With these definitions we have
    \begin{align*}
    \GP(\cL, I, \left(G_\cLelA\right)_{\cLelA \in \cL}) = \genr{ \Gamma }{ ab = [ab] \text{ for } \cLelA \in \cL, a,b\in \Gamma_\cLelA, ab = ba \text{ for } (a,b)\in \IGamma }\!.
  \end{align*}

  \begin{example}
    If $I = \emptyset$, then $ \GP(\cL, I, \left(G_\cLelA\right)_{\cLelA \in \cL})$ is simply the free product $*_{\cLelA\in \cL} \, G_\cLelA$.
  \end{example}

  \begin{example}
    If all the base groups are the infinite cyclic group (\ie for each $\cLelA \in \cL$ we have $G_\cLelA = \mathbb{Z}$), then the graph product $ \GP(\cL, I, \left(G_\cLelA\right)_{\cLelA \in \cL})$ is the RAAG $G(\cL, I)$.
  \end{example}

Let $G = \GP(\cL, I, \left(G_\cLelA\right)_{\cLelA \in \cL})$ be a graph product and $M = M(\Gamma, \IGamma)$ the corresponding trace monoid
 (see Section~\ref{sec-traces}).   Notice that $M $ satisfies the setting of \cref{sec:big_trace_monoid_prelims}~-- so these results and definitions apply to the case of graph products.
  We can represent elements of $G$ by elements of $M$. More precisely, there is a  canonical surjective homomorphism $h : M \to G$.
  A \emph{reduced} representative of a group element $g \in G$ is
  a trace $w$ of minimal length such that $h(w) = g$. We also say that $w$ is
  {\em reduced}.
  Equivalently, $w \in M$ is reduced if there is no two-letter factor $ab$ of $w$ such that $\alphabet(a) = \alphabet(b)$.
   A trace $w \in M$ is called \emph{cyclically reduced} if all transpositions of $w$ are reduced.
   Equivalently, $w$ is cyclically reduced if it is reduced and it cannot be written in the form $axb$ with $a,b \in \Gamma_\cLelA$ for some $x \in M$.
   Note that this definition agrees with \cite{kausch2017parallel}, whereas in \cite{FigeliusGLZ20} a slightly different definition is used.
   We call a trace  $w \in M$ \emph{composite} if $\abs{\alphabet(w)} \geq 2$. Notice that a trace $w$, where every connected component is composite,
   %composite and connected trace $w$
   is cyclically reduced if and only if $ww$ is reduced (then, every $w^k$ with $k \geq 2$ is reduced).
   A word $w \in \Gamma^*$ is called reduced/cyclically reduced/composite if the
   trace represented by $w$ is reduced/cyclically reduced/composite.

   Note that a word $w \in \Gamma^*$ is cyclically reduced if and only if every cyclic permutation of the word $w$ is reduced as a trace (be aware of the subtle difference between a cyclic permutation of a word $w$ and a transposition of the trace represented by $w$):
   If the trace represented by $w$ is cyclically reduced, then clearly every cyclic permutation of $w$ must be reduced.
   On the other hand, assume that $w =_M a w' b$ with $a,b \in \Gamma_\cLelA$. Then we can write the word $w$ as $w = x a y b z$ such that $(a, xz) \in I$. Then $y b zx a$ is a cyclic permutation of $w$ that is not reduced.

On the free monoid $\Gamma^*$ we can
define an involution $(\cdot)^{-1}$ by $(a_1 a_2 \cdots a_n)^{-1} = a_n^{-1}
\cdots a_2^{-1} a_1^{-1}$, where $a_i^{-1}$ is the inverse of $a_i$ in the group
$G_{\alphabet(a_i)}$. Note that $u =_M v$ implies $u^{-1} =_M v^{-1}$. Therefore,
we obtain a well-defined involution $(\cdot)^{-1}$ on $M$. Moreover, $u^{-1}$ indeed represents the inverse of $u$ in the group $G$.

% as follows: if the trace $w$ is represented
%by the word $u$ then $w^{-1}$ is the trace represented by the word $u^{-1}$.

The counterpart of the rewriting system $\freeRS$ for graph products
is the trace rewriting system
\begin{equation} \label{rewrite-system-T}
\traceRS = \set{ a b \to [ab] }{ a,b \in \Gamma,  \alphabet(a) = \alphabet(b)}.
\end{equation}
Note that $G = M / \traceRS$ and that
$\IRR(\traceRS)$ is the set of reduced traces. Moreover, $T$ is terminating and confluent; the latter is shown in \cite[Lemma 6.1]{KuskeL06}. The following lemma can be found in
\cite[Lemma~24]{haubold2012compressed}.

  \begin{lemma}
    \label{lem:conn_gp_monoid}
    Let  $u, v \in \Gamma^*$. If $u =_M v$, then also $ u =_G v$.
    Moreover, if $u$ and $v$ are reduced, then $u =_M v $ if and only if  $u =_G v$.
  \end{lemma}

The following commutative diagram summarizes the mappings between
the sets introduced in this section ($\hookdoubleheadrightarrow$
indicates a bijection):
\begin{align}\label{eq:irr_embedding}
\begin{split}
  & \eqmathbox[eqn1]{\Gamma^*} \eqmathbox[eqn2]{\quad\twoheadrightarrow\quad} \eqmathbox[eqn3]{M(\Gamma, I)} \eqmathbox[eqn4]{\quad\hookrightarrow\quad}\eqmathbox[eqn5]{G(\Gamma, I)}\\
  &   \eqmathbox[eqn1]{}\eqmathbox[eqn2]{}\eqmathbox[eqn3]{\rotatebox{90}{$\sse$}}\eqmathbox[eqn4]{}\eqmathbox[eqn5]{\rotatebox{90}{$\twoheadleftarrow$}} \\
  &         \eqmathbox[eqn1]{}\eqmathbox[eqn2]{}\eqmathbox[eqn3]{\IRR(\traceRS)}\eqmathbox[eqn4]{\hookdoubleheadrightarrow}\eqmathbox[eqn5]{\GP(\cL, I, (G_\cLelA)_{\cLelA \in \cL})}
  \end{split}
\end{align}
The embedding $M(\Gamma,I)\hookrightarrow G(\Gamma,I)$ is induced by the embedding  $M(\Gamma,I)\hookrightarrow M(\Gamma \cup \finv{\Gamma},I)$ composed with the projection $M(\Gamma \cup \finv{\Gamma},I) \twoheadrightarrow G(\Gamma,I)$ from \cref{sec-RAAG}.\footnote{
In the trace monoid $M(\Gamma \cup \finv{\Gamma},I)$ for every symbol
$a \in \Gamma$ there is a formal inverse $\finv{a}$ such that $(\finv{a},b) \in I$ if and only
if $(a,b) \in I$. This formal inverse $\finv{a}$ is different from the
inverse of $a$ in base group $G_{\alphabet(a)}$, but the surjection
$G(\Gamma,I) \twoheadrightarrow \GP(\cL, I, (G_\cLelA)_{\cLelA \in \cL})$ maps the formal inverse $\finv{a}$ to the inverse of $a$
in base group $G_{\alphabet(a)}$.
A trace $u \in  M(\Gamma \cup \finv{\Gamma},I)$
is reduced with respect to the RAAG $G(\Gamma,I)$ if it does not contain
a factor $a\finv{a}$ or $\finv{a}a$ with $a \in \Gamma$. In particular, every trace
from $M(\Gamma,I)$ is reduced with respect to $G(\Gamma,I)$, even if it
is non-reduced in our sense (i.e., with respect to the graph product
$\GP(\cL, I, (G_\cLelA)_{\cLelA \in \cL})$).}

An $I$-clique is a trace $a_1 a_2 \cdots a_k \in M$
such that $a_i \in \Gamma$ and $(a_i,a_j) \in I$ for all $i \neq j$.
Note that $|v| \leq \sigma$ for every $I$-clique $v$.
The following lemma is a generalization of a statement from
\cite{DiekertL08} (equation (21) in the proof of Lemma 22), where
only the case $q=1$ is considered.

\begin{lemma}
    \label{lem:reducing-pqr}
    Let  $p,q,r,s \in M$ such that $pq, qr, s \in \IRR(\traceRS)$ and $p\,q\,r \rewrite{*}{\traceRS} s$.
    Then there exist factorizations
    \[ p =_M p' t \, u, \qquad r =_M u^{-1} v\, r', \qquad s =_M p' q\, w\, r'
    \]
    with the following properties:
    \begin{itemize}
        \item $t,v,w$ are $I$-cliques with $tv \rewrite{*}{\traceRS} w$,
        \item $\alphabet(t) = \alphabet(v) = \alphabet(w)$, and
        \item $(q,tu) \in I$ (hence also $(q,v), (q,w) \in I$).
    \end{itemize}
  \end{lemma}

\begin{proof}
    We prove the lemma by induction over the length of $\traceRS$-derivations (recall that $T$ is terminating).
    The case that $p\,q\,r \in \IRR(\traceRS)$ is clear (take $t = u = v = w = 1$).
    Now assume that $p\,q\,r$ is not reduced. Since $pq, qr \in \IRR(\traceRS)$,
    the trace $p\,q\,r$ must contain a factor $ab$ with $\alphabet(a) = \alphabet(b)$, where $a$ is a maximal letter of $p$, $b$ is a minimal letter of $r$ and $(a,q), (b,q) \in I$. Let us write $p =_M \tilde{p} a$ and
    $r =_M b \tilde{r}$.

    We distinguish two cases. If $[ab] = 1$, i.e., $b = a^{-1}$, then
    \[ p\,q\,r \ =_M\  \tilde{p}\, a\, q\, a^{-1} \tilde{r} \ \rewrite{}{\traceRS} \
       \tilde{p} \,q\, \tilde{r} \ \rewrite{*}{\traceRS}\ s .
    \]
    Since $(a,q)\in I$, we must have
    $\tilde{p} q,  q\tilde{r} \in \IRR(\traceRS)$.
    Hence, by induction we obtain factorizations
    \[ \tilde{p} =_M p' t\, x, \qquad \tilde{r} =_M x^{-1} v\, r', \qquad s =_M p' q\, w\, r'
    \]
    with the following properties:
    \begin{itemize}
        \item $t,v,w$ are $I$-cliques with $tv \rewrite{*}{\traceRS} w$,
        \item $\alphabet(t) = \alphabet(v) = \alphabet(w)$, and
        \item $(q,tx) \in I$.
    \end{itemize}
    If we set $u = xa$, we obtain exactly the situation from the lemma.

    Now assume that $[ab] = c \neq 1$.
    We obtain
    \[ p\,q\,r \ =_M \ \tilde{p}\, a\, q\, b\, \tilde{r} \ \rewrite{}{\traceRS} \
       \tilde{p}\, c\, q\, \tilde{r}\ \rewrite{*}{\traceRS} \ s .
    \]
    Note that $(c,q) \in I$. Since $\tilde{p}\, a\,q, \, b\,q\,\tilde{r} \in \IRR(\traceRS)$, we also have
    $\tilde{p}\,c\,q,  c\,q\,\tilde{r} \in \IRR(\traceRS)$.
    Hence, by induction we obtain factorizations
    \[ \tilde{p} =_M p' t' u, \qquad \tilde{r} =_M u^{-1} v' r', \qquad s =_M p' c\, q\, w' r'
    \]
    with the following properties:
    \begin{itemize}
        \item $t',v',w'$ are $I$-cliques with $t'v' \rewrite{*}{\traceRS} w'$,
        \item $\alphabet(t') = \alphabet(v') = \alphabet(w')$, and
        \item $(cq,t'u) \in I$.
    \end{itemize}
    We define $t = t'a$, $v = v'b$, $w = w'c$. These are $I$-cliques
    (since $(c,t') \in I$) that satisfy the conditions from the lemma.
    Moreover, $(c,u) \in I$ implies $(a,u) \in I$ and hence
    $p =_M \tilde{p}\,a =_M p' t' u \, a =_M
    p' t' a \, u =_M  p' t \, u$. Similarly, we get
    $r =_M u^{-1} v\, r'$ and $s =_M p' q\, w\, r'$ (using $(c,q) \in I$).
\end{proof}

  Since $\Gamma$ might be an infinite alphabet, for inputs of algorithms, we need to encode elements of $\Gamma$ over a finite alphabet. For $\cLelA \in  \cL$ let $\Sigma_\cLelA$ be a finite generating set for $G_\cLelA$
  such that $\Sigma_\cLelA \cap \Sigma_\cLelB = \emptyset$ for $\cLelA \neq \cLelB$.
Then $\Sigma = \bigcup_{\cLelA \in \cL} \Sigma_\cLelA$ is a generating set for $G$. Every element of $\Gamma_\cLelA$ can be represented as a word from $\Sigma_\cLelA^*$.  However, in general, representatives are not unique. Deciding whether two words $w,v \in \Sigma_\cLelA^*$ represent the same element of $\Gamma_\cLelA$ is the word problem for $G_\cLelA$. We give more details how to represent power words in \cref{sec:input-encoding-spwp}.

Let $\mathcal{C}$ be a countable class of finitely generated groups with finite descriptions. One might for instance take a subclass of finitely (or recursively) presented groups.
Then a graph product $\GP(\cL, I, (G_\cLelA)_{\cLelA \in \cL})$ with $G_\cLelA \in \mathcal{C}$ for all $\cLelA$
has a finite description as well: such a group is given by  the finite graph $(\cL, I)$ and a list of the finite descriptions of the  groups $G_\cLelA \in \mathcal{C}$
for $\cLelA \in \cL$. We denote with $\GP(\mathcal{C})$ the class of all such graph products.

\subsection{Complexity} \label{sec-complexity}

We assume that the reader is familiar with the complexity classes {\sf P} and \NP; see e.g.\ \cite{AroBar09} for details. Let $\cC$ be any complexity class and $K\sse \Delta^*$, $L \sse \Sigma^*$ languages. Then $L$ is $\cC$-many-one reducible to $K$ ($L\leq_{\mathrm{m}}^\cC  K$) if there exists a $\cC$-computable function $f: \Sigma^* \to \Delta^*$ with $x\in L $ if and only if $f(x) \in K$.

\subsubsection{Circuit complexity}
We use circuit complexity for classes below deterministic logspace (\L for short).
Instead of defining these classes directly, we introduce the slightly more general notion of \Ac0-Turing reducibility.
A language $L \subseteq \{0,1\}^*$  is \Ac0-Turing-reducible to $K \sse \{0,1\}^*$ if there is a family of constant-depth, polynomial-size Boolean circuits with oracle gates for $K$ deciding $L$. More precisely, we can define the class of language $\Ac0(K)$ which are \Ac0-Turing-reducible to $K \sse \{0,1\}^*$: a language $L \subseteq \{0,1\}^*$ belongs to $\Ac0(K)$ if there exists a family $(C_n)_{n \geq 0}$ of Boolean circuits with the following
properties:
\begin{itemize}
	\item $C_n$ has $n$ distinguished input gates $x_1, \ldots, x_n$ and a distinguished output gate $o$.
	\item $C_n$ accepts exactly the words from $L \cap \{0,1\}^n$, i.e., if the input gate $x_i$ receives the input $a_i \in \{0,1\}$ for all $i$, then
	the output gate $o$ evaluates to $1$ if and only if $a_1 a_2 \cdots a_n \in L$.
	\item Every circuit $C_n$ is built up from input gates, \emph{not}-gates, \emph{and}-gates, \emph{or}-gates, and oracle gates for $K$ (which output $1$ if and only if their input is in $K$). The incoming wires for an oracle gate for $K$ have to be ordered since the language $K$ is not necessarily closed under permutations of symbols.
	\item All gates may have unbounded fan-in, \ie there is no bound on the number of incoming wires for a gate.
	\item There is a polynomial $p(n)$ such that $C_n$ has at most $p(n)$ many gates and wires.
	\item There is a constant $d$ such that every $C_n$ has depth at most $d$ (the depth is the length of a longest path
	from an input gate $x_i$ to the output gate $o$).
\end{itemize}
This is in fact the definition of non-uniform $\Ac0(K)$. Here ``non-uniform'' means that the mapping $n \mapsto C_n$ is
not restricted in any way. In particular, it can be non-computable. For algorithmic purposes one usually adds some uniformity
requirement to the above definition. The most ``uniform'' version of $\Ac0(K)$ is \DLOGTIME-uniform $\Ac0(K)$. For this,
one encodes the gates of each circuit $C_n$ by bit strings of length $\mathcal{O}(\log n)$. Then the circuit family $(C_n)_{n \geq 0}$
is called \emph{\DLOGTIME-uniform}  if (i) there exists a deterministic Turing machine that computes for a given gate $u \in \{0,1\}^*$
of $C_n$ ($|u| \in \mathcal{O}(\log n)$) in time $\mathcal{O}(\log n)$ the type of gate $u$, where the types are $x_1, \ldots, x_n$, \emph{not}, \emph{and}, \emph{or}, oracle gate,
and (ii) there exists a deterministic Turing machine that decides for two given gates $u,v \in \{0,1\}^*$
of $C_n$ ($|u|, |v| \in \mathcal{O}(\log n)$) and a binary encoded integer $i$ with $\mathcal{O}(\log n)$ many bits
in time $\mathcal{O}(\log n)$ whether $u$ is the $i$-th input gate for $v$.
In the following, we write $\uACz(K)$ for \DLOGTIME-uniform $\Ac0(K)$.
For more details on these definitions we refer to~\cite{Vollmer99}.
If the language $L$ (or $K$) in the above definition of $\uACz(K)$ is defined over a non-binary alphabet $\Sigma$, then one first has to fix a binary
encoding of $\Sigma$ as words in $\{0,1\}^\ell$ for some large enough $\ell \in \N$.

If $\mathcal{C} = \{K_1, \ldots, K_n\}$ is a finite class of languages, then $\Ac0(\mathcal{C})$ is the same as
$\Ac0( \{ (w,i) \mid i \in [1,n], w \in K_i\} )$. If $\mathcal{C}$ is an infinite complexity class, then
$\uACz[\mathcal{C}]$ is the union of all classes $\uACz(K)$ for $K \in \mathcal{C}$.
Note that $\uACz[\mathcal{C}](K)$ is the same as
$\bigcup_{L \in \mathcal{C}} \uACz(K,L)$.

The class \Nc1 is defined as the class of languages accepted by
\DLOGTIME-uniform families of Boolean circuits having bounded fan-in, polynomial size, and logarithmic depth. As a consequence of Barrington's theorem \cite{Barrington86}, we have $\Nc1 = \uACz(\WP(A_5))$, where $A_5$ is the alternating group over 5 elements \cite[Corollary 4.54]{Vollmer99}.
	Moreover, the word problem for any finite group $G$ is in $\Nc1$. If $G$ is finite non-solvable, its word problem is $\Nc1$-complete~-- even under $\uACz$-many-one reductions. Robinson proved that the word problem for the free group $F_2$ is \Nc1-hard  \cite{Robinson93phd}, i.e., $\Nc1 \subseteq \uACz(\WP(F_2))$.
	%ASK

The class $\uTCz$ is defined as $\uACz(\text{\textsc{Majority}})$ where \textsc{Majority} is the language of all bit strings containing more $1$s than $0$s. %When talking about hardness for \uTCz or \Nc1 we use \uACz-Turing reductions unless stated otherwise.
Important problems that are complete (under \uACz-Turing reductions) for \uTCz are:
\begin{itemize}
	\item the languages $\{ w \in \{0,1\}^* \mid |w|_0 \leq |w|_1 \}$ and $\{ w \in \{0,1\}^* \mid |w|_0 = |w|_1 \}$, where $|w|_a$ denotes
	the number of occurrences of $a$ in $w$, see e.g. \cite{Vollmer99},
	\item the computation (of a certain bit) of the binary representation of the product of two or any (unbounded) number of
	binary encoded integers \cite{HeAlBa02},
	\item the computation (of a certain bit) of the binary representation of the integer quotient of two binary encoded integers \cite{HeAlBa02},
	\item the word problem for every infinite finitely generated solvable linear group \cite{KoenigL17},
	\item the conjugacy problem for the Baumslag-Solitar group  $\mathsf{BS}(1,2)$  \cite{DiekertMW14}.
\end{itemize}

\subsubsection{Counting complexity classes}
  Counting complexity classes are built on the idea of counting the number of accepting and rejecting computation paths of a Turing machine.
  For a non-deterministic Turing machine $M$, let $\operatorname{accept}_M$
  (resp., $\operatorname{reject}_M$)
  be the function that assigns to an input $x$ for $M$ the number of accepting
  (resp., rejecting) computation
 paths on input $x$.
  We define  the function $\operatorname{gap}_M : \Sigma^* \to \Z$ by $\operatorname{gap}_M(x) = \operatorname{accept}_M(x) - \operatorname{reject}_M(x)$.
  The class of functions $\GapL$ and the class of languages $\CeqL$ are defined as follows:
    {\allowdisplaybreaks\begin{align*}
                          \GapL &= \left\{ \operatorname{gap}_M \;\middle\vert\; \begin{matrix*}[l]
                               M \text{ is a non-deterministic, logarithmic space-bounded}\\
                            \text{Turing machine}
                          \end{matrix*} \right\}\\
                          \CeqL &= \left\{ L \;\middle\vert\; \text{there is } f \in \GapL \text{ with } \forall w \in \Sigma^* : w \in L \Longleftrightarrow f(w) = 0 \right\}
  \end{align*}}%
  We write $\GapL^K$ and  $\CeqL^{K}$ to denote the corresponding classes where the Turing machine $M$ is equipped with an oracle for the language $K$.
  We have the following relationships of $\CeqL$ with other complexity classes; see \eg~\cite{allender2004arithmetic}:
  \[
    \uTC^0 = \uAC^0(\WP(\mathbb{Z}))  \subseteq \uAC^0(\WP(F_2)) \subseteq \LOGSPACE \subseteq \NL \subseteq \CeqL \subseteq \uAC^0(\CeqL)
  \]

\section{Groups with an easy power word problem}

In this section we start with two easy examples of groups where the power word problem can be solved efficiently.

\begin{theorem}\label{thm:nilpotent_power_wp}
	If $G$ is a finitely generated nilpotent group, then $\PowWP(G)$ is in \uTCz.
\end{theorem}
\begin{proof}
	In \cite{MyasnikovW17}, the so-called word problem with binary exponents was shown to be in \uTCz. Here the input is a power word $u_1^{x_1} \cdots u_n^{x_n}$ but all the $u_i$ are required to be one of the standard generators of the group $G$. For arbitrary power words, we can apply the same techniques as in  \cite{MyasnikovW17}: we compute Mal'cev normal forms of all $u_i$ using \cite[Theorem~5]{MyasnikovW17}, then we use the power polynomials from \cite[Lemma~2]{MyasnikovW17} to compute Mal'cev normal forms with binary exponents of all $u_i^{x_i}$. Finally, we compute the Mal'cev normal form of $u_1^{x_1} \cdots u_n^{x_n}$ again using \cite[Theorem~5]{MyasnikovW17}.
\end{proof}
Theorem~\ref{thm:nilpotent_power_wp} has been generalized in \cite{FigeliusGLZ20}, where it is shown that the power word problem for a wreath product $G \wr \mathbb{Z}$
with $G$ finitely generated nilpotent belongs to \uTCz. Other classes of groups where the power word problem belongs to $\TC^0$ are iterated wreath products of the form
$\Z^r \wr (\Z^r \wr (\Z^r \cdots ))$, free solvable groups \cite{FigeliusGLZ20} and solvable Baumslag-Solitar group $\mathsf{BS}(1,q)$ \cite{LohreyZ20}.

The Grigorchuk group (defined in \cite{Grig80} and also known as the \emph{first Grigorchuk group}) is a finitely generated subgroup of the automorphism group of an infinite
binary rooted tree. It is a torsion group (every element has order $2^k$ for some $k$) and it was the first example of a group of intermediate growth.

\begin{theorem}\label{thm:grigorchuck}
	The power word problem for the Grigorchuk group is \uACz-many-one-reducible to its word problem (under suitable assumptions on the input encoding).
\end{theorem}

\begin{proof}
	Let $G$ denote the Grigorchuk group.
	By \cite[Theorem~6.6]{BartholdiGS03}, every element of $G$ that can be represented
 by a word of length $m$ over a finite set of generators has order at most $Cm^{3/2}$ for some constant $C$. \Wlog $C = 2^\ell$ for some $\ell \in \N$.
	On input of a power word $u_1^{x_1} \cdots u_n^{x_n}$
 with all words $u_i$ of length at most $m$, we can compute the smallest $k$ with $2^k \geq m$ in \uACz.
	We have $2^k \le 2m$. Now, we know that an element of length $m$ has order bounded by $2^{2k+\ell}$. Since the order of every element of $G$ is a power of two, this means that $g^{2^{2k+\ell}} = 1$ for all $g \in G$ of length at most $m$. Thus, we can reduce all exponents modulo $2^{2k+\ell}$ (\ie we drop all but the $2k+\ell$ least significant bits). Now all exponents are at most
	$2^{2k+\ell} \leq 4Cm^2$ and the power word can be written as an ordinary word (to do this in \uACz, we need a neutral letter to pad the output to a fixed word length). Note that this can be done by a uniform circuit family.
\end{proof}
\cref{thm:grigorchuck} applies only if the generating set contains a neutral letter. Otherwise, the reduction is in \uTCz. It is well-know that the word problem for the Grigorchuk group is in \L (see \eg\ \cite{Nekrashevych05,MiasnikovV17}). Thus, also the power word problem is in \L.  On the other hand, the compressed word problem (mentioned in the introduction) for the Grigorchuk group is \PSPACE-complete \cite{BartholdiFLW20}.

\section{Power word problems in finite extensions}

Also for finite groups the power word problem is easy; it belongs to \Nc{1}. The following result generalizes this fact:

\begin{theorem}\label{thm:finite_index}
	Let $G$ be finitely generated and let $H\leq G$ have finite index. Then $\PowWP(G)$ is \Nc{1}-many-one-reducible to $\PowWP(H)$.
\end{theorem}

\begin{proof}
	Since $H \leq G$ is of finite index, there is a normal subgroup $N \leq G$ of finite index with $N\leq  H$ (\eg $N= \bigcap_{g\in G} gHg^{-1}$). As $N \leq H$, $\PowWP(N)$ is reducible via a homomorphism (\ie in particular in \uTCz) to $\PowWP(H)$.
	Thus, we can assume that from the beginning $H$ is normal and that $Q=G/H$ is a finite quotient group.
	Notice that $H$ is finitely generated as $G$ is so; see e.g. \cite[1.6.11]{Rob96}.
	Let $R\sse G$ denote a set of representatives of $Q$ with $1 \in R$. If we choose a finite generating set $\Sigma$ for $H$, then
	$\Sigma \cup (R \setminus \{1\})$ becomes a finite generating set for $G$.

	Let $u = u_1^{x_1} \cdots u_n^{x_n}$ denote the input power word.
	As a first step, for every exponent $x_i$  we compute numbers $y_i,z_i \in \Z$ with
	$x_i = y_i\abs{Q} + z_i$ and $0 \leq z_i < \abs{Q}$ (\ie we compute the division with remainder by $\abs{Q}$).
	This is possible in \Nc{1} \cite{HeAlBa02}.
	Note that $u_i^{\abs{Q}}$ is trivial in the quotient $Q = G/H$ and, therefore, represents an element of $H$.
	Using the conjugate collection process from \cite[Theorem~5.2]{Robinson93phd} we can compute in $\Nc{1}$ a
	word $h_i \in \Sigma^*$ such that $u_i^{\abs{Q}} =_G h_i$.
	Then we replace in the input word every $u_i^{x_i}$ by $h_i^{y_i} u_i^{z_i}$ where we write $u_i^{z_i}$ as a word without exponents.
	We have obtained a word where all factors with exponents represent elements of $H$. Finally, we proceed like Robinson \cite{Robinson93phd} for the ordinary word problem
	treating words with exponents as single letters (this is possible because they are in $H$).

	To give some more details for the last step, let us denote the result of the previous step as
	$g_0 h_1^{y_1} g_1 \cdots h_n^{y_n} g_n$ with $g_i \in (\Sigma \cup R \setminus \{1\})^*$ and $ h_i \in \Sigma^*$. By \cite[Theorem~5.2]{Robinson93phd} we can rewrite in \Nc{1} $g_i$ as $g_i = \tilde h_i r_i$ with $r_i \in R$ and $\tilde h_i \in \Sigma^*$.
	Once again, we follow \cite{Robinson93phd} and write $\tilde h_0 r_0 h_1^{y_1} \tilde h_1 r_1 \cdots h_n^{y_n} \tilde h_n r_n$ as
	\[
	\tilde h_0 w_0 (a_1 h_1^{y_1} \tilde h_1a_1^{-1}) w_1  (a_2 h_2^{y_2} \tilde h_2a_2^{-1})w_2  \cdots (a_n h_n^{y_n} \tilde h_na_n^{-1})w_na_{n+1}
	\]
	where $a_i$ is the representative of $r_0 \cdots r_{i-1}$ in $R$
 ($a_0 = 1$) and $w_i = a_i r_i a_{i+1}^{-1}$. The element $(a_i h_i^{y_i} \tilde h_ia_i^{-1})$ belongs to $H$ since $H$ is normal in $G$. It is obtained from
 $h_i^{y_i} \tilde h_i$ by conjugation with $a_i$, i.e., by a homomorphism from
 a fixed finite set of homomorphisms. Thus, a power word $P_i$ over the alphabet $\Sigma$ with $P_i =_H (a_i h_i^{y_i} \tilde h_ia_i^{-1})$
 can be computed in $\uTCz$.
 Also all $w_i$ belong to $H$, since $a_i r_i$ and $a_{i+1}$ belong to the same coset of $H$. Moreover, every $w_i$
 comes from a fixed finite set (namely $R \cdot R \cdot R^{-1}$) and, thus,  can be rewritten to a word $w_i' \in \Sigma^*$. Now it remains to verify whether $a_{n+1}= 1$ (solving the word problem for $Q$, which is in $\uNC^1$). If this is not the case, we output any non-identity word in $H$, otherwise we output the power word
	$P = \tilde h_0 w_0' P_1 w_1' P_2 w_2'  \cdots P_n w_n'$. As $a_{n+1}= 1$, we have $P =_G u$.
\end{proof}

 \section{Power word problems in graph products}

The main results of this section are transfer theorems for the complexity
of the power word problem in graph products. We will prove such a transfer
theorem for the non-uniform setting (where the graph product is fixed) as well
as the uniform setting (where the graph product is part of the input).
Before, we will consider a special case, the so called
simple power word problem for graph products, in Section~\ref{sec:simple_pwp}.
In Section~\ref{ch:conjugacy-in-graph-groups-and-graph-products} we have to prove some further combinatorial results on
traces. Finally, in Section~\ref{ch:the-power-word-problem-in-graph-products} we prove the transfer theorems for
graph products.

 \subsection{The simple power word problem for graph products}\label{sec:simple_pwp}

In this section we consider a restricted version of the power word problem for graph products.
Later, we will use this restricted version in our algorithms for the unrestricted power word problem.

%\paragraph{The simple power word problem.}
Let $G = \GP(\cL, I, (G_\cLelA)_{\cLelA \in \cL})$ be a graph product
and define $\Gamma_\cLelA, \Sigma_\cLelA, \Gamma, \Sigma$ as in Section~\ref{sec:graph_prod_prelims}.
A \emph{simple power word} is a word $w = w_1^{x_1} \cdots w_n^{x_n}$, where $w_1, \dots, w_n \in \Gamma$
and $x_1, \dots, x_n \in \mathbb{Z}$ is a list of binary encoded integers.
Each $w_i$ encoded as a word over some finite alphabet $\Sigma_\cLelA$.
Note that this is more restrictive than a power word:
we only allow powers of elements from a single base group.
The \emph{simple power word problem} $\SPowWP(G)$ is to decide whether $w =_G 1$, where $w$ is a simple power word.
We also consider a uniform version of this problem. With $\USPowWP(\GP(\mathcal{C}))$ we denote the {\em uniform simple power word problem} for
graph products from the class $\GP(\mathcal{C})$ (see the last paragraph in
Section~\ref{sec:graph_prod_prelims}).
The following results on the complexity of the (uniform) simple power word problem are obtained by using the corresponding algorithm for the (uniform) word problem~\cite[Theorem 5.6.5, Theorem 5.6.14]{kausch2017parallel} and replacing the oracles for the word problems of the base groups with oracles for the power word problems in the base groups.

\begin{proposition}
	\label{lem:spowwp}
	For the (uniform) simple power word problem the following holds. %
	\begin{itemize}%
		\item Let $G = \GP(\cL, I, (G_\cLelA)_{\cLelA \in \cL})$ be a fixed graph product of f.g.~groups.
		Then $\SPowWP(G) \in \uAC^0\bigl(\{\WP(F_2)\} \cup \{ \PowWP(G_\cLelA) \mid \cLelA \in \cL \}\bigr)$.
		\item Let \(\mathcal{C}\) be a non-trivial class of f.g.~groups.
		Then $\USPowWP(\GP(\mathcal{C}))$ is in $\CeqL^{\UPowWP(\mathcal{C})}$.
	\end{itemize}
\end{proposition}

We adapt the proof from \cite{kausch2017parallel} for the word problem to the setting of the simple power word problem. The proofs for the non-uniform and uniform case are quite different. Indeed, in the non-uniform case, we can work by induction over the size of the (in-)dependence graph, while for the uniform case we rely on an embedding into some linear space of infinite dimension.  Therefore, we split the proofs into two subsections: in \cref{sec:simple_nonuniform} we work on the non-uniform case and later, in \cref{sec:the-uniform-case}, we develop an algorithm for the uniform case.

	\subsubsection{Input encoding}
	\label{sec:input-encoding-spwp}

	Let us give some details how to encode the input for the (simple) power word problem in graph products.
    There are certainly other ways how the represent the input for our algorithms without changing the complexity; but whenever the encoding is important, we assume that is is done as described in this section.
	We will use blocks of equal size to encode the different parts of the input. This makes it possible that parts of the computation can be done in $\uAC^0$. We assume that there is a letter for $1 \in \Sigma$ representing the group identity.

    The input of the power word problem in a graph product is $p_1^{x_1}\cdots p_n^{x_n}$ where $p_i = a_{i,1}\cdots a_{i,m_i} \in \Sigma^*$ (note that each letter of $\Gamma$ can be written as a word over $\Sigma$).
    We can pad with the identity element, so that each $p_i$ has length $n$, \ie $p_i = a_{i,1}\cdots a_{i,n}$ with $a_{i,j}\in \Sigma$.

     We encode each letter $a \in \Sigma$ as a tuple $(\cLelA, a)$ where $\cLelA = \alphabet(a)$.
    In the non-uniform case, there is a constant $k$, such that $k$ bits are sufficient to encode any element of $\cL$ and any letter of any $\Sigma_\cLelA$ for any $\cLelA \in \cL$.
    In the uniform case we encode the elements of $\cL$ as well as the elements of each  $\Sigma_\cLelA$  using $n$ bits.
    The encoding of a word $p_i$ is illustrated by the following figure.

\bigskip

	\begin{center}
	\small
		\begin{tikzpicture}[scale=.8]
			\draw (0,0) rectangle (2,1) node[pos=.5] {$\alphabet(a_{i,1})$};
			\draw (2,0) rectangle (4,1) node[pos=.5] {$a_{i,1}$};
			\draw (4,0) rectangle (8,1) node[pos=.5] {$\cdots$};
			\draw (8,0) rectangle (10,1) node[pos=.5] {$\alphabet(a_{i,n})$};
			\draw (10,0) rectangle (12,1) node[pos=.5] {$a_{i,n}$};
		\end{tikzpicture}
	\end{center}

\bigskip
    Encoding a word $p_i$ requires $2nk$ bits in the non-uniform case and $2n^2$ bits in the uniform case.
    For the simple power word problem we impose the restriction $\alphabet(a_{i,j}) = \alphabet(a_{i,k})$ for all $i,j ,k \in \{1, \dots, n\}$, as mixed powers are not allowed.

     We combine the above encoding for the words $p_i$ with a binary encoding for the exponents $x_i$ to obtain the encoding of a power word.
     Each exponent is encoded using $n$ bits.
     Note that we can do this because, if an exponent is smaller, we can pad it with zeroes and,
     if an exponent is larger, we can choose a larger $n$ and pad the input word with the identity element $1$.
     This leads us to the following encoding of a power word, which in the non-uniform case uses $(2k+1)n^2$ bits.

\bigskip

	\begin{center}
	\small
		\begin{tikzpicture}[scale=.8]
			\draw (0,0) rectangle (2,1) node[pos=.5] {$p_1$};
			\draw (2,0) rectangle (3,1) node[pos=.5] {$x_1$};
			\draw (3,0) rectangle (5,1) node[pos=.5] {$\cdots$};
			\draw (5,0) rectangle (7,1) node[pos=.5] {$p_n$};
			\draw (7,0) rectangle (8,1) node[pos=.5] {$x_n$};
		\end{tikzpicture}
	\end{center}
\bigskip

In the uniform case, this encoding requires $(2n+1)n^2$ bits.
 Furthermore, we also need to encode the descriptions of the base groups and the independence graph.
 By padding the input appropriately, we may assume that there are $n$ base groups and that each can be encoded using $n$ bits.
 The independence graph can be given as adjacency matrix, using $n^2$ bits.

\subsubsection{The non-uniform case}\label{sec:simple_nonuniform}

Before solving the simple power word problem in $G$ we prove several lemmata to help us achieve this goal. \Cref{lem:free_prod_am_prod} below is
due to Kausch~\cite{kausch2017parallel}. For this lemma, we have to introduce first
some notation and the notion of a semidirect product:
Take two groups $H$ and $N$ with a left action of $H$ on $N$
(a mapping $(h, g) \mapsto h \circ g$ for $h \in H$, $g \in N$ such that
$1 \circ g = g$, $(h_1 h_2) \circ g = h_1 \circ (h_2  \circ g)$ and for each $h \in H$ the map $g \mapsto h \circ g$ is an automorphism of $G$).
 The corresponding {\em semidirect product} $N \rtimes H$  is a group with underlying set $N \times H$ and the multiplication is defined
 by $(n_1, h_1) (n_2, h_2) = (n_1 (h_1 \circ n_2), h_1 h_2)$.

If $B$ is a group  and $u$ an arbitrary object, we write $B_u = \{ (g,u) \mid g \in B \}$ for an isomorphic copy of $B$ with multiplication $(g,u) (g',u) = (gg',u)$.
In the following let $B$ be finitely generated.
We begin by looking at the free product
$G \simeq *_{k \in \mathbb{N}} B_k$ of countable many copies of $B$.
Kausch~\cite[Lemma 5.4.5]{kausch2017parallel} has shown that the word problem
for $G$ can be solved in $\uAC^0$ with oracle gates for $\WP(B)$ and $\WP(F_2)$.
We show a similar result for the simple power word problem.
Our proof is mostly identical to the one presented in~\cite{kausch2017parallel}, with only a few changes to account for the different encoding of the input.
We use the following lemma on the algebraic structure of $G$.

\begin{lemma}{\cite[Lemma 5.4.4]{kausch2017parallel}}
	\label{lem:free_prod_am_prod}
	Let $B$ be a f.g.~group and $G = *_{k\in \mathbb{N}} B_k$.
	Then, we have $G \simeq F(X) \rtimes B$, where $F(X)$ is a free group with basis
	\begin{align*}
		X = \set{ (g,k) (g,0)^{-1} }{g \in B \setminus \{1\}, k \in \mathbb{N} \setminus \{0\} }
	\end{align*}
	and $g \in B$ acts on $F(X)$ by conjugating with $(g,0)$:
	\begin{align*}
	   \left(g\;,\;\; (h,k) (h,0)^{-1}\right) \mapsto (g,0) (h,k) (h,0)^{-1}  (g,0)^{-1}.
	\end{align*}
\end{lemma}
Note that
\[(g,0) (h,k) (h,0)^{-1}  (g,0)^{-1} = (g,0) (g,k)^{-1} (gh,k) (gh, 0)^{-1} \in F(X).\]
The choice of $0 \in \mathbb{N}$ in Lemma~\ref{lem:free_prod_am_prod} as the distinguished element from $\mathbb{N}$ is arbitrary.

With Lemma~\ref{lem:free_prod_am_prod}, we can solve the simple power word problem for $G$.

\begin{lemma}
	\label{lem:spowwp_free_prod}
	Let $B$ and $G$ be as in \cref{lem:free_prod_am_prod}. Given a power word
	\[ w = (w_1,k_1)^{x_1} \cdots (w_n,k_n)^{x_n}\in (B \times \mathbb{N} \times \mathbb{Z})^*, \] where the exponents $x_i \in \mathbb{Z}$ are encoded as binary numbers,
	one can decide in $\uAC^0(\{\PowWP(B), \WP(F_2) \})$ whether  $w =_G 1$.
\end{lemma}

\begin{proof}
	By \cref{lem:free_prod_am_prod} we have $G \simeq F(X) \rtimes B$.
	The set $X$ is given by
	\begin{align*}
		X = \{ (g,k)(g,0)^{-1} \; \mid \; g \in B \setminus \{1\},\, k \in \mathbb{N} \setminus \{0\} \} \text{.}
	\end{align*}
	Let $\varphi: G \to B$ be the homomorphism defined by $\varphi(b,k) = b$.
	We can assume $\varphi(w) = w_1^{x_1}\cdots w_n^{x_n} =_B 1$ as otherwise $w \neq_G 1$.
	Now our aim is to write $w$ as a member of the kernel of $\varphi$, which is $F(X)$.

	Let $g_i = w_1^{x_1}\cdots w_i^{x_i} \in (B \times \mathbb{Z})^*$.
	Observe that we can construct the $g_i$ in $\uAC^0$.
	We have
	\begin{align*}
		w =_G (g_1,k_1) \prod_{i=2}^{n} (g_{i-1}, k_i)^{-1} (g_i,k_i) \text{.}
	\end{align*}
	Using the fact that $g_n = \varphi(w) =_B 1$ (and hence $(g_n,k_n) =_G 1$), we can rewrite $w$ as part of the kernel over the basis $X$:
	\begin{align*}
		w  &=_G \prod_{i=1}^{n-1} (g_i,k_i) (g_i,k_{i+1})^{-1}\\
		&=_G \prod_{i=1}^{n-1} (g_i,k_i) (g_i,0)^{-1} (g_i,0) (g_i,k_{i+1})^{-1}\\
		&=_G \prod_{i=1}^{n-1} (g_i,k_i)  (g_i,0)^{-1} \bigl((g_i,k_{i+1})(g_i,0)^{-1}\bigr)^{-1} \text{.}
	\end{align*}
	Next we define a finite subset $Y \subseteq X$ such that $w \in F(Y) \le F(X)$.
	To achieve this we set
	\begin{align*}
		Y = \{ (g_i,k_i)(g_i,0)^{-1},\, (g_i,k_{i+1})(g_i,0)^{-1}\; \mid\; i \in [1, n-1] \}\text{.}
	\end{align*}
	From this definition it follows that $|Y| \le 2(n-1)$.
	Two elements $(g_i,k) (g_i,0)^{-1}$ and $(g_j,\ell)(g_j,0)^{-1}$ from $Y$ are equal if and only if $k = \ell$ and $g_i g_j^{-1} =_B 1$.
    Note that $g_i g_j^{-1} =_B 1$ is an instance of $\PowWP(B)$. Hence, using
    an oracle for $\PowWP(B)$ one can  decide whether two elements of $Y$ represent the same generator of $F(Y)$.

	As a last step we simplify the basis $Y$ by mapping it to the integer interval $[1, 2(n-1)]$.
	We use the following map $\psi: Y \to [1, 2(n-1)]$:
 \begin{eqnarray*} (g_i,v_i)(g_i,0)^{-1} & \mapsto &
 \min \{ j \le i \mid (g_i,k_i) =_G (g_j,k_j) \} \\
    (g_i,k_{i+1})(g_i,0)^{-1} & \mapsto &
    \left\{ \begin{matrix*}[l]
		 \min \{ j \leq n-1 \mid (g_i,k_{i+1}) =_G (g_j,k_j) \}\, \text{ if such a $j$ exists,}\\[2mm]
	      \min \{ j \leq n-1 \mid (g_i,k_{i+1}) =_G (g_j,k_{j+1}) \}\!+n-1\,  \text{ otherwise.}
		\end{matrix*} \right.
 \end{eqnarray*}
	The map $\psi$ can be computed in $\uAC^0$ with oracle gates for $\PowWP(B)$ and defines an isomorphism between $F(Y)$ and $F([1, 2(n-1)])$.
	It is well known that the free group $F(\mathbb{N})$ can be embedded into $F_2$ by the mapping $k \mapsto a^{-k} b a^k$. Since this
	mapping can be computed in $\uTC^0 \subseteq \uAC^0(\WP(F_2))$, we can finally check $w =_{F(Y)} 1$ in $\uAC^0(\{\PowWP(B), \WP(F_2) \})$.
\end{proof}
For the following lemma we need the notion of an amalgamated product.
For groups $A$, $P$ and $Q$ and injective homomorphisms $\phi: A \to P$ and $\psi: A \to Q$ the amalgamated product
$P *_{A} Q$ is the free product $P * Q$ modulo the relations $\{ \phi(a) = \psi(a) \mid a \in A \}$.
In the following, $A$ is a subgroup of $P$ and $Q$ and $\phi$ and $\psi$ are the identity.
The following lemma is due to Kausch~\cite{kausch2017parallel}.\footnote{In \cite{kausch2017parallel} the additional condition $\pi(b) = 1$ for all $b \in B$ is missing. This condition is needed in the proof of the lemma.}

\begin{lemma}{\cite[Lemma 5.5.2]{kausch2017parallel}}
	\label{lem:am_prod_kernel}
	Let $G = P *_A (B \times A)$ and consider the surjective homomorphism $\pi : G \to P$ with $\pi(g) = g$ for $g \in P$ and $\pi(b) = 1$ for all $b \in B$.
	Then we have $G \simeq (\ker \pi) \rtimes P$ and
	\begin{align*}
		\ker \pi \ \simeq \ *_{v \in P / A} B_v\text{,}
	\end{align*}
	where the isomorphism $\varphi : *_{v \in P / A} B_v \to \ker \pi$ maps $(b,v)$ to $v b v^{-1}$.
\end{lemma}

We want to solve the simple power word problem by induction. For the inductive step, we actually will need to solve the following slightly more general problem:

\begin{definition}%[generalized simple power word problem]
	Let $G$ be a graph product and $H \le G$ a fixed subgroup.
	We denote by $\GSPowWP(G, H)$ the \emph{generalized simple power word problem}:
	\ynproblem{A list of elements $a_1, \dots, a_n \in \Gamma$ and a list of binary encoded integers $x_1, \dots, x_n \in \mathbb{Z}$.}{Does $a_1^{x_1} \cdots a_n^{x_n} \in H$ hold?}
\end{definition}

\begin{lemma}
	\label{lem:gp_gspowwp}
	Let $G = \GP(\cL, I, (G_\cLelA)_{\cLelA \in \cL})$ be a graph product of f.g.~groups.
	For a subset $S \subseteq \cL$ we define the induced subgroup $G_S = \GP(S, I_S, (G_\cLelA)_{\cLelA \in S})$, where $I_S = I \cap (S \times S)$.
	We have
	\begin{align*}
		\GSPowWP(G_S, G) \in \uAC^0(\SPowWP(G)),
	\end{align*}
	that is the generalized simple power word problem $\GSPowWP(G_S, G)$ can be decided in $\uAC^0$ with an oracle for the simple power word problem in $G$.
\end{lemma}

\begin{proof}
	Consider the projection $\pi: \Gamma^* \mapsto \Gamma_S^*$ (where $\Gamma_S = \bigcup_{ \cLelA \in S} \Gamma_\cLelA$), with
	\begin{align*}
		\pi(a) = \left\{\begin{matrix*}[l]
			a\quad &\text{if }\alphabet(a) \in S,\\
			1 & \text{otherwise.}
		\end{matrix*}\right.
	\end{align*}
	Let $w = a_1^{x_1}\cdots a_n^{x_n}$ be the input to the generalized simple power word problem and let $\pi(w) = \pi(a_1)^{x_1}\cdots \pi(a_n)^{x_n}$
	We have $w =_G \pi(w)$ if and only if $w \in G_S$.
	This is equivalent to $w^{-1}\pi(w) =_G 1$. Moreover, the projection $\pi$ can be computed in $\uAC^0$ when elements of $\Gamma$ are represented by words
 from $\bigcup_{\cLelA\in\cL} \Sigma^*_\cLelA$  (since we assume $1 \in \Sigma$).
 \end{proof}

\begin{lemma}
	[\cref{lem:spowwp}, Part 1]
	\label{lem:gp_spowwp}
	Let $G = \GP(\cL, I, (G_\cLelA)_{\cLelA \in \cL})$ be a graph product of f.g.~groups.
	We have
	\begin{align*}
		\SPowWP(G) \in \uAC^0(\{\WP(F_2)\} \cup \{ \PowWP(G_\cLelA) \mid \cLelA \in \cL \}),
	\end{align*}
	that is the simple power word problem in $G$ can be solved in $\uAC^0$ with oracles for the power word problem in each base group $G_\cLelA$ and the word problem for the free group $F_2$.
\end{lemma}

\begin{proof}
	We proceed by induction on the cardinality of $\cL$.
	If $|\cL| = 1$, we can solve the simple power word problem in $G$ by solving the power word problem in the base group.
	Otherwise, fix an arbitrary $\cLelB \in \cL$.
	We define $\cL' = \cL \setminus \{\cLelB\}$, $I' = I \cap (\cL' \times \cL')$, $\link(\cLelB) = \{\cLelA\in\cL \mid (\cLelB, \cLelA) \in I\}$ and
	the three groups
	\begin{eqnarray*}
        P &=& \GP(\cL', I', (G_\cLelA)_{\cLelA \in \cL'}),  \\
        A &=& \GP(\link(\cLelB), I \cap (\link(\cLelB) \times \link(\cLelB)),(G_\cLelA)_{\cLelA \in \link(\cLelB)}), \\
        B & = & G_\cLelB .
         \end{eqnarray*}
	Now we can write $G$ as an amalgamated product: $G = P *_A (A \times B)$.

	By the induction hypothesis we can solve $\SPowWP(P)$ and $\SPowWP(A)$ in $\uAC^0$ with
	oracles for $\PowWP(G_\cLelA)$ (for all $\cLelA \in \cL$) and $\WP(F_2)$.
	By \cref{lem:gp_gspowwp} we can solve $\GSPowWP(A, P)$ in $\uAC^0$ with an
	oracle for $\SPowWP(P)$.
	It remains to show how to solve the simple power word problem in the amalgamated product.

	Let the input be $w = a_1^{x_1} \cdots a_n^{x_n} \in (\Gamma \times \mathbb{Z})^*$.
	Recall that $\Gamma_\cLelA = G_\cLelA \setminus \{1\}$ and $\Gamma = \bigcup_{ \cLelA \in \cL} \Gamma_\cLelA$,
	and let $\Gamma_P = \bigcup_{ \cLelA \in \cL' } \Gamma_\cLelA$ and $\Gamma_B = \Gamma_\cLelB$.
	We define the projections $\pi_P: \Gamma^* \to \Gamma_P^*$ and $\pi_B: \Gamma^* \to \Gamma_B^*$ by
	\begin{align*}
		\pi_P(a) &= \left\{ \begin{matrix*}[l]
			a\quad & \text{if } a \in  \Gamma_P,\\
			%a & \text{if } a \in \Gamma_A,\\
			1 & \text{if } a \in \Gamma_B,
		\end{matrix*} \right.
		&
		\pi_B(a) &= \left\{ \begin{matrix*}[l]
			1\quad & \text{if } a  \in \Gamma_P,\\
			%1 & \text{if } a \in \Gamma_A,\\
			a & \text{if } a \in \Gamma_B \text{.}
		\end{matrix*} \right.
	\end{align*}
	Let $p_i = \pi_P(a_i)$ and $b_i = \pi_B(a_i)$.
	Note that $b_i = 1$ or $p_i = 1$ for all $i$ since $w$ is a simple power word.
	For the following construction we assume that $\pi_P(w) = p_1^{x_1} \cdots p_n^{x_n} =_P 1$, i.e., $w \in \ker\pi_P$,
	as otherwise $w \neq_G 1$.
	By \cref{lem:am_prod_kernel} we have
	\[ G \simeq \left(*_{v \in P / A} B_v \right) \rtimes P, \]
	where $\ker\pi_P \simeq *_{v \in P / A} B_v$.
	We want to write $w$ as part of $\ker\pi_P$.
	We define $g_i = p_1^{x_1} \cdots p_i^{x_i}$.
	Note that the $g_i$ can be computed in $\uAC^0$.
	We have
	\[ w =_G g_1 b_1^{x_1} g_1^{-1} g_2 b_2^{x_2} \cdots g_{n-1}^{-1}g_n b_n^{x_n}. \]
	Observe that $g_n = \pi_P(w) =_P 1$ and thus
	\[ w =_G g_1 b_1^{x_1} g_1^{-1} \cdots g_n b_n^{x_n} g_n^{-1} \in *_{v \in P / A} B_v, \]
	where we identify every $g_i \in P$ with a coset representative of $A$.
	We compute
	\[
		\mu_i = \min \{ j \in [1,n] \mid g_{i}A = g_{j}A \} = \min \{ j \in [1,n] \mid g_{i}g_j^{-1} \in A\}\text{.}
	\]
	The computation can be reduced to $\SPowWP(P)$ in $\uAC^0$ by \cref{lem:gp_gspowwp}.

	Now we have $w =_G 1$ if and only if $\pi_P(w) =_P 1$ and
	\[ (b_1, g_{\mu_1}A)^{x_1} \cdots (b_n,g_{\mu_n}A)^{x_n}  = 1\] in $*_{v \in  P / A} B_v$ or, equivalently,  $(b_1,\mu_1)^{x_1} \cdots (b_n,\mu_n)^{x_n}  = 1$
	in $*_{\mu \in  \mathbb{N}} B_\mu$.
	The lemma follows with \cref{lem:spowwp_free_prod}.
\end{proof}

\subsubsection{The uniform case}
\label{sec:the-uniform-case}

Let $G = \GP(\cL, I, (G_\cLelA)_{\cLelA\in\cL})$ be a graph product of~f.g.~groups.
The following embedding of $G$ into a (possibly infinite-dimensional) linear group has been presented in~\cite{kausch2017parallel}.
We write $\mathbb{Z}^{(\Gamma)}$ for the free abelian group with basis $\Gamma$.
It consists of all mappings $f : \Gamma \to \mathbb{Z}$ such that
$f(c) \neq 0$ for only finitely many $c \in \Gamma$.
We write such a mapping $f$ as a formal sum $S = \sum_{c \in \Gamma} \lambda_c \cdot c$ with $\lambda_c = f(c) \in \mathbb{Z}$ and
call $\lambda_c$ the coefficient of $c$ in $S$.
The mapping $\sigma: G \to \operatorname{GL}(\mathbb{Z}^{(\Gamma)})$
 is defined by $w \mapsto \sigma_w$, where $\sigma_w = \sigma_{a_1}\cdots \sigma_{a_n}$ for  $w = a_1 \cdots a_n$ with $a_i \in \Gamma$.
For $a \in \Gamma$ the mapping $\sigma_a : \mathbb{Z}^{(\Gamma)} \to \mathbb{Z}^{(\Gamma)}$ is defined as the linear extension of

\begin{align*}
	\sigma_a(b) = \left\{ \begin{matrix*}[l]
		-a & \text{if } a, b \in \Gamma_{\cLelA} \text{ for some } \cLelA \text{ and } ab =_{G_\cLelA} 1,\\
		[ab] - a \quad & \text{if } a, b \in \Gamma_{\cLelA} \text{ for some } \cLelA \text{ and } ab \neq_{G_\cLelA} 1,\\
		b + 2a & \text{if } a \in \Gamma_{\cLelA}, b \in \Gamma_{\cLelB} \text{ for some } \cLelA \neq \cLelB \text{ and } (\cLelA, \cLelB) \notin I,\\
		b & \text{if } a \in \Gamma_{\cLelA}, b \in \Gamma_{\cLelB} \text{ for some } \cLelA \neq \cLelB \text{ and } (\cLelA, \cLelB) \in I.
	\end{matrix*} \right.
\end{align*}

\begin{lemma}{\cite[Lemma 3.3.4]{kausch2017parallel}}
	\label{lem:trace_sigma_sum}
	Let $w \in \Gamma^*$ be reduced and $wb =_G ubv$ such that $b \in \Gamma_\cLelB$, $u,v \in \Gamma^*$, $(b,v) \in I$ and $b$ is the unique maximal letter of $ub$.
	Moreover, let
	\begin{align*}
		\sigma_w(b) = \sum_{c\in\Gamma} \lambda_c\cdot c\text{,}
	\end{align*}
	and let $u = u_{0}a_1 u_1 \cdots a_n u_n$ with $a_i \in \Gamma_\cLelA$ and $u_i \in (\Gamma \setminus \Gamma_\cLelA)^*$. Then for all $c \in \Gamma_\cLelA$ we have $\lambda_c \ge 0$ and
	\begin{align*}
		\lambda_c > 0 \ \Longleftrightarrow \ \text{for some } i \in \{1,\dots, n\} : c =_{G_\cLelA} \left\{ \begin{matrix*}[l]
			a_i\cdots a_n &\text{if }\cLelA \neq \cLelB\text{,}\\
			a_i\cdots a_n b \quad &\text{if } \cLelA = \cLelB\text{.}
		\end{matrix*} \right.
	\end{align*}
\end{lemma}

\begin{algorithm}
	\small\begin{algorithmic}[5]
		\State \textbf{Input:} $a \in \Gamma_\cLelA, a_1^{x_1} a_2^{x_2} \cdots a_n^{x_n}$ with $a_i \in \Gamma_{\cLelA_i}$ and
  $b_i = [a_i^{x_i}] \in \Gamma_{\cLelA_i}$
		\State $(k, \ell, s) \gets (n + 1, n + 1, 1)$
		\For{$i$ in $[n, \dots, 1]$}
		\If{\(k = n + 1 \land \ell = n + 1\)}
		\Comment{\(\sigma_{b_i}(\chi) = 2b_i + \chi\)}
		\State Guess ``branch 1'', ``branch 2'' or ``branch 3''
		\State \textbf{if} ``branch 1'' or ``branch 2'' \textbf{then} $(k, \ell, s) \gets (i, i, s)$
		\State \textbf{if} ``branch 3'' \textbf{then} $(k, \ell, s) \gets (n + 1, n + 1, s)$
		\Else
			\If
				{\(\cLelA_i = \cLelA_k \land b_i\pi_{\cLelA_k}(b_k\cdots b_\ell) =_{G_{\cLelA_i}} 1\)}
			\Comment{\(\sigma_{b_i}(\pi_{\cLelA_k}(b_k\cdots b_\ell)) = - b_i\)}
				\State $(k, \ell, s) \gets (i, i, -s)$
		\ElsIf
		{\(\cLelA_i = \cLelA_k\)}
		\Comment{\(\sigma_{b_i}(\pi_{\cLelA_k}(b_k\cdots b_\ell)) = [b_i\pi_{\cLelA_k}(b_k\cdots b_\ell)] - b_i\)}
		\State Guess ``branch 1'' or ``branch 2''
		\State \textbf{if} ``branch 1'' \textbf{then} $(k, \ell, s) \gets (i, \ell, s)$
		\State \textbf{if} ``branch 2'' \textbf{then} $(k, \ell, s) \gets (i, i, -s)$
		\ElsIf
		{\((\cLelA_i, \cLelA_k) \notin I\)}
		\Comment{\(\sigma_{b_i}(\pi_{\cLelA_k}(b_k\cdots b_\ell)) = \pi_{\cLelA_k}(b_k\cdots b_\ell) + 2b_i\)}
		\State Guess ``branch 1'', ``branch 2'' or ``branch 3''
		\State \textbf{if} ``branch 1'' \textbf{then} $(k, \ell, s) \gets (k, \ell, s)$
		\State \textbf{if} ``branch 2'' or ``branch 3'' \textbf{then} $(k, \ell, s) \gets (i, i, s)$
		\Else
		\Comment{\(\sigma_{b_i}(\pi_{\cLelA_k}(b_k\cdots b_\ell)) = \pi_{\cLelA_k}(b_k\cdots b_\ell)\)}
		\State $(k, \ell, s) \gets (k, \ell, s)$
		\EndIf
		\EndIf
		\EndFor
		\If{$k \neq n+1 \land \cLelA_k = \cLelA \land a =_{G_\cLelA} \pi_\cLelA(b_k\cdots b_\ell)$}
		%\\ \Comment{Use oracle for \(\UPowWP(\mathcal{C})\)}
		\State \textbf{if} {$s = 1$}  \textbf{then} {accept}
		\State \textbf{if} {$s = -1$}  \textbf{then} {reject}
		\Else
		\State Guess ``branch 1'' or ``branch 2''\;
		\State \textbf{if} {``branch 1''}  \textbf{then} {accept}
		\State \textbf{if} {``branch 2''}  \textbf{then} {reject}
		\EndIf
	\end{algorithmic}
	\caption{Computing the coefficient of $a \in \Gamma_\cLelA$ in $\sigma_w(\chi)$}
	\label{alg:sigma_coefficients_GapL}
\end{algorithm}

\begin{figure}
	\begin{tikzpicture}[level/.style={sibling distance=38mm/#1}]
		\node {$\chi$}
		child { node {$b$}
			child { node {$b$}
				child { node {accept}}
				child { node {reject}}
			}
		}
		child { node {$b$}
			child { node {$b$}
				child { node {accept}}
				child { node {reject}}
			}
		}
		child { node {$\chi$}
			child { node {$a$}
				child { node {accept}}
			}
			child { node {$a$}
				child { node {accept}}
			}
			child { node {$\chi$}
				child { node {accept}}
				child { node {reject}}
			}
		};
	\end{tikzpicture}
	\caption{Computation of the coefficient $\lambda_a$ of $\sigma_{ab}(\chi)$ by \cref{alg:sigma_coefficients_GapL}. We assume $(a, b) \in I$. Each inner node is labeled with the coefficient it contributes to. The algorithm stores the coefficient using two indices $k$ and $\ell$. The nodes on the second level correspond to $\sigma_b(\chi) = 2b+\chi$, the nodes on the third level correspond to $\sigma_{ab}(\chi) = \sigma_a(2b+\chi) = 2b+2a+\chi$. If they are labeled with $a$, they have one leaf node as a child which is an accepting path (or a rejecting path if the sign is negative -- here all signs are positive). If they are not labeled with $a$, then there are two leaf node children, one is an accepting path, the other a rejecting path, so they do not affect the difference of accepting and rejecting paths. Here, the difference of accepting and rejecting paths is $2$, which is the coefficient $\lambda_a$ of $\sigma_{ab}(\chi)$.}
	\label{fig:coefficient_computation_example}
\end{figure}
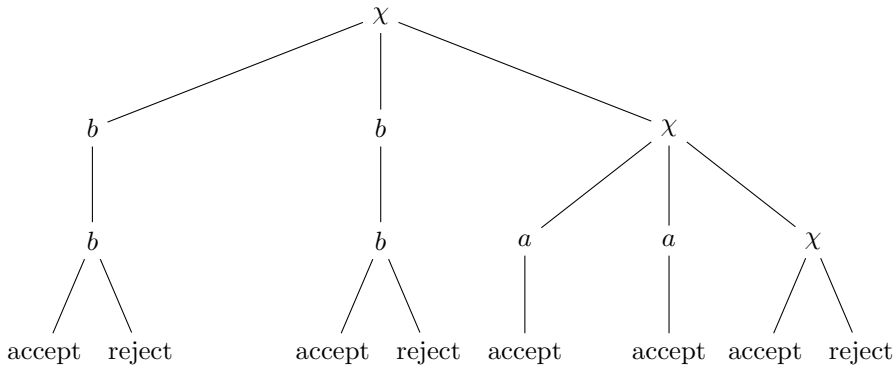

Our solution to the uniform simple power word problem is based on the solution to the word problem presented in~\cite{kausch2017parallel}.
The underlying idea is to add an additional free group $\left< \chi \right>$
for a new generator $\chi$ to the graph product, which is dependent on all other groups.
Let $\pi_\cLelA$ be the projection onto $\Gamma_\cLelA$, defined by $\pi_\cLelA(a) = a$ for $a \in \Gamma_\cLelA$ and $\pi_\cLelA(a) = 1$ for $a \notin \Gamma_\cLelA$.
As a consequence of \cref{lem:trace_sigma_sum} we have $\sigma_w(\chi) = \chi$ if and only if $w =_G 1$.
Non-zero coefficients of $\sigma_w(\chi)$ are coefficients of $[u]$ for a factor $u$ of $\pi_\cLelA(w)$ for some $\cLelA \in \cL$.

\begin{lemma}
	[\cref{lem:spowwp}, Part 2]
	\label{lem:gp_spowwp_uniform}
	Let \(\mathcal{C}\) be a non-trivial class of f.g.~groups. Then $\USPowWP(\GP(\mathcal{C}))$ belongs to $\CeqL^{\UPowWP(\mathcal{C})}$.
\end{lemma}

\begin{proof}
	Let $w = a_1^{x_1}\cdots a_n^{x_n} \in G$, where $a_i \in \Gamma_{\cLelA_i}$ and $x_i \in \mathbb{Z}$.
	If $w \neq_G 1$ then there are $\cLelA \in \cL$ and $1 \le k \le \ell \le n$ such that the coefficient of
	$[\pi_\cLelA(a_k^{x_k} \cdots a_\ell^{x_\ell})]$ in $\sigma_w(\chi)$ is not zero.

	To compute the coefficients we use \cref{alg:sigma_coefficients_GapL}.
	For simplicity we assume $a_i^{x_i} \neq_{G_{\cLelA_i}} 1$ for all $i \in [1, n]$.
	This can be enforced by a precomputation using $\UPowWP(\mathcal{C})$ as an oracle.
	Let $b_i \in \Gamma_{\cLelA_i}$ with $b_i =_{G_{\cLelA_i}} a_i^{x_i}$.

         Our nondeterministic logspace algorithm will produce a computation tree such that the coefficient of $a \in \Gamma$ in $\sigma_w(\chi)$
         will the number of accepting leaves minus the number of rejecting leaves (as required by the definition of $\GapL$).
	The algorithm stores in each configuration an element $[\pi_{\cLelA_k}(b_k \cdots b_\ell] \in \Gamma$ using the two
	indices $k$ and $\ell$. We use $(k, \ell) = (n+1, n+1)$ to represent $\chi$. In addition to $k$ and $\ell$ we store a sign $s$ ($1$ or $-1$), saying
	whether the configuration gives a positive or negative contribution to the coefficient of $[\pi_{\cLelA_k}(b_k \cdots b_\ell]$.

	The root node of the computation tree corresponds to $\chi$.
	Let $w = w'a$, with $a \in \Gamma$.
	Then $\sigma_w(\chi) = \sigma_{w'}(\sigma_a(\chi))$.
	The nodes on the second level, that is the children of the root node, correspond to $\sigma_a(\chi)$.
	The last level made up of inner nodes corresponds to $\sigma_w(\chi)$.
	At that point the algorithm checks if the node corresponds to the input element $a \in \Gamma$, i.e., whether
	$a = [\pi_{\cLelA_k}(b_k\cdots b_\ell)]$ holds.
	This is done using the oracle for the uniform power word problem in $\mathcal{C}$.
	If it holds, then the computation will accept the input if the stored sign $s$ is $1$, and reject if $s=-1$.
	If $a = [\pi_{\cLelA_k}(b_k\cdots b_\ell)]$ does not hold, then the algorithm branches into two leaf nodes,
	one accepting and one rejecting, which gives a zero contribution to the coefficient of $a$.
	In this way, it is ensured that the
	coefficient of $a$ is the difference of the number of accepting paths and the number of rejecting paths.
	Therefore, the computation of a coefficient is in $\GapL^{\UPowWP(\mathcal{C})}$, and we can check in $\CeqL^{\UPowWP(\mathcal{C})}$ whether a coefficient is zero.
	An example of a computation tree is presented in \cref{fig:coefficient_computation_example}.

	Finally, we can check in  $\CeqL^{\UPowWP(\mathcal{C})}$ whether all coefficients of  elements $[\pi_{\cLelA_k}(b_k\cdots b_\ell)]$
	are zero, as $\CeqL$ is closed under conjunctive truth table reductions~\cite[Proposition 17]{AllenderO96} and the proof holds for every relativized version $\CeqL^{A}$.
\end{proof}

  \subsection{Combinatorics on Traces}
  \label{ch:conjugacy-in-graph-groups-and-graph-products}

  In this section we develop various tools concerning combinatorics on traces, which later will be used to solve the power word problem in graph products.
As a motivation and an easy example, we start with the analogous construction for free groups we presented in~\cite{LohreyW19},  before looking into the more technical case of graph products.
The first task for solving the power word problem in a free group is to compute certain unique normal forms for the words $u_i$ of an instance of the power word problem as in \eqref{eq:word-w} below.

We use the notation from \cref{sec-free-groups}. In particular, we use the rewriting system $\freeRS= \set{a \finv{a} \to 1}{a \in \Sig}$.
Fix an arbitrary order on the input alphabet $\Sigma$. This gives us a lexicographic order on $\Sigma^*$, which is denoted by $\preceq$.
Let $\Omega \sse \IRR(\freeRS) \sse \Sigma^*$
denote the set of words $w$ such that
\begin{itemize}
	\item $w$ is non-empty,
	\item $w$ is cyclically reduced (i.e, $w$ cannot be written as $a u \finv{a}$ for $a \in \Sigma$),
	\item $w$ is primitive (i.e, $w$ cannot be written as $u^n$ for $n \ge 2$),
	\item $w$ is lexicographically minimal among all cyclic permutations of $w$ and $\finv{w}$ (\ie $w \preceq uv$ for all $u,v \in \Sigma^*$ with $vu =w$ or $vu = \finv{w}$).
\end{itemize}
Notice that $\Omega$ consists of Lyndon words \cite[Chapter 5.1]{lot83} with the stronger requirement of being freely reduced, cyclically reduced and also minimal among the conjugacy class of the inverse. %\cite{Lyndon54,BerstelP07}
In \cite{LohreyW19}, the first step is to rewrite the input power word in the form
\begin{align} \label{eq:word-w}
\qquad\qquad w&=s_0u_1^{x_1}s_1 \cdots u_n^{x_n}s_n\qquad \text{with } u_i \in \Omega \text{ and } s_i \in \IRR(\freeRS).
\end{align}
This transformation can be done by a rather easy $\uACz(F_2)$ computation.
The reason to do this lies in the following crucial lemma: essentially it says that, if a long factor of $u_i^{x_i}$ cancels with some $u_j^{x_j}$, then already $u_i= u_j$. Thus, if a power word of the form \eqref{eq:word-w} represents the group identity, every $u_i$ with a large exponent must cancel with other occurrences of the very same word $u_i$. %
Thus, only the same $u_i$ can cancel implying that we can make the exponents of the different $u_i$ independently smaller.

\begin{lemma}\label{lem:short_cancellation}
	Let $p,q \in \Omega$, $x,y \in \mathbb{Z}$ and let $v$ be a factor of $p^x$ and $w$ a factor of  $q^{y}$. If $vw \rewrite{*}{\freeRS} 1$ and $\abs{v} = \abs{w} \geq \abs{p} + \abs{q} - 1$, then $p=q$.
\end{lemma}

\begin{proof}
	Since $p$ and $q$ are cyclically reduced, $v$ and $w$ are freely reduced, i.e., $v=\finv{w}$ as words. Thus, $v$ has two periods $\abs{p}$ and $\abs{q}$. Since $v$ is long enough, by the theorem of Fine and Wilf~\cite{fine1965uniqueness} it also has the period $\gcd(\abs{p},\abs{q})$. This means that also $p$ and $q$ have period $\gcd(\abs{p},\abs{q})$ (since cyclic permutations of $p$ and $q$ are factors of $v$). Assuming $\gcd(\abs{p},\abs{q}) < \abs{p}$, would mean that $p$ is a proper power contradicting the fact that $p$ is primitive. Hence, $\abs{p}=\abs{q}$. Since $\abs{v} \geq  \abs{p} + \abs{q}  - 1= 2\abs{p} -1$, $p$ is a factor of $v$, which itself is a factor of $q^{-y}$. Thus, $p$ is a cyclic permutation of $q$ or of $\finv{q}$. By the last condition on $\Omega$, this implies $p = q$.
\end{proof}

In the remainder of this section, we develop the requirements for a special normal form (like $\Omega$ above) and generalize \cref{lem:short_cancellation} to graph products.
In particular, we aim for some special kind of cyclic normal forms ensuring uniqueness within a conjugacy class (see \cref{def:omega} below).

Let us fix a graph product $G = \GP(\cL, I, (G_\cLelA)_{\cLelA\in\cL})$ with $G_\cLelA$ finitely generated by $\Sigma_\cLelA$,
define the sets $\Gamma_\cLelA = G_\cLelA \setminus \{ 1 \}$,
  $\Gamma = \bigcup_{\cLelA \in \cL} \Gamma_\cLelA$ and $\Sigma = \bigcup_{\cLelA \in \cL} \Sigma_\cLelA$ as before
and let $M = M(\Gamma,I)$ be the corresponding trace monoid.

\subsubsection{Cyclic normal forms and conjugacy}

Recall that by \cref{lemm:conjugacy_equivalent},
traces $u,v \in M$ are conjugate if and only if they are related by a sequence of transpositions.

\begin{lemma}[\mbox{\cite[Lemma 7.3.8]{kausch2017parallel}}] \label{lem:CPapprox}
  	Let $G = \GP(\cL, I, \left(G_\cLelA\right)_{\cLelA \in \cL})$ be a graph product and let $u, v \in M$ be cyclically reduced, connected and composite.
  	Then $u$ and $v$ are conjugate in $G$ if and only if $u$ and $v$ are conjugate in $M$.
\end{lemma}

  Note that \cite[Lemma 7.3.8]{kausch2017parallel} requires $\alphabet(u) = \alphabet(v)$.
  We do not need this requirement as on the one hand $u$ and $v$ being conjugate in $M$ clearly implies $\alphabet(u) = \alphabet(v)$, and on the other hand by \cite[Lemma 7.3.6]{kausch2017parallel} $u$ and $v$ being cyclically reduced and conjugate in $G$ implies $\alphabet(u) = \alphabet(v)$.

  By $\le_\cL$ we denote a linear order on the set $\cL$. For $a,b \in \Gamma$ we write $a <_{\cL} b$ if
  $\alphabet(a) <_{\cL} \alphabet(b)$.
  The \emph{length-lexicographic normal form} of $g \in G$ is the reduced representative $\nf_G(g) = w \in \Gamma^*$ for $g$ that is lexicographically smallest.
  Note that this normal form is on the level of $\Gamma$.
  Each letter of $\Gamma$ still might have different representations over the finite generating set $\Sigma$ as outlined above.
  If $G$ is clear from the context, we also write $\nf(g)$. Moreover, for a word $u \in \Gamma^*$ (or trace $u \in M$) we write $\nf(u)$ for $\nf(g)$, where $g$ is the group element represented by $u$.

  \begin{definition}
    \label{def:cyclic_nf}
    Let $w \in \Gamma^*$.
    We say $w$ is a \emph{cyclic normal form} if $w$ and all its cyclic permutations are length-lexicographic normal forms and $w$ is composite.
  \end{definition}

	\begin{remark}
		\label{rem:cnf_cycl_perm}
		Observe that if $w$ is a cyclic normal form, then as a trace from $M$ it is cyclically reduced and all cyclic permutations of $w$ are cyclic normal forms themselves.
	\end{remark}

  Cyclic normal forms have been introduced in~\cite{crisp2009conjugacy} for RAAGs. Moreover, by~\cite{crisp2009conjugacy}, given $w \in \Gamma^*$, which has a cyclic normal form, a cyclic normal form for $w$ can be computed in linear time.
  In \cref{th:cyclic_nf_graph_product} below, we show that cyclic normal forms also exist for certain elements in the case of graph products and that they also can be computed efficiently.

  It is easy to see that every cyclic normal form is connected (see \cref{rem:cyclicNFconnected}).
  In particular, not every element has a cyclic normal form.
Moreover, there can be more than one cyclic normal form per conjugacy class; however, by \cref{lem:cnf_conjugate} below they are all cyclic permutations of each other.

  \begin{remark}
    \label{rem:cyclicNFconnected}
    Notice that, if $w \in \Gamma^*$ is a cyclic normal form, then it is connected. Indeed, let $d \in \Gamma$ be a $\le_\cL$-largest letter occurring in $w$. After a cyclic permutation we can write $w = dw'$. Now, assume that $w =_{M} uv$ with $(u,v) \in I$.
    Without loss of generality $d$ belongs to $u$. Let $c$ denote the first letter of $v$. Since $(u,v) \in I$,
    we must have $\alphabet(c) \neq \alphabet(d)$ and therefore $c <_\cL d$.
    Since $(c,u) \in I$ we obtain $w =_M cw''$ for some $w''$.
    But then $w$ cannot start with $d$, which is contradiction.
  \end{remark}

For the following considerations, it is useful to
embed the trace monoid $M = M(\Gamma,I)$ (and, thus, $\IRR(\traceRS)$) via the trace monoid $M(\Gamma \cup \finv{\Gamma},I)$ into the right-angled Artin group $G(\Gamma, I)$ as in \eqref{eq:irr_embedding}. Note that this means that we add a formal inverse $\finv{a}$ for every $a \in \Gamma$ (which is different from the inverse $a^{-1}$ of $a$ in the group $G_{\alphabet(a)}$).
Be aware that $\Gamma$ might be infinite and that a trace $u \in M(\Gamma \cup \finv{\Gamma},I)$ is reduced with respect to $G(\Gamma, I)$ if it does not contain a
$a\finv{a}$ or $\finv{a}a$ for $a \in \Gamma$ (but it may contain a factor
$ab$ with $\alphabet(a) = \alphabet(b)$ and therefore be non-reduced with respect to the graph product $G$).

\begin{lemma} \label{lemma:conjugateM=conjugateG}
  Two traces $u,v \in M(\Gamma,I)$ are conjugate in $M(\Gamma,I)$ if and only if they are conjugate in the RAAG $G(\Gamma, I)$.
\end{lemma}

\begin{proof}
    This follows from \cite[Lemma 2.9]{crisp2009conjugacy}, which states that
    traces $u,v \in M(\Gamma \cup \finv{\Gamma},I)$ that are cyclically reduced
    with respect to the RAAG $G(\Gamma,I)$ (i.e., no transposition of $u$ or $v$
    contains a factor $a\finv{a}$ or $\finv{a}a$ for $a \in \Gamma$)
    are conjugate in the RAAG $G(\Gamma,I)$ if and only
    if $u$ and $v$ are related by a sequence of transpositions. Note that
    traces $u,v \in M(\Gamma,I)$ are always cyclically reduced in the RAAG-meaning.
    Hence, $u,v \in M(\Gamma,I)$ are conjugate in $G(\Gamma,I)$ if and only
    if they are related by a sequence of transpositions, i.e., if and only if
    they are conjugate in $M(\Gamma,I)$.
\end{proof}

  \begin{lemma}\label{lem:cnf_conjugate}
    Let $u, v\in \Gamma^*$ be cyclic normal forms.
    Then $u$ and $v$ are conjugate in $G$ (or, equivalently, conjugate in $M$ by \cref{lem:CPapprox} and \cref{rem:cyclicNFconnected})
    if and only if the word $u$ is a cyclic permutation of the word $v$.
  \end{lemma}

\begin{proof}
    The lemma can be shown by almost a verbatim repetition of the proof of \cite[Proposition 2.21]{crisp2009conjugacy}. However, we can also use that result as a black-box: it states that two cyclic normal forms in a RAAG are conjugate if and only if they are cyclic permutations of each other.\footnote{A word $w \in (\Gamma \cup \finv{\Gamma})^*$ is a cyclic
    normal form in the RAAG $G(\Gamma,I)$ if $w$ and all its cyclic permutations are length-lexicographic normal forms. The latter means
    that $w$ and all its cyclic permutations are reduced with respect to
    the RAAG $G(\Gamma,I)$ and lexicographic normal forms with respect to a
    fixed linear order on $\Gamma$; see \cite[Definition~2.19]{crisp2009conjugacy}.}

    We apply this result to the RAAG $G(\Gamma,I)$. In
    \cite[Proposition 2.21]{crisp2009conjugacy} it is assumed that
    $\Gamma$ is finite, whereas our $\Gamma$ is infinite. But we can
    restrict $\Gamma$ do those symbols that appear in $u$ and $v$.

    Moreover, while we are given a linear order on $\cL$, we need a linear order on $\Gamma$ to obtain the notion of a cyclic normal form in a RAAG.
    To solve this problem we fix
    on each $\Gamma_\cLelA$ for $\cLelA \in \cL$ an arbitrary linear order (for different $\cLelA,\cLelB \in \cL$ we use our order on $\cL$). This gives a linear order on $\Gamma$.  As our definition of $\IRR(\traceRS)$ implies that there are never two consecutive letters from $\Gamma_\cLelA$ for the same $\cLelA \in \cL$, the outcome for the cyclic normal form does not depend on the actual orders we chose on the $\Gamma_\cLelA.$
    Therefore, every cyclic normal form according to \cref{def:cyclic_nf} is also a cyclic normal form in the RAAG $G(\Gamma, I)$.

    Let us now take to cyclic normal forms $u,v \in \Gamma^*$
    according to \cref{def:cyclic_nf}. Then $u$ and $v$ are also
    cyclic normal forms with respect to the RAAG $G(\Gamma,I)$.
    As traces from $M = M(\Gamma,I)$, $u$ and $v$ are cyclically reduced,
    connected and composite.
    Therefore, by \cref{lem:CPapprox},
    $u$ and $v$ are conjugate in the graph product $G$ if and only
    if they are conjugate as traces from $M(\Gamma,I)$.
    By \cref{lemma:conjugateM=conjugateG} the latter holds if and only
    if $u$ and $v$ are conjugate in the RAAG $G(\Gamma,I)$. Finally,
    \cite[Proposition 2.21]{crisp2009conjugacy}, tells us that the latter
    is equivalent to $v$ being a cyclic permutation of $u$.
\end{proof}

  \begin{lemma}
    \label{lem:nf_k}
    Let $w = dw' \in \Gamma^*$ (with $d \in \Gamma$) be a cyclically reduced and composite length-lexicographic normal form such that $w$ does not contain any letter $c$ with $d <_{\cL} c$.
    Then for all $k \ge 1$ we have $\nf(w^k) = w^k$ and $w$ is a cyclic normal form.
  \end{lemma}

  \begin{proof}
  Note that $d$
  must be the unique minimal letter in the trace represented by $w$: if $a \neq d$
  would be also minimal, then $(a,d) \in I$ (in particular, $a$ and $d$ do not belong
  to the same $\Gamma_\cLelA$) and $d <_{\cL} a$ (since $dw'$ is a
  length-lexicographic normal form),
  contradicting the assumption on $d$. In particular, $w$ must be connected as a trace.

    We show that $\nf(w^k) = w^k$ by showing that $w^k$ is a length-lexicographic normal form.
    Since $w$ is cyclically reduced (and hence in particular reduced), connected and composite, also $w^k$ is cyclically reduced and composite.

    Let us now prove that $w^k$ is a length-lexicographic normal form: Assume the converse. The characterization of lexicographically smallest words of Anisimov
    and Knuth \cite{AnisimovK79} implies that $w^k$ contains a factor $bua$ where $a <_{\cL} b$
    and $(a,bu) \in I$. Since $w$ is a length-lexicographic normal form,
    the factor $bua$ does not belong to some factor $w$ of $w^k$.
    Therefore, $w$ has a prefix $ya$, where $y$ is a suffix of $u$.
    Since $(a,y) \in I$, $a$ is a minimal letter of $w$. Since $d$ is the unique
    minimal letter of $w$, we have $d = a$, i.e., $d <_{\cL} b$, which contradicts the assumptions on $d$. Hence, $w$ is indeed a
    length-lexicographic normal form, i.e., $\nf(w^k) = w^k$.

    Finally, we show that $w$ is a cyclic normal form. Every cyclic permutation $v$
    of $w$ is a factor of $w^2$.
    Since every factor of a length-lexicographic normal form is again a length-lexicographic normal form, $v$ is a length-lexicographic normal form.
\end{proof}

	\begin{corollary}
		Let $u = du' \in \Gamma^*$ (with $d \in \Gamma$) be a cyclic normal form such that $u$ does not contain any letter $c$ with $d <_{\cL} c$.
		If $u =_M w^k$ for a trace $w$, then $\nf(w)$ is a cyclic normal form and $u = \nf(w)^k$ (as words).
		\label{cor:cyclic_normal_form_primitive}
	\end{corollary}

	\begin{proof}
        The case $k=1$ is trivial, so let us assume that $k \geq 2$.
		Let $w \in \Gamma^*$ such that $u =_M w^k$.
        By the argument from the proof of \cref{lem:nf_k}, $d$ is the unique minimal
        letter of $u$.
        Since $u$ is composite and connected and $\alphabet(u) = \alphabet(w)$, also
        $w$ is composite and connected.
		As $ww$ is a factor of $u$, $ww$ must be reduced. Hence,
        $w$ is cyclically reduced. We can also write $w$ as $w =_M dw'$
        and $d$ is also the unique minimal letter of $w$. In particular,
        $\nf(w) = d v$ for some word $v \in \Gamma^*$.
		Applying \cref{lem:nf_k} (with $w$ replaced by $\nf(w)$) we obtain that $\nf(w^k) = \nf(\nf(w)^k) = \nf(w)^k$ and $\nf(w)$ is a cyclic normal form.
	\end{proof}

  \begin{theorem}
    \label{th:cyclic_nf_graph_product}
  The following holds:
    \begin{itemize}
      \item Let $G = \GP(\cL, I, \left(G_\cLelA\right)_{\cLelA \in \cL})$ be a graph product of f.g.~groups. Then the following problem is in $\uTC^0 \subseteq \uAC^0(\WP(F_2))$:
      \compproblem{a cyclically reduced, composite  and connected $w \in \Gamma^*$}{a cyclic normal form that is conjugate in $G$ to $w$}
      \item Let  \(\mathcal{C}\) be any non-trivial class of f.g.~groups.
      Then the following problem is in is in $\uAC^0[\NL]$:
      \compproblem{$G = \GP(\cL, I, \left(G_\cLelA\right)_{\cLelA \in \cL})$ given by $(\cL,I)$ and $G_\cLelA\in \mathcal{C}$ for $\cLelA \in \cL$ and a cyclically reduced, composite  and connected $w \in \Gamma^*$}{a cyclic normal form that is conjugate in $G$ to $w$ }
    \end{itemize}
    \noindent In both cases the output word starts with a letter that is maximal \wrt $\le_\cL$.
  \end{theorem}
  Note that due to \cref{lem:CPapprox}, in the situation of \cref{th:cyclic_nf_graph_product}, being conjugate in $G$ is equivalent to being conjugate in the trace monoid $M = M(\Gamma,I)$ or being related by a sequence of transpositions.

  \begin{proof}[Proof of \cref{th:cyclic_nf_graph_product}]
    Let $w \in \Gamma^*$ be the input word.
    Note that $w$ is already cyclically reduced, composite and connected.
    The cyclic normal form can be computed with the following algorithm:
    \begin{enumerate}%[label=\arabic*.]
      \item Compute the length-lexicographic normal form $\tilde{w} = \nf_G(w^\sigma)$ where, as before, $\sigma = \abs{\cL}$.
      \item Let $\tilde{w} = y d z$, where $d \in \Gamma_\cLelA$ is such that $\cLelA$ is maximal \wrt $\le_\cL$, $y \in (\Gamma \setminus \Gamma_\cLelA)^*$ and $z \in \Gamma^*$.
      Compute the cyclic permutation $d z y$.
      That is, we rotate the first occurrence of $d$ to the front.
      \item Compute the length-lexicographic normal form of $d z y$.
      We have $\nf_G(d z y) = u^\sigma$, where $u$ is a cyclic normal form conjugate to $w$.
    \end{enumerate}

    \noindent First, we show that our algorithm is correct, \ie $\nf_G(d z y)$ has the form $u^\sigma$ and $u$ is a cyclic normal form conjugate to $w$.

    For this we first prove that $d$ is the unique minimal letter of the trace represented by $dzy$.
    To get a contradiction, assume that $a \neq d$ is another minimal letter. In particular, $(a,d) \in I$, which implies $a <_{\cL} d$.
    If $a$ belongs to $z$, then we can write $z = az'$ and get $ydz =_M yadz'$ contradicting the fact that $y d z$ is a length-lexicographic normal form.
    Now assume that $a$ belongs to $y$, i.e., $y =_M a y'$ and $(a, dz) \in I$. Hence, in the trace monoid $M$, $ya$ is a prefix
    of $w^{2\sigma} =_M (ydz) (ay'dz)$.
    \hyperref[lemma-levi]{Levi's Lemma} yields the following diagram (where $yav =_M w^{2\sigma}$):
    \begin{center}
  \begin{tabular}{c"c|c|c|c|c|c|c|c|}\hline
  $v$  & $v_{1}$ & $v_{2}$ & $\cdots$  &  $v_{\sigma}$  & $v_{\sigma+1}$
  & $v_{\sigma+2}$ & $\cdots$ & $v_{2\sigma}$ \\ \hline
  $ya$  & $y_{1}$ & $y_{2}$  & $\cdots$  &  $y_{\sigma}$  & $y_{\sigma+1}$
  & $y_{\sigma+2}$ & $\cdots$ & $y_{2\sigma}$\\ \thickhline
    & $w$          & $w$ & $\cdots$  & $w$ & $w$
          & $w$ & $\cdots$ & $w$
  \end{tabular}
  \end{center}
  None of the $v_i$ can be $1$, since $d$ belongs to $w$ but not to $ya$.
  By \cref{lem:factor_shape} we obtain $y_j = 1$ for all $j \geq \sigma$.
  In particular, $ya$ is already a prefix of $w^\sigma =_M ydz$. But
  $a$ does not occur in $dz$ (we have $(a, dz) \in I$).
  Hence, since $ya$ contains more
  $a$'s than $y$, this is a contradiction.

    Next, let us show that $y$ is a prefix of $wy$ in the trace monoid $M$. To see this, observe that $ydzw =_M w^{\sigma+1} =_M wydz$.
    Since $\abs{y}_a \leq \abs{wy}_a$ for all $a \in \Gamma$,
    \cref{lem:lettercount} implies that $y$ is a prefix of $wy$.

    Thus, we can find some $\hat{u} \in M$ with $y\hat{u} =_M wy$. Observe that by the very definition $\hat{u}$ is conjugate to $w$. By \cref{lemm:conjugacy_equivalent}, $\hat{u}$ can be obtained from the cyclically
    reduced $w$ by a sequence of transpositions. Since the property of being
    cyclically reduced is preserved by transpositions, it follows that also
    $\hat{u}$ is cyclically reduced.

    Since $y \hat{u}^\sigma =_M w^\sigma y =_M y dz y$ and $M$ is cancellative,
    we get $d z y =_M \hat{u}^\sigma$. In particular, $d$ is the unique minimal letter of $\hat{u}$.
    Therefore, by \cref{lem:nf_k}, $\nf_G(d z y) = \nf_G(\hat{u}^{\sigma}) = \nf_G(\hat{u})^{\sigma}$ and $u = \nf_G(\hat{u})$ is a cyclic normal form.

    Second, we look at the complexity in the non-uniform case.
    Our algorithm requires solving the normal form problem twice.
    Since the input is cyclically reduced, computing the normal form only required computing a lexicographically smallest ordering, which by~\cite[Theorem 6.3.7]{kausch2017parallel} can be done in $\uTC^0 \subseteq \uAC^0(\WP(F_2))$.
    In addition, we need to compute a cyclic permutation, which can be done in $\uAC^0$.

    Third, we look at the complexity in the uniform case.
    By \mbox{\cite[Theorem 6.3.13]{kausch2017parallel}}, solving the normal form problem can be done in $\uTCz$ with oracle gates for $\NL$ (more precisely, that theorem states that it can be decided in \NL which of two letters comes first in the normal form~-- as sorting is in $\uTCz$, this statement follows). Note that $\uTCz[\NL] = \uACz[\NL]$.
    Finally, again, the cyclic permutation can be computed in $\uAC^0$.
  \end{proof}
Finally, we need the following lemma that requires that none of the base groups $G_\cLelA$ contains elements of order two. Hence,
$a \neq a^{-1}$ holds for all $a \in \Gamma$.

\begin{lemma} \label{lemma-p-conjugate-p-1}
Assume that $a \neq a^{-1}$ for all $a \in \Gamma$.
If $p \in M \setminus \{1\}$ is reduced, then $p$ and $p^{-1}$ are not conjugate (in $M$).
\end{lemma}

\begin{proof}
Assume that $p \in M \setminus \{1\}$ is conjugate to $p^{-1}$.
We show that $p$ is not reduced.
Let us first consider the case that $p =_M p^{-1}$.
We show by induction on $|p|$ that $p$ is not reduced.
Since $p \neq 1$, we can write $p =_M a s$ for $a \in \Gamma$
and $s \in M$. We obtain $a s =_M p =_M p^{-1} =_M s^{-1} a^{-1}$.
Since $a \neq a^{-1}$, we can write $s$ as $s =_M t a^{-1}$ and
obtain $a t a^{-1} =_M as =_M s^{-1} a^{-1} =_M a t^{-1} a$. Since $M$
is cancellative, we get $t =_M t^{-1}$. If $t=1$, then $p _M= aa^{-1}$ is not reduced and, if $t \neq 1$, then $t$ is not reduced by induction.

Now assume that $p \neq p^{-1}$. Since $p$ and $p^{-1}$ are
conjugate, \cite[Proposition 4.4.5]{DieRoz95} yields factorizations $p =_M q_1 q_2 \cdots q_k$ and $p^{-1} =_M q_k q_{k-1} \cdots q_1$ for traces $q_1, \ldots, q_k \in M \setminus \{1\}$.
Define $r =_M q_2 \cdots q_k$ and $s =_M q_2^{-1} \cdots q_k^{-1}$. We obtain
$q_1 r =_M q_1^{-1} s$. \hyperref[lemma-levi]{Levi's Lemma} yields factorizations
$q_1 =_M t u$ and  $q_1^{-1} =_M t v$ with $(u,v) \in I$.

We claim that $u=v=1$:
For every $\cLelA \in \cL$, the number of letters
from $\Gamma_\cLelA$ in $q_1$ and $q_1^{-1}$ is the same.
Hence, also the number of letters
from $\Gamma_\cLelA$ in $u$ and $v$ is the same. Since
$(u,v) \in I$, this is only possible if $u=v=1$.

We now get $q_1 = q_1^{-1}$. By the first paragraph of the proof,
this shows that $q_1$ (and hence $p$) is not reduced.
\end{proof}

  \subsubsection{A variant of Lyndon traces}

  A Lyndon trace $w$ from a trace monoid $M(\Sigma,I)$ is a connected primitive trace
  that is lexicographical minimal in its conjugacy class (where a trace $u$ is smaller
  than a trace $v$ if the lexicographic normal form of $u$ is length-lexicographically
  smaller than the lexicographic normal form of $v$), see e.g.~\cite[Section~4.4]{DieRoz95}. We will work with the following variant of Lyndon traces.
  As usual let $G = \GP(\cL, I, \left(G_\cLelA\right)_{\cLelA \in \cL})$ be a graph product and $M = M(\Gamma, \IGamma)$ the corresponding trace monoid.

  \begin{definition}
    \label{def:omega}
    Let $\Omega$ be the set of all $w  \in \Gamma^*$ satisfying the following properties:
    \begin{itemize}
      \item $w$ is a cyclic normal form (in particular, it is composite, cyclically reduced, and connected),
      \item $w$ represents a primitive element of $M$,
      \item $w$ is lexicographically minimal (\wrt $\le_\cL$) among its cyclic permutations and the cyclic permutations of a cyclic normal form conjugate (in $M$) to $w^{-1}$.
    \end{itemize}
  \end{definition}
  Note that the last point in \cref{def:omega} makes sense because if $w$ is composite, cyclically reduced, connected and primitive, then $w^{-1}$ as well
  as every $w'$ conjugate in $M$ to $w$ or $w^{-1}$ have the same properties (that
  a trace that is conjugate to a primitive trace is primitive too follows from
  Lemma~\ref{lemm:conjugacy_primitive}).
  Moreover, by \cref{th:cyclic_nf_graph_product} there is a
  cyclic normal form that is conjugate to $w^{-1}$ and by \cref{lem:cnf_conjugate} all cyclic normal forms conjugate to $w^{-1}$ are cyclic permutations of each other.

\begin{remark} \label{remark-Omega}
Note that if $u,v \in \Omega$ and the traces represented by $u$ and $v$
are conjugate in $M$ (or equivalently $G$), then $u = v$: By \cref{lem:cnf_conjugate}, $u$
and $v$ are cyclic permutations of each other, and the last point of \cref{def:omega} implies that $u=v$. If $u^{-1}$ and $v^{-1}$ are conjugate in $M$, then
also $u$ and $v$ are conjugate in $M$, and we obtain again $u=v$.
Finally, assume that $u$ and $v^{-1}$ are conjugate in $M$.
Fix a cyclic normal form $w$ that is conjugate to $v^{-1}$. Hence, $u$ and $w$
are conjugate in $M$ and by \cref{lem:cnf_conjugate}, $u$ is a cyclic permutation
of $w$. So, $u$ is a cyclic permutation of a cyclic normal form conjugate to $v^{-1}$,
which implies $v \leq_{\cL} u$. Since also $u^{-1}$ and $v$ are conjugate
in $M$, we also have $v \leq_{\cL} u$ and, thus, $u=v$.
\end{remark}

\begin{example}
Notice that the elements of $\Omega$ are not really Lyndon traces. This is because they are only minimal among their cyclic permutations but not minimal in the full conjugacy class. Indeed, consider the trace monoid $M(\Sigma,I)$ where $\Sigma = \{a,b,c\}$, $I = \{(a,c),(c,a)\}$, and $a < b < c$.\footnote{This trace monoid $M(\Sigma,I)$ would arise from the graph product $(\mathbb{Z}_2 \times \mathbb{Z}_2)*\mathbb{Z}_2$, where $a$, $b$, and $c$ are the generators of the three base groups.}
Then the trace $abc$ is lexicographically smallest in its conjugacy class. However, it is not a cyclic normal form, since the cyclic permutation $bca$ is not lexicographically minimal. The corresponding lexicographically smallest conjugate cyclic normal form would be $acb$.
\end{example}

\begin{remark}
 Notice that, when solving the uniform power word problem for graph products, it is important that in $\Omega$ we require lexicographical minimality only among cyclic normal forms conjugate to $w^{\pm 1}$. A straightforward generalization of the approach for free groups (see beginning of \cref{ch:conjugacy-in-graph-groups-and-graph-products}) would search for a lexicographically minimal element in the full conjugacy class of $w^{\pm 1}$. However,
this approach does not seem to be feasible because
 there might be exponentially many conjugate traces for a given trace. Here is an example:

 Let $M(\Sigma,I)$ where $\Sigma = \{a_1, \dots, a_n\}$ and $I = \set{(a_i,a_j)}{\abs{j-i} \geq 2}$ and $u=a_1\cdots a_n$. Then every permutation $a_{\pi(1)}\cdots a_{\pi(n)}$ is conjugate to $u$. In particular, there are $n!$ many conjugate traces to $u$.

 An interesting open question is the complexity of the following problem: given a trace monoid and a trace, find the lexicographically smallest conjugate trace.
 This problem can easily be seen to be in $\P^\NP$. However, it is totally unclear to us whether it can be actually solved in polynomial time~-- or whether its decision variant is \NP-complete.
\end{remark}

  The crucial property of $\Omega$ is that each $w \in \Omega$ is a unique representative for its conjugacy class and the conjugacy class of its inverse.
  Similar to the case of a free group (see \cref{lem:short_cancellation})
this fact leads us to the following theorem, which is central to solving the power word problem in graph products (see \cref{lem:gp-rewrite_bounds_exponents} below).
  As for \cref{lem:short_cancellation}, the intuition behind it is that, if there are two powers $p^x$ and $q^y$, where $p, q \in \Omega$ and $q \neq p$, then in $p^xq^y$ only a small number of letters can cancel out.
  Conversely, if a sufficiently large suffix of $p^x$ cancels with a prefix of $q^y$, then $p = q$. In the end this will allow us to decrease all the exponents of $p$ simultaneously as described in \cref{def:shortened}.

  \begin{theorem}
    \label{th:cancel_length_factor_gp}
    Let ${p}, {q} \in \Omega$, and ${x}, {y} \in \mathbb{Z}$.
    Moreover, let $u \in M$ (resp., $v \in M$)  be a factor of the trace represented by ${p}^{{x}}$ (resp., ${q}^{{y}}$) such that $uv =_G 1$.
    If $|u| = |v| > 2\sigma(|{p}| + |{q}|)$, then ${p} = {q}$.
  \end{theorem}

Note that in the above theorem $u$ and $v$ are reduced as they are factors of $p^x$ (resp.\ $q^y$).
For the proof of \cref{th:cancel_length_factor_gp} we apply the
Lemmata~\ref{lem:projection_eq_gp} and \ref{lem:power_gp}
        from \cref{sec:pojection-to-free-monoids} to the trace monoid $M(\Gamma, I)$ that corresponds to the graph product $G = \GP(\cL, I, \left(G_\cLelA\right)_{\cLelA \in \cL})$.
To do so, we use the cliques and projections as defined in \eqref{clique-covering-A} in \cref{sec:big_trace_monoid_prelims}, which we recall for convenience:
\begin{equation*} %\label{clique-covering-A}
 \{A_1, \ldots, A_k\} = \left\{ \Gamma_\cLelA \cup \Gamma_\cLelB \mid (\cLelA, \cLelB) \in D,  \cLelA \neq \cLelB\right\} \cup \{ \Gamma_\cLelA \mid \cLelA \text{ is isolated} \},
\end{equation*}
   where $\cLelA$ is isolated if  there is no $\cLelB \neq \cLelA$ with  $(\cLelA, \cLelB) \in D$.
 Now, $\pi_i: M(\Gamma, I) \to A_i^*$ denotes the canonical projection and $\Pi: M(\Gamma, I) \to A_1^*\times\cdots\times A_k^*$ with $\Pi(w) = (\pi_1(w), \dots, \pi_k(w))$ is an injective monoid morphism by \cref{lem:projection_eq_gp}.
We derive \cref{th:cancel_length_factor_gp} from the
  following lemma.

  \begin{lemma}
    \label{lem:projections_conjugate_gp}
    Let $p, q, v \in M(\Gamma, \IGamma)$, $x, y \in \mathbb{N}\setminus\{0\}$ such that $p$ and $q$ are primitive and connected, and $p^x$ and $q^y$  have the common factor $v$.
    If $p^2$ and $q^2$ are factors of $v$, then for all $i$ the projections $\pi_i(p)$ and $\pi_i(q)$ are conjugate as words.
  \end{lemma}

  \begin{proof}
  Note that $\alphabet(p) = \alphabet(v) = \alphabet(q)$.
    We define the set
    \begin{equation*}
         J_v = \{ i \mid \alphabet(A_i) \cap \alphabet(v) \neq \emptyset\} \text{.}
        %J_v = \{ i \mid \alphabet(A_i) \subseteq \alphabet(v) \} \text{.}
    \end{equation*}
    Note that $\pi_i(p) = \pi_i(q) = 1$ for all $i \in [1,k] \setminus J_v$.
    It therefore suffices to show that $\pi_i(p)$ and $\pi_i(q)$ are conjugate
    for all $i \in J_v$.

    For each $i \in J_v$ we write $\pi_i(p) = \tilde{p}_i^{s_i}$ and $\pi_i(q) = \tilde{q}_i^{r_i}$ where $\tilde{p}_i$ and $\tilde{q}_i \in A_i^*$ are primitive.
    As $v$ is a common factor of $p^x$ and $q^y$, its projection $\pi_i(v)$ is a common factor of $\pi_i(p^x) = \tilde{p}_i^{{}s_i x}$ and $\pi_i(q^y) = \tilde{q}_i^{{}r_i y}$.
    Thus, $\pi_i(v)$ has periods $|\tilde{p}_i|$ and $|\tilde{q}_i|$.
    Since $p^{2}$ is a factor of $v$, $\pi_i(p)^{2}$ is a factor of $\pi_i(v)$.
    This yields the lower bound $2|\tilde{p}_i|$, and by symmetry $2|\tilde{q}_i|$, on the length of $\pi_i(v)$.
    Combining those, we obtain
    \begin{equation*}
        |\pi_i(v)| \ge \max\{2|\tilde{p}_i|, 2|\tilde{q}_i|\} \ge |\tilde{p}_i| + |\tilde{q}_i| \ge |\tilde{p}_i| + |\tilde{q}_i| - 1\text{.}
    \end{equation*}
    By the theorem of Fine and Wilf~\cite{fine1965uniqueness}, $\gcd(|\tilde{p}_i|,|\tilde{q}_i|)$ is a  period of $\pi_i(v)$. As $\tilde{p}_i$ and  $\tilde{q}_i$ are primitive, it follows that $|\tilde{p}_i| = |\tilde{q}_i|$.
    As $p$ is a factor of $v$, in particular, $\tilde{p}_i$ is a factor of $\pi_i(v)$ and, thus, also of $\pi_i(q^y) = \tilde{q}_i^{r_i \cdot y}$.
    Hence, $\tilde{p}_i$ and $\tilde{q}_i$ are conjugate words for all $i \in J_v$.

    \newcommand{\multA}{\lambda}
    \newcommand{\multB}{\mu}

In order to show that $\pi_i(p)$ and $\pi_i(q)$ are conjugate for all $i \in J_v$, it
suffices to show that $s_i = r_i$ for all $i \in J_v$.
    Assume for a contradiction that for some $i \in J_v$ we have $s_i \neq r_i$.
    Then, there are $\multA, \multB \in \mathbb{N}\setminus\{0\}$ such that $\multA s_i = \multB r_i$.
    \Wlog let $\multB > 1$ and $\gcd\{\multA, \multB\} = 1$.
    Now $\multB$ divides $s_i$.
    Let
    \[ J = \{ j \in J_v \mid \multA s_j = \multB r_j \}.\]
 %   \begin{claim}
\medskip\noindent{\em Claim}: $J = J_v$.
%    \end{claim}
 %   \medskip\noindent
%{\em Proof of Claim1}:
\begin{proof}[Proof of the Claim:]

    Clearly $i \in J$. We want to show that $J = J_v$.
    This is the case if $\alphabet(v) = \{\cLelA\}$ for some
    isolated $\cLelA$ (then, also $J_v$ is a singleton).
    Otherwise, for all $i \in J_v$ the set $A_i$ is of the form
    $A_i = \Gamma_\cLelA \cup \Gamma_\cLelB$ with $(\cLelA,\cLelB) \in D$ and
    $\cLelA \neq \cLelB$.

    Assume now that $J \subsetneq J_v$ and
    let $j \in J_v \setminus J$.
    Let $\cLelA \in \alphabet(v) \cap \alphabet(A_i)$ and
    $\cLelB \in \alphabet(v) \cap \alphabet(A_j)$ ($\cLelA = \cLelB$ is possible).
    Since $p$ (and hence $v$) is connected, there must exist a (possibly empty) simple path
    $\cLelA = \cLelA_0, \cLelA_1, \ldots, \cLelA_{n-1}, \cLelA_n = \cLelB$
    in $(\alphabet(v), D)$. Each edge $(\cLelC,\cLelD) \in D$ along this
    path corresponds to an element $\Gamma_{\cLelC} \cup \Gamma_{\cLelD} = A_\ell$
    with $\ell \in J_v$ and $\cLelC, \cLelD \in \alphabet(v)$.
    Therefore, there are $i' \in J$ and $j' \in J_v \setminus J$
    with $\alphabet(A_{i'} \cap A_{j'}) = \{\cLelA\}$ for some $\cLelA \in \alphabet(v) = \alphabet(p) = \alphabet(q)$.

    For the further consideration, let us rename $i'$ (resp., $j'$) into $i$ (resp., $j$).
    We have $s_j |\tilde{p}_j|_\cLelA = |p|_\cLelA = s_i |\tilde{p}_i|_\cLelA$ (recall that $|p|_\cLelA = |p|_{\Gamma_\cLelA}$).
    Similarly, we have $r_j |\tilde{q}_j|_\cLelA = |q|_\cLelA = r_i|\tilde{q}_i|_\cLelA$, which is equivalent to $r_j |\tilde{p}_j|_\cLelA = r_i |\tilde{p}_i|_\cLelA$ (as $\tilde{p}_\ell$ and $\tilde{q}_\ell$ are conjugate for all $\ell \in J_v$).
    Since $i \in J$, we obtain
    \[\multA s_j |\tilde{p}_j|_\cLelA = \multA s_i |\tilde{p}_i|_\cLelA = \multB r_i|\tilde{p}_i|_\cLelA = \multB r_j |\tilde{p}_j|_\cLelA.\]
    Since  $|\tilde{p}_j|_\cLelA \neq 0$, we conclude $\multA s_j = \multB r_j$, i.e.,
    $j \in J$, which is a contradiction. Therefore, we get $J = J_v$, which concludes the proof of the claim.
\end{proof}

    Thus, every $s_i$ (for $i \in J_v$) is divisible by $\multB > 1$. By \cref{lem:power_gp} we can write $p =_M u^{\multB}$ for some trace $u$,
    contradicting $p$ being primitive. This concludes the proof of  the lemma.
  \end{proof}

  The proof idea for \Cref{th:cancel_length_factor_gp} is as follows:
  We use the length bound from \cref{lem:factor_shape} in order to show that the requirements of \cref{lem:projections_conjugate_gp} are satisfied. After applying that lemma, we show that $p$ and $q$ are conjugate using \cref{lem:projection_eq_gp}.  Then we conclude from the definition of $\Omega$ that $p=q$.

  \begin{proof}
    [Proof of \Cref{th:cancel_length_factor_gp}]
    We have $u =_G v^{-1}$ and hence $u =_M v^{-1}$ (since, $u$ and $v^{-1}$ are reduced). Thus,
    $v^{-1}$ is a factor of ${p}^x$ and therefore, $v$ is a factor of ${p}^{-x}$.
    Let $\Omega^\pm = \Omega \cup \Omega^{-1}$ be the extension of $\Omega$ that includes the inverse of each element.
    Let $\hat{x} = |{x}|$ and $ \hat{y} = |{y}|$.
    Then there are $\hat{p}, \hat{q} \in \Omega^\pm$ such that $\hat{p} \in \{{p}, {p}^{-1}\}$, $\hat{q} \in \{{q}, {q}^{-1}\}$ and $v$ is a common factor of $\hat{p}^{\hat{x}}$ and $\hat{q}^{\hat{y}}$.
    As $|v| > 2\sigma(|p| + |q|) \ge 2\sigma|p|$, by \cref{lem:factor_shape}, $v$ can be written as $v = u_1\cdots u_t \hat{p}^z v_s \cdots v_1$, where $z \ge 2$.
    Hence, $\hat{p}^2$ is a factor of $v$.
    By symmetry $\hat{q}^2$ is a factor of $v$.

    By \cref{lem:projections_conjugate_gp}, for all $i$ the projections $\pi_i(\hat{p})$ and $\pi_i(\hat{q})$ are conjugate words. In particular, for each $\cLelA\in\cL$ we have
    $\abs{\hat{p}}_\cLelA = \abs{\hat{q}}_\cLelA$. Thus, as $\hat{q}$ is a factor of $\hat{p}^{\hat{x}}$, it follows from \cref{lem:lettercount_conjugacy}  that  $\hat{p}$ is conjugate in $M$ to $\hat{q}$. Since $p,q \in \Omega$, this finally implies $p=q$; see
    \cref{remark-Omega}.
  \end{proof}

  \subsection{Main proofs for the power word problem in graph products}
  \label{ch:the-power-word-problem-in-graph-products}

  In this section we show our main results for graph products (according to the conference version \cite{StoberW22}): In order to solve the power word problem,
  we follow the outline of~\cite{LohreyW19} (which is for free groups).
  In particular, our proof also consists of three major steps:
  \begin{itemize}
    \item In a preprocessing step we replace all powers with powers of elements of $\Omega$ (\cref{sec:gp-preprocessing}).
    \item We define a symbolic rewriting system which we use to prove correctness (\cref{sec:gp-symbolic-rewriting-system}).
    \item We define the shortened word, replacing each exponent with a smaller one, bounded by a polynomial in the input (\cref{subsec:gp-the-shortened-word}).
  \end{itemize}
Finally, in \cref{sec:pwp_wrap_up}, we combine these steps for the solution of the power word problem.

 The main difference to \cite{LohreyW19} is that here we rely on \cref{th:cancel_length_factor_gp} instead of~\cite[Lemma 11]{LohreyW19} (see \cref{lem:short_cancellation}), an easy fact about words. This is because the combinatorics of traces/graph products is much more involved than of words/free groups.
 Furthermore, for free groups we did not have to bother with elements of order two, which led to the mistake in \cite{StoberW22,StoberW22arxivOld} (see \cref{rem:mistake}).
 Another major difference to the case of free groups is that we need the results for the simple power word problem considered in \cref{sec:simple_pwp}.
 Apart from that, all steps are the same (with some minor technical differences).

  \subsubsection{Preprocessing}
  \label{sec:gp-preprocessing}

  Let $G = \GP(\cL, I, \left(G_\cLelA\right)_{\cLelA \in \cL})$ be a graph product of f.g.~groups. As usual, $\sigma = |\cL|$.
  We define the alphabet $\tilde{\Gamma} = \Gamma \times \mathbb{Z}$, where $(v, z)$ represents the power $v^z$.
  Note that $\tilde{\Gamma}$ is the alphabet of the simple power word problem in $G$.
  During preprocessing, the input power word is transformed into the form
  \begin{equation} \label{eq-w-preproc}
       w = u_0 p_1^{x_1} u_1 \cdots p_n^{x_n} u_n, \qquad \text{ where $p_i \in \Omega$ and $u_i \in \tilde{\Gamma}^*$ for all $i$.}
  \end{equation}
  We denote the uniform word problem for graph products with base groups in $\mathcal{C}$ by $\UWP(\GP(\cC))$.
  For some further thoughts on how to encode the input, see \cref{sec:input-encoding-spwp}.
  The preprocessing consists of five steps:
  \begin{description}
    \item[Step 1: Cyclically reducing powers.]
    Cyclically reducing every $p_i$ can be done using the procedure from
    \cite[Lemma~7.3.4]{kausch2017parallel}. We also need to compute a trace $y_i$ such that
    $y_i^{-1} p_i y_i$ is cyclically reduced. It follows from the proof of
     \cite[Lemmata~7.3.2 and 7.3.3]{kausch2017parallel} that such $y_i$ can be obtained
    as a prefix of $p_i$. Let us quickly
     repeat the argument: assume that $p_i$ is already reduced.
    First one computes the longest prefix $t_i$ of $p_i$ such that
    $t_i^{-1}$ is also a suffix of $p_i$. Thus we can write $p_i =_M t_i p'_i t_i^{-1}$. The trace $p'_i$ is not necessarily cyclically reduced.
    But there are elements
    $a_1, \ldots, a_k, b_1, \ldots, b_k \in \Gamma$ such that
    $\alphabet(a_i) = \alphabet(b_i)$ for all $i$ and $(a_i,a_j) \in I$ for $i \neq j$
    (in particular, $k \leq \sigma$) such that $p'_i = a_1 \cdots a_n \tilde{p}_i b_1 \cdots b_n$
    and $\tilde{p}_i [a_1 b_1] \cdots [a_k b_k]$ is cyclically reduced.
    Let us define the prefix $y_i = t_i a_1 \cdots a_k$ of $p_i$.  Then we have
    $\tilde{p}_i =_G y_i^{-1} p_i y_i$ and $|y_i| \leq |p_i|$.
    We then replace the power $p_i^{x_i}$ with $y_i^{-1} \tilde{p}_i^{x_i} y_i$; moreover, $y_i^{-1}$ and $y_i$ can be merged with $u_{i-1}$ and $u_i$, respectively.
    Thus we can assume that for the next step the
    input again has the form $w = u_0 p_1^{x_1} u_1 \cdots p_n^{x_n} u_n$, but now all $p_i$ are cyclically reduced.

\smallskip
    \item[Step 2: Replacing powers with powers of connected elements.]
    We compute the connected components of $p_i$. More precisely, we compute $p_{i,1}$, \dots, $p_{i,k_i}$ such that each $p_{i, j}$ is connected, $p_i =_G p_{i,1}\cdots p_{i,k_i}$,  and $(p_{i,j}, p_{i,\ell}) \in I$ for $j \neq \ell$.
    Observe that $k_i \le |\cL|$.
    We replace the power $p_i^{x_i}$ with $p_{i,1}^{x_i}\cdots p_{i,k_i}^{x_i}$.

\smallskip
    \item[Step 3: Removing powers of a single letter.]
    We use the alphabet $\tilde{\Gamma} = \Gamma \times \mathbb{Z}$, where $(v, z)$ represents the power $v^z$.
    We replace each power $p_i^{x_i}$ where $p_i \in \Gamma_\cLelA$ for some $\cLelA \in \cL$ with the corresponding letter $p_i^{x_i} \in \tilde{\Gamma}$.
    Note that there is no real work to do in this step.
    What happens is that powers of a single letter will be ignored (\ie treated as if they were part of the $u_i$ from \eqref{eq-w-preproc}) in the remaining preprocessing steps and when computing the shortened word.
    As a consequence, in the remaining preprocessing steps and during the computation of the shortened word we may assume that we only have powers of composite words.
    At the end, powers of a single letter will be the only powers remaining in the shortened word, and therefore they are the reason for reducing to the simple power word problem.
    For the next step we still assume that the input has the shape $w = u_0 p_1^{x_1} u_1 \cdots p_n^{x_n} u_n$, however, from here on $u_i \in \tilde{\Gamma}^*$.

\smallskip
    \item[Step 4: Replace each letter with a normal form specific to the input.]
		Whereas the previous steps work on the level of the trace monoid, this step computes a normal form for the elements of $\Gamma$ itself.
    For each $i$ we write $p_i = a_{i, 1}\cdots a_{i, k_i}$, where $a_{i, j} \in {\Gamma}$.
		Recall that elements of $\Gamma$ are given as words over $\Sigma$ \ie{} the generators of the respective base groups.
    Let \[N = [a_{1, 1},a_{1,1}^{-1}, \ldots, a_{1, k_1},a_{1,k_1}^{-1},
    \ldots, a_{n,1},a_{n,1}^{-1}, \ldots, a_{n, k_n},a_{n,k_n}^{-1}]\] be the list of letters of $\Gamma$ (and their inverses) occurring in some power.
    For convenience, we write $N = [b_1, \dots, b_{m}]$, where $m = |N|$.
    We replace each $p_i$ with $\tilde{p}_i = \tilde{a}_{i, 1}\cdots \tilde{a}_{i, k_i}$, where $\tilde{a}_{i, j}$ is the first element in $N$ equivalent to $a_{i, j}$.
    Note that we need to solve the word problem in the base groups $G_\cLelA$ to compute this.
    After that transformation, any two $\tilde{a}_{i, j}$ and $\tilde{a}_{\ell, m}$ representing the same element of $\Gamma$ are equal as words over $\Sigma$ (and so bit-wise equal).
    Thus, the letters are in a normal form.
    Be aware that this normal form is dependent on the input of the power word problem in $G$, but that is not an issue for our application.
    Again, we assume the input for the next step to be $w = u_0 p_1^{x_1} u_1 \cdots p_n^{x_n} u_n$.

\smallskip
    \item[Step 5: Making each $p_i$ a primitive cyclic normal form.]
  %%%%
    The following is done for each $i \in [1, n]$.  Let us write $p^x$ for $p_i^{x_i}$.
    We apply the algorithm presented in the proof of \cref{th:cyclic_nf_graph_product} and compute
    a cyclic normal form $q$ that is conjugate to $p$ in $M$.
    We have $y p =_M q y$ for some $y$ with
    $|y| < \sigma \cdot |p|$. We replace $p^{x}$ with
    $y^{-1} q^{x} y$ and merge $y^{-1}$ with $u_{i-1}$ and $y$ with $u_i$.
    Note that also $q$ must be connected, composite and cyclically reduced
    as a trace.

    Observe that any cyclic normal form $u$ computed by the algorithm from the proof of \cref{th:cyclic_nf_graph_product} starts with a letter $d$ such that $u$ does not contain any letter $c$ with $d <_{\cL} c$.
		If such a cyclic normal form is not primitive in the trace monoid $M$, i.e.,
  $u =_M w^k$ with $k > 1$, then, by \cref{cor:cyclic_normal_form_primitive}, $u = \nf(w)^k$ (as words) and $\nf(w)$ is a cyclic normal form.
	Therefore, we compute a primitive word $r\in \Gamma^*$ such that $q = r^k$ for some $k \geq 1$ and replace $q^x$ by $r^{kx}$ (clearly, $q=r$ if $q$ is already
 primitive).
 Also $r$ must be connected, composite and cyclically reduced
    as a trace. Moreover, $r$ is a cyclic normal form as well, since each cyclic permutation of $r$ is a factor of $q = r^k$ if $k \geq 2$ and hence must be a
    length-lexicographic normal form.
    Again, we write the resulting power word as $u_0 p_1^{x_1} u_1 \cdots p_n^{x_n} u_n$ for the next step.

\smallskip
    \item[Step 6: Replace each power with a power of an element in $\Omega$.]
    Let $\Omega$ be as in \cref{def:omega}.
		The previous steps have already taken care of most properties of $\Omega$.
		In addition, Step 4 ensured that individual letters are in a normal form.
    The only requirement not yet fulfilled is that every $p_i$ must be
    minimal \wrt $\le_\cL$ among its cyclic permutations and the cyclic permutations of a cyclic normal form conjugate (in $M$) to $p_i^{-1}$.

    Using \cref{th:cyclic_nf_graph_product},
    we compute a cyclic normal form $p'_i$ that is conjugate (in $M$) to $p_i^{-1}$.
    It must be primitive too: if $p'_i =_M s^\ell$ for some $\ell \geq 1$ and $s \in M$, then $(s^{-1})^\ell$ is conjugate in $M$ to
    $p_i$, which implies by \cref{lemm:conjugacy_primitive} that $p_i =_M r^\ell$ for some $r\in M$. As $p$ is primitive, we have $\ell = 1$.
    Hence, $p'_i$ is primitive.

    Finally, we consider all cyclic permutations of $p_i$ and $p'_i$ and take
    the lexicographically smallest one; call it $\hat{p}_i$.
    Moreover, let $\iota \in \{-1,1\}$ be such that $\iota = 1$ if $\tilde{p}_i$ is conjugate to $p_i$ and $\iota = -1$ if $\tilde{p}_i$ is conjugate to $p_i^{-1}$.
    Then we can replace
    the power $p_i^{x_i}$ by $t_i^{-1} \hat{p}_i^{\iota x_i} t_i$ for an appropriate conjugator $t_i$ of length at most $|p_i|$ (we can choose $t_i$ as a prefix of $\hat{p_i}^{\iota\sgn(x_i)}$). Finally, $t_i^{-1}$ and $t_i$ can be merged
    with $u_{i-1}$ and $u_i$, respectively.
  \end{description}

  \begin{remark}\label{rem:mistake_preprocessing}
        Notice that it might happen that $p_i$ is conjugate to $p_i^{-1}$. In this case the outcome of Step 6 is not uniquely defined: $p_i^{x_i}$ could be either replaced by $\tilde{p}_i^{x_i}$ or $\tilde{p}_i^{-x_i}$ (plus some appropriate conjugators). For the preprocessing itself this ambiguity is not a problem; however, it prevents \cref{lem:gp-rewrite_bounds_exponents} below from being true. Therefore, in the later steps of our proof, we will require that  $a\neq a^{-1}$ for all $a \in \Gamma$, which, by \cref{lemma-p-conjugate-p-1}, implies that $p_i$ cannot be conjugate to $p_i^{-1}$.
  \end{remark}

  \begin{lemma}
    \label{lem:gp-preprocessing}
    The preprocessing can be reduced to the word problem; more precisely:
    \begin{itemize}
      \item Let $G = \GP(\cL, I, \left(G_\cLelA\right)_{\cLelA \in \cL})$ be a fixed graph product of f.g.~groups.
      Then computing the preprocessing is in $\uAC^0(\WP(G), \WP(F_2))$.
      \item Let \(\mathcal{C}\) be a non-trivial class of f.g.~groups.
      Given $(\cL,I)$, $G_\cLelA\in \mathcal{C}$ for $\cLelA \in \cL$ and an element $w$ of the graph product
      $G = \GP(\cL, I, \left(G_\cLelA\right)_{\cLelA \in \cL})$, the preprocessing can be done in $\uAC^0[\NL](\UWP(\GP(\cC))$.
    \end{itemize}
  \end{lemma}

\begin{proof}
    We look at the complexity of the individual steps of the preprocessing.
    For this proof we split step 5 into two parts: a) computing a cyclic normal form and b) making it primitive.
    \begin{center}
        \begin{tabular}{lcc}
            \toprule
            Step & non-uniform & uniform \\
            \midrule
            1. making $p_i$ cyclically reduced & $\uACz(\WP(G))$ & $\uACz(\UWP(\GP(\cC)))$\\
            2. making $p_i$ connected & $\uACz$ & $\uACz[\NL]$ \\
            3. powers of single letters & $\uACz$ & $\uACz$ \\
            4. normal form of letters & $\uACz(\{\WP(G_{\cLelA}) \mid \cLelA \in \cL \})$ & $\uACz(\UWP(\cC))$ \\
            5a. making $p_i$ cyclic normal forms & $\uACz(\WP(F_2))$ & $\uACz[\NL]$ \\
            5b. makig $p_i$ primitive & $\uACz$ & $\uACz$ \\
            6. bringing $p_i$ to $\Omega$ & $\uACz(\WP(F_2))$ & $\uACz[\NL]$ \\\bottomrule
        \end{tabular}
    \end{center}

\begin{description}
    \item[Step 1.] By~\cite[Lemma 7.3.4]{kausch2017parallel}, the cyclically reduced conjugate trace for a $p_i$ can be computed in $\uAC^0$ with oracle gates for the word problem in $G$ in the non-uniform case and in $\uAC^0$ with oracle gates for $\UWP(\GP(\cC))$ (the uniform word problem for graph products with base groups in $\mathcal{C}$) in the uniform case.
    Also the conjugating element (called $y_i$ in Step 1 above) can be computed
    within the same bound.
%    \markus{Letzteres steht eigentlich nicht in \cite{kausch2017parallel}, sollte aber passen.}

    \item[Step 2.] To compute the connected components of a power $p_i^{x_i}$ ($i \in [1,n]$), let us define $\cL_i = \alphabet(p_i)$ and
    the symmetric predicate $\operatorname{con}_i(\cLelA, \cLelB)$ for $\cLelA, \cLelB \in \cL$, which is true if and only if there is a path from $\cLelA$ to $\cLelB$ in
    the dependence graph $(\cL_i, (\cL_i \times \cL_i) \setminus I)$.
    If $\cLelA$ or $\cLelB$ does not belong to $\cL_i$, then $\operatorname{con}_{i}(\cLelA, \cLelB)$ is false.
    Note that if there is a path from $\cLelA$ to $\cLelB$, then there is a path of length at most $\sigma - 1$.
    Moreover, there is a path of length exactly $\sigma - 1$, because the complement of $I$ is reflexive.

    Therefore, in the non-uniform case the following formula is equivalent to $\operatorname{con}_{i}(\cLelA, \cLelB)$:
    \begin{align*}
        \exists \cLelC_1, \dots, \cLelC_{\sigma} \in \cL_i : \cLelC_1 = \cLelA \land \cLelC_{\sigma} = \cLelB \land \bigwedge_{j= 1}^{\sigma - 1} (\cLelC_j, \cLelC_{j+1}) \notin I .
    \end{align*}

    In the uniform case computing the predicate $\operatorname{con}_{i}$ requires solving the undirected path connectivity problem, which is in $\NL$.
    Furthermore, we define the predicate $\operatorname{smallest}_{i}(\cLelA)$, which for $\cLelA \in \cL$ is true if and only if $\cLelA \in \cL_i$ is the smallest member of $\cL$ in the connected component of $(\cL_i, (\cL_i \times \cL_i) \setminus I)$.
    The following formula is equivalent to $\operatorname{smallest}_{i}(\cLelA)$:
    \begin{align*}
        \cLelA \in \cL_i \land \forall \cLelB \in \cL : \operatorname{con}_{i}(\cLelA, \cLelB) \to \cLelA \leq_{\cL} \cLelB .
    \end{align*}
    We define the projection
    $\pi_{i, \cLelA} : \Gamma^* \to \Gamma^*$ for $i \in [1,n]$ and  $\cLelA \in \cL$ by
    \begin{align*}
        \pi_{i, \cLelA}(a) = \left\{ \begin{matrix*}[l]
          a \quad & \quad\text{if } \operatorname{con}_{i}(\cLelA, \alphabet(a)) \land \operatorname{smallest}_{i}(\cLelA)\text{,}\\
    1 & \quad\text{otherwise.}
        \end{matrix*}
        \right.
    \end{align*}
    Observe that $p_i =_G \prod_{\cLelA \in \cL} \pi_{i,\cLelA}(p_i)$
and each $\pi_{i, \cLelA}(p_i)$ is connected.
    \item[Step 3.] Identifying powers of single letters is obviously in $\uACz$.
    This step does not actually replace them, instead they will be ignored during the remaining preprocessing steps and when computing the shortened word.

    \item[Step 4.] Recall that $N = [b_1, \dots, b_{m}]$.
   To compute our normal form, we define the mapping $f : [1,m] \to [1,m]$ by
   \[ f(i) = \min\{ j \in [1,i] \mid \alphabet(b_j) = \alphabet(b_i) \land
      b_i =_{G_{\alphabet(b_i)}} b_j \}
   \]
    and the mapping $\text{nfletter} : \{ b_1, \dots, b_{m} \} \to \{ b_1, \dots, b_{m} \}$ by $\text{nfletter}(b_i) = b_{f(i)}$.
    Then $f$ and hence $\text{nfletter}$ can be computed in $\uACz$ with oracle gates for the word problems in the base groups (resp., the uniform word problem for the class $\mathcal{C}$ in the uniform case).

    \item[Step 5a.] By \Cref{th:cyclic_nf_graph_product}, we can compute a cyclic normal form in $\uAC^0$ with oracle gates for the word problem in $F_2$ in the non-uniform case and in $\uAC^0$ with oracle gates for $\NL$ in the uniform case.

    \item[Step 5b.] Checking for periods in words and replacing each power with a primitive factor is obviously in $\uAC^0$ (recall that we encode every element $\Gamma$ using the same number of bits).

    \item[Step 6.] A cyclic normal $p'_i$ form conjugate to $p_i^{-1}$ can be computed as in step 5a.
   Computing all cyclic permutations of $p_i$ and $p'_i$ and selecting the lexicographically smallest one is obviously in $\uAC^0$.
\end{description}

    From the complexities of the individual steps we conclude that the preprocessing can be done in $\uAC^0$ using oracle gates for the word problem in $G$ and $F_2$ in the non-uniform case and in $\uAC^0$ using oracle gates for $\UWP(\GP(\cC))$ and $\NL$ in the uniform case.
\end{proof}

  \subsubsection{A symbolic rewriting system} \label{sec:gp-symbolic-rewriting-system}
  \newcommand{\bigRS}{R}
  \newcommand{\DelEl}[3]{(#1,#2,#3)}

  We continue with a graph product of f.g.~groups $G = \GP(\cL, I, \left(G_\cLelA\right)_{\cLelA \in \cL})$.
  Recall the trace rewriting system  $\traceRS$ from~\eqref{rewrite-system-T} in \cref{sec:graph_prod_prelims}. As before, let $\sigma = |\cL|$.
From now on, we assume that $a \neq a^{-1}$ for all $a \in \Gamma$. Therefore, by \cref{lemma-p-conjugate-p-1}, if $p \in M \setminus \{1\}$ is reduced, then $p$ and $p^{-1}$ are not conjugate.

  For $x \in \mathbb{Z} \setminus \{0\}$ we denote by $\sgn x \in \{-1,1\}$ the sign of $x$. Moreover, let $\sgn 0 = 0$.
  For every $p \in \Omega$, we define the alphabet
  \begin{alignat*}{2}
    &\Delta_p &&= \left\{\DelEl{\beta}{p^x}{\alpha} \,\middle\vert\, \begin{matrix*}[l]
                                                             x \in \mathbb{N},\\
                                                             \alpha \text{ is a prefix of } p^{\sigma}\text{ and } p \text{ is no prefix of } \alpha,\\
                                                             \beta \text{ is a suffix of } p^{\sigma} \text{ and } p \text{ is no suffix of } \beta
    \end{matrix*}
    \right\}\text{.}
  \end{alignat*}
Note that a triple $\DelEl{\beta}{p^x}{\alpha}$ is viewed as single letter.
For $\DelEl{\beta}{p^x}{\alpha} \in \Delta_p$ we define $\DelEl{\beta}{p^x}{\alpha}^{-1} =
\DelEl{\alpha^{-1}}{p^{-x}}{\beta^{-1}}$. We write $\Delta_p^{-1}$ for the set of
all $\DelEl{\beta}{p^x}{\alpha}^{-1}$ with $\DelEl{\beta}{p^x}{\alpha} \in \Delta_p$.
Finally, we define the alphabet
  $\Delta = \Delta' \cup \Gamma$, where
 $\Delta' = \bigcup_{p \in \Omega} \Delta_p \cup \Delta_p^{-1}$.
 Notice that $\beta p^x\alpha$ is reduced for each $\DelEl{\beta}{p^x}{\alpha} \in \Delta'$, since every $p \in \Omega$ is cyclically reduced, connected and composite.

Note that when we talk about the length of a word $w \in \Delta^*$, every
occurrence of a
triple $\DelEl{\beta}{p^x}{\alpha} \in \Delta'$ in $w$ contributes by one to
the length of $w$. To emphasize this, we write $\abs{w}_\Delta$ for the length of
a word $w \in \Delta^*$. Moreover, $\abs{w}_\Gamma$ (resp., $\abs{w}_{\Delta'}$)
denotes the number of occurrences of symbols from $\Gamma$ (resp., $\Delta'$) in
$w$. In particular, $\abs{w}_\Delta = \abs{w}_\Gamma + \abs{w}_{\Delta'}$ holds.
Note that for $\abs{w}_\Gamma$ the symbols from $\Gamma$ that appear within
a triple $\DelEl{\beta}{p^x}{\alpha} \in \Delta'$ do not contribute.
For instance, if $w = ab (\beta, p^x, \alpha) c (\delta, q^y, \gamma)$ with $a,b,c \in \Gamma$, then $\abs{w}_\Delta = 5$, $\abs{w}_\Gamma = 3$ and
$\abs{w}_{\Delta'} = 2$.

  \begin{lemma}
    \label{lem:gp_alpha_beta_length}
    For $\DelEl{\beta}{p^x}{\alpha} \in \Delta'$ we have $|\alpha| < (\sigma - 1)|p|$ and $|\beta| < (\sigma - 1)|p|$.
  \end{lemma}
  \begin{proof}
    By \Cref{lem:factor_shape} we can write $\alpha = p^{k}u_1\cdots u_s$ with $s < \sigma$ where each $u_i$ is a proper prefix of $p$.
    As $p$ is not a prefix of $\alpha$, we have $k=0$.
    Regarding the length of $\alpha$, we obtain $|\alpha| = \sum_{i=1}^{s} |u_i| < \sum_{i=1}^{s} |p| = s|p| \le (\sigma - 1)|p|$.
    The bound on the length of $\beta$ follows by symmetry.
  \end{proof}

 \begin{lemma} \label{lemma-min-max-Delta'}
 Let $\DelEl{\beta}{p^x}{\alpha} \in \Delta'$.
 \begin{itemize}
     \item If $a$ is a minimal letter of
 $\beta p^x  \in M(\Gamma,I)$, then there are
 $\beta' \in M(\Gamma,I)$ and $d \in \{0, \sgn x\}$ with
 $\DelEl{\beta'}{p^{x-d}}{\alpha} \in \Delta'$ and
 $\beta p^x =_M a \beta' p^{x-d}$.
 \item If $a$ is a maximal letter of
 $p^x \alpha \in M(\Gamma,I)$, then there are
 $\alpha' \in M(\Gamma,I)$ and $d \in \{0, \sgn x\}$ with
 $\DelEl{\beta}{p^{x-d}}{\alpha'} \in \Delta'$ and
 $p^x \alpha =_M p^{x-d} \alpha' a$.
 \end{itemize}
 \end{lemma}

 \begin{proof}
  We only prove the first statement, the second statement can be shown in the same way. Moreover, assume that $x \ge 0$, the case $x \le 0$ is analogous.
  So, assume that $a$ is a minimal letter of the trace
 $\beta p^x$. The case that $a$ is a minimal letter of $\beta$, i.e.,
 $\beta =_M a \beta'$ for some $\beta'$, is clear. Otherwise, $x > 0$ and $a$ must be a minimal letter of $p$ with $(a,\beta) \in I$. Let $\gamma \in M(\Gamma,I)$ such that $p =_M a \gamma$. We obtain
 $\beta p^x =_M a \beta \gamma p^{x-1}$.
 It remains to show that $\DelEl{\beta \gamma}{p^{x-1}}{\alpha} \in \Delta'$, i.e., that $\beta \gamma$ is a suffix of $p^{\sigma}$ and $p$ is not a suffix
 of $\beta \gamma$.
 We have $p^\sigma = u \beta$ for some $u \in M$. Moreover, $p$ is not a suffix
 of $\beta$. The first statement of \cref{lem:factor_shape} implies that $\beta$ is a suffix of $p^{\sigma-1}$.
 Hence, $\beta \gamma$ is a suffix of $p^\sigma$.
 As $(a,\beta) \in I$, we have $\abs{\beta}_a = 0$. Therefore, $\abs{\beta\gamma}_a = \abs{\gamma}_a = \abs{p}_a-1$. Hence, $p$ cannot be a suffix of $\beta\gamma$.
 \end{proof}

  We define the projection $\pi:\Delta^* \to M(\Gamma,I)$ as the unique homomorphism with $\pi(a) = a$ for $a \in \Gamma$ and $\pi(\DelEl{\beta}{p^x}{\alpha}) = \beta p^x \alpha$ for $\DelEl{\beta}{p^x}{\alpha} \in \Delta'$.
    Moreover, for $a=\DelEl{\beta}{p^x}{\alpha}\in \Delta'$ we define $\alphabet(a) = \alphabet(p)\sse\cL$ (and $\alphabet(a)$ as defined before for $a\in \Gamma$).
Now, we can define an independence relation $I_\Delta$ on $\Delta$ by setting $(t,u) \in I_\Delta$ if and only if $\alphabet(t)\times \alphabet(u) \sse I$. Notice that, if $t, u \in \Delta$ are not of the form $\DelEl{\beta}{p^x}{\alpha}$ with $x=0$, then $(t,u) \in I_\Delta$ iff $(\pi(u),\pi(t)) \in I$; for elements of the form $\DelEl{\beta}{p^0}{\alpha}$, however, this is not an equivalence.

Let us consider the corresponding trace monoid $M(\Delta, I_\Delta)$: it contains $M(\Gamma,I)$ and $\pi$ defines a surjective homomorphism $M(\Delta, I_\Delta) \to M(\Gamma,I)$, which we denote by the same letter $\pi$.
We define a  trace rewriting system $\bigRS$ over $M(\Delta, I_\Delta)$ by the rules given in Table~\ref{tab-rules}.

\newcounter{rulecounter}
\newcommand*{\rulelabel}[1]{\refstepcounter{rulecounter}\therulecounter\label{#1}}
\newcommand*{\ruleref}[1]{\eqref{#1}}

\newcounter{condcounter}
\newcommand*{\condlabel}[1]{\refstepcounter{condcounter}\thecondcounter\label{#1}}
\newcommand*{\condref}[1]{\eqref{#1}}

\newcommand{\abExp}{f}

\begin{table}[t]
\begin{center}
\renewcommand{\arraystretch}{1.5}
\normalsize
\begin{tabular}{r|l}
rule (\rulelabel{eq:gp-rewrite_powereq}) & $\DelEl{\beta}{p^x}{\alpha} \,  \DelEl{\delta }{p^y}{\gamma} \to \DelEl{\beta}{ p^{x+y+\abExp}}{\gamma} $ \\ \hline
\multirow{2}{*}{condition (\condlabel{cond:gp-rewrite_powereq})} & \multirow{2}{*}{$\alpha \delta \rewrite{*}{\traceRS} p^\abExp \text{ for some $\abExp \in \mathbb{Z}$ } $} \\
& \\ \thickhline
rule (\rulelabel{eq:gp-rewrite_powereq_ncancel}) & $\DelEl{\beta}{p^x}{\alpha}\, \u \, \DelEl{\delta }{p^y}{\gamma} \to \DelEl{\beta}{p^{x-d}}{\alpha'} \, v\, \u \, \DelEl{\delta'}{p^{y-e} }{\gamma}$ \\ \hline
\multirow{2}{*}{condition (\condlabel{cond:gp-rewrite_powereq_ncancel})} &
$\bigl(\text{if } \exists z \in \mathbb{Z} : ~\alpha \delta \rewrite{*}{\traceRS} p^z \text{, then } (p,\pi(u)) \notin I\bigr),\ \beta p^x\alpha \pi(u) \in \IRR(\traceRS),$  \\
& $\pi(u) \delta p^y \gamma \in \IRR(\traceRS)$, $p^{x} \alpha \pi(u)\delta p^y \rewrite{+}{\traceRS} p^{x-d} \alpha' v\pi(u) \delta' p^{y-e} \in \IRR(\traceRS)$
\\ \thickhline
rule (\rulelabel{eq:gp-rewrite_powerneq}) & $\DelEl{\beta}{p^x}{\alpha} \, u \,\DelEl{\delta }{q^y}{\gamma} \to \DelEl{\beta}{p^{x-d}}{\alpha'} \, v \,u \, \DelEl{\delta'}{q^{y-e}}{\gamma}$ \\ \hline
\multirow{2}{*}{condition (\condlabel{cond:gp-rewrite_powerneq})} & $p \neq q ,\ \beta p^x\alpha \pi(u) \in \IRR(\traceRS),\ \pi(u) \delta p^y \gamma \in \IRR(\traceRS)\text{ and } $ \\
& $p^{x}\, \alpha\, \pi(u)\, \delta \, q^y \rewrite{+}{\traceRS} p^{x-d} \,\alpha'\, v \,\pi(u)\, \delta' \,q^{y-e} \in \IRR(\traceRS)$\\ \thickhline
rule (\rulelabel{eq:gp-rewrite_powertoplain}) &  $\DelEl{\beta}{p^0}{\alpha} \to \beta\alpha$  \\ \hline
\multirow{2}{*}{condition (\condlabel{cond:gp-rewrite_powertoplain})} &  \multirow{2}{*}{none
%$\beta \alpha \rewrite{*}{\traceRS} r \in \IRR(\traceRS) \text{ and } x = 0$
} \\
& \\ \thickhline
rule (\rulelabel{eq:gp-rewrite_plainpower}) &  $a  \DelEl{\beta}{p^x}{\alpha} \to a'  \DelEl{\beta'}{p^{x-d}}{\alpha}$ \\ \hline
\multirow{2}{*}{condition (\condlabel{cond:gp-rewrite_plainpower})} & \multirow{2}{*}{$a \, \beta p^x \rewrite{}{\traceRS} a' \, \beta' p^{x-d} \in \IRR(\traceRS)$} \\
& \\ \thickhline
rule (\rulelabel{eq:gp-rewrite_powerplain}) & $\DelEl{\beta}{p^x}{\alpha}  b \to \DelEl{\beta}{p^{x-d}}{\alpha'}  a'$ \\ \hline
\multirow{2}{*}{condition (\condlabel{cond:gp-rewrite_powerplain})} & \multirow{2}{*}{$ p^x \alpha \, b \rewrite{}{\traceRS} p^{x-d} \alpha' \, a' \in \IRR(\traceRS)$} \\
& \\ \thickhline
rule (\rulelabel{eq:gp-rewrite_plain}) & $a  b \to [ab]  $ \\ \hline
\multirow{2}{*}{condition (\condlabel{cond:gp-rewrite_plain})} & \multirow{2}{*}{$ \alphabet(a) = \alphabet(b)$} \\
& \\ \thickhline
\end{tabular}
\end{center}
\caption{\label{tab-rules}\normalsize The rules for the rewriting system $\bigRS$. All triples
$\DelEl{\beta}{p^x}{\alpha}$, $\DelEl{\delta }{p^y}{\gamma}$, etc.~belong
to $\Delta'$, $a, b \in \Gamma$, $a' \in \Gamma \cup \{ 1 \}$, $d \in \llbracket x \rrbracket$, $e \in \llbracket y \rrbracket$,
$\u \in M(\Delta, I_\Delta)$, and $v \in M(\Gamma, I_\Delta)$ is an $I$-clique with $(\pi(u),v) \in I$.}
\end{table}

\begin{remark} \label{remark-rule-2-3}
To understand rules \eqref{eq:gp-rewrite_powereq_ncancel} and \eqref{eq:gp-rewrite_powerneq}  one has to apply
\cref{lem:reducing-pqr} to the traces $p^x \alpha$, $\pi(u)$, and $q^y \delta$ (where $p=q$ might hold). Note that $p^x \alpha \pi(u), \pi(u)\delta q^y \in \IRR(\traceRS)$. If $p^x \alpha \pi(u) \delta q^y \notin \IRR(\traceRS)$
then \cref{lem:reducing-pqr} tells us there must exist a prefix
$s$ of $p^x\alpha$, a suffix $t$ of $\delta q^y$, and an $I$-clique
$v$ such that
\[ p^x\alpha \, \pi(u) \,\delta q^y \ \rewrite{*}{\traceRS}\
   s \, v \, \pi(u) \, t  \]
and $(\pi(u),v) \in I$. Moreover, by \cref{lemma-min-max-Delta'}
we can write $s$ and $t$ as $p^{x'} \alpha'$ and $\delta' q^{y'}$, respectively,
where $x' \in \llbracket x \rrbracket$, $y' \in \llbracket y \rrbracket$,
 and $\DelEl{\beta}{p^{x'}}{\alpha'},
 \DelEl{\delta'}{q^{y'}}{\gamma} \in \Delta'$.
\end{remark}

\begin{remark} \label{remark-rule-2-3'}
In rules  \eqref{eq:gp-rewrite_powereq_ncancel} and \eqref{eq:gp-rewrite_powerneq}
we  allow $u$ to contain a minimal letter $a$ such that $(p,a) \in I$.
Similarly, $u$ may contain a maximal letter $b$ such that $(b,p) \in I$ in
rule \eqref{eq:gp-rewrite_powereq_ncancel} and $(b, q) \in I$ in rule \eqref{eq:gp-rewrite_powerneq}. On the other hand, we could forbid this situation
and require that $(\beta, p^x,\alpha)$ is the only minimal letter of the left-hand
sides of rules \eqref{eq:gp-rewrite_powereq_ncancel} and \eqref{eq:gp-rewrite_powerneq} (and similarly for the maximal letters).
This would not change the arguments in our further considerations.
\end{remark}
The following facts about $\bigRS$ are crucial.

  \begin{lemma}
    \label{lem:gp-correctenss_T}
    For $u, v \in M(\Delta, I_\Delta)$ we have
    \begin{enumerate}%[label=(\roman*)]
      \item $\pi(\IRR(\bigRS)) \subseteq \IRR(\traceRS)$, \label{gp-correctenss_T-stmt1}
      \item $u \rewrite{*}{\bigRS} v$ implies $\pi(u) \rewrite{*}{\traceRS} \pi(v)$,\label{gp-correctenss_T-stmt2}
      \item $R$ is terminating,\label{gp-correctenss_T-terminating}
      \item $\pi(u) =_G 1$ if and only if $u \rewrite{*}{\bigRS} 1$.\label{gp-correctenss_T-stmt3}
    \end{enumerate}
  \end{lemma}

  \begin{proof}
  We start with statement \ref{gp-correctenss_T-stmt1}.
    Assume we have an element $t \in \IRR(\bigRS)$ with $\pi(t) \notin \IRR(\traceRS)$. So, none of the rules of $\bigRS$ can be applied to
    $t$. For rule \eqref{eq:gp-rewrite_powertoplain} this implies that $x \neq 0$ for every
    $\DelEl{\beta}{p^x}{\alpha} \in \Delta'$ that occurs in $t$.
    Since $\pi(t) \notin \IRR(\traceRS)$,
    there is a factor $ab$ in $\pi(t)$ with $a, b \in \Gamma$ and $\alphabet(a) = \alphabet(b)$. We have
    $ab \rewrite{}{\traceRS} [ab]$ (we may have $[ab]=1$). Let $t = t_1 t_2 \cdots t_m$
    with $t_i \in \Delta$. As every $\pi(t_i) \in M(\Gamma,I)$ is reduced
    with respect to $\traceRS$, $a$ and $b$ must be located in different factors $\pi(t_i)$.
    Assume that $a$ belongs to $\pi(t_i)$ and $b$ belongs to $\pi(t_j)$ for some
    $j > i$ (note that $j< i$ is not possible since $(a,b) \in D$).
    Let $u = t_{i+1}\cdots t_{j-1}$ (which might be empty).
    It follows that $a$ is a maximal letter of  $\pi(t_i)$, $b$ is a minimal letter of $\pi(t_j)$, $(\pi(u),a)\in I$ and $(\pi(u),b)\in I$. Moreover, we can assume that
$i$ and $j$ are chosen such that $j-i$ is minimal, which implies
that $\pi(u t_j), \pi(t_i u) \in \IRR(\traceRS)$.

If $t_i$ and $t_j$ are both in $\Gamma$, then rule \eqref{eq:gp-rewrite_plain} can be applied, which is a contradiction.
If $t_i = \DelEl{\beta}{p^x}{\alpha} \in \Delta'$ and
$t_j = b \in \Gamma$, then $a$ must be a maximal letter of $p^x \alpha$
(since $x \neq 0$). Hence, by \cref{lemma-min-max-Delta'}, rule
\eqref{eq:gp-rewrite_powerplain} can be applied. Similarly,
if $t_i \in \Gamma$ and $t_j \in \Delta'$ then rule
\eqref{eq:gp-rewrite_plainpower} can be applied. In both cases
we obtain a contradiction.
Finally, if $t_i, t_j \in \Delta'$, the situation is a bit more subtle.
%%%%
Assume that $t_i = \DelEl{\beta}{p^x}{\alpha}$ and $t_j = \DelEl{\delta }{q^y}{\gamma}$ for $x \neq 0 \neq y$.
Clearly,  $a$ is a maximal letter of $p^x\alpha$, and $b$ is a
minimal letter of $\delta q^y$. Moreover, $p^x \alpha \pi(u), \pi(u) \delta q^y \in \IRR(\traceRS)$. Our consideration from \cref{remark-rule-2-3}
shows that
\[ p^x\alpha \, \pi(u) \,\delta q^y \rewrite{+}{\traceRS} p^{x'} \alpha'
    \, v \pi(u) \, \delta' q^{y'} \in \IRR(\traceRS) \]
 for some $I$-clique $v$ with $(\pi(u),v) \in I$,
 $x' \in \llbracket x \rrbracket$, $y' \in \llbracket y \rrbracket$,
 and $\alpha', \delta'$ with
 $\DelEl{\beta}{p^{x'}}{\alpha'},
 \DelEl{\delta'}{q^{y'}}{\gamma} \in \Delta'$.
If $p \neq q$ then rule \eqref{eq:gp-rewrite_powerneq} can be applied to $t$.
On the other hand, if $p = q$ then, rule \eqref{eq:gp-rewrite_powereq}
or \eqref{eq:gp-rewrite_powereq_ncancel} can be applied.
    Altogether, it follows that one of the rules of $\bigRS$ can be applied contradicting $t \in \IRR(\bigRS)$.
    Thus $\pi(\IRR(\bigRS)) \subseteq \IRR(\traceRS)$.

    For statement \ref{gp-correctenss_T-stmt2} observe that the rules of $\bigRS$ only allow such reductions that are also allowed in $T$. To see statement \ref{gp-correctenss_T-terminating}, consider a rewriting step
  $u \rewrite{}{\bigRS} v$. This means that we have $u \rewrite{*}{\traceRS} v$ and either $|u|_{\Delta'} > |v|_{\Delta'}$ (for rules  \eqref{eq:gp-rewrite_powereq} and \eqref{eq:gp-rewrite_powertoplain}) or
   $u \rewrite{+}{\traceRS} v$ and $|u|_{\Delta'} = |v|_{\Delta'}$ (for the other rules).  Hence, as $T$ is terminating, so is $R$ (indeed, in \cref{cor:gp-steps_to_termination} below, we give explicit bounds on the number of possible rewriting steps).

    Statement \ref{gp-correctenss_T-stmt3} follows from statements \ref{gp-correctenss_T-stmt1} and \ref{gp-correctenss_T-stmt2}.
    If $u \rewrite{*}{\bigRS} 1$, then $\pi(u) \rewrite{*}{\traceRS} 1$ by statement \ref{gp-correctenss_T-stmt2}, i.e., $\pi(v) =_G 1$.
    On the other hand, if $u \rewrite{*}{\bigRS} 1$ does not hold,
    then, since $\bigRS$ is terminating, there exists
    $v \in \IRR(\bigRS)$ with
$u \rewrite{*}{\bigRS} v\neq 1$. We obtain $\pi(u) \rewrite{*}{\traceRS} \pi(v) \neq 1$ by statement \ref{gp-correctenss_T-stmt2} and $\pi(v) \in \IRR(\traceRS)$ by statement \ref{gp-correctenss_T-stmt1}.
Since $\traceRS$ is terminating and confluent this implies $\pi(u) =_G \pi(v) \neq_G 1$.
  \end{proof}

  \begin{lemma}
    \label{lem:gp-rewrite_bounds_exponents}
    The following length bounds hold:
    \begin{itemize}
       \item Rule~\eqref{eq:gp-rewrite_powereq}: $\abs{\abExp} \le 2\sigma$
      \item Rule~\eqref{eq:gp-rewrite_powereq_ncancel}: $|d| \le 5\sigma$ and $|e| \le 5\sigma$
      \item Rule~\eqref{eq:gp-rewrite_powerneq}: $|d| \le 4\sigma|q|$ and $|e| \le 4\sigma|p|$
      \item Rule~\eqref{eq:gp-rewrite_powertoplain}: $|\beta\alpha| < 2(\sigma - 1)|p|$
      \item Rules~\eqref{eq:gp-rewrite_plainpower} and~\eqref{eq:gp-rewrite_powerplain}: $|d| \le 1$
    \end{itemize}
  \end{lemma}

\begin{remark}\label{rem:mistake}
    Note that the proof of the bound for rule~\eqref{eq:gp-rewrite_powereq_ncancel} in \cref{lem:gp-rewrite_bounds_exponents} essentially relies on the assumption $a\neq a^{-1}$ for $a\in \Gamma$. Indeed, without this requirement we can construct examples where $d$ and $e$ in rule~\eqref{eq:gp-rewrite_powereq_ncancel} can be arbitrarily large.
    In the corresponding \cite[Lemma 15]{StoberW22,StoberW22arxivOld}, there is the unfortunate mistake that this condition is not required.
    Notice that the correctness of the whole shortening process described below depends on the bounds provided by \cref{lem:gp-rewrite_bounds_exponents}. Moreover, \cref{lem:gp-rewrite_bounds_exponents} is the only place in our construction for the power word problem in graph products  where we explicitly use the requirement $a\neq a^{-1}$ for $a\in \Gamma$.
\end{remark}

  \begin{proof}[Proof of \cref{lem:gp-rewrite_bounds_exponents}]
  We look at the individual statements:

\medskip\noindent
{\em Rule~\eqref{eq:gp-rewrite_powereq}}: The bound on $\abs{\abExp}$ is trivial
since $\abs{\alpha}, \abs{\beta} \leq \sigma \abs{p}$.

\medskip\noindent
{\em Rule~\eqref{eq:gp-rewrite_powereq_ncancel}}:  Let $\iota = \sgn x$ and $\kappa = \sgn y$. We distinguish two cases.
    First, assume that  $(p, \pi(\u)) \in I$. Then, due to condition \eqref{cond:gp-rewrite_powereq_ncancel},
    $\alpha\delta$ does not reduce with $T$ to a power of $p$.
    Most importantly, we have $\alpha\delta \neq_G 1$.

 We apply \cref{lem:reducing-pqr} (with $q=1$) to the reduced traces $ p^x \alpha$ and $ \delta p^y$.
    Due to the form of rule~\eqref{eq:gp-rewrite_powereq_ncancel} we obtain factorizations
    $ p^x \alpha =_M p^{x-d} \alpha' rs$ and
    $ \delta p^y =_M s^{-1}t \delta' p^{y-d}$
    such that $r$ and $t$ are $I$-cliques with $rt \rewrite{*}{\traceRS} v$ for an $I$-clique
    $v$ with $\alphabet(r) = \alphabet(t) = \alphabet (v)$ and $s$ is a suffix of $p^{x}\alpha$ and $s^{-1}$ is a prefix of $\delta p^{y} $. Moreover, we can assume that $p^\iota$ is not a prefix of $\alpha'$ and $p^\kappa$ is not a suffix of $\delta'$.

Assume that $\abs{s}\geq 3 \sigma \abs{p}$. We will deduce a contradiction.
Since $\abs{\alpha}, \abs{\delta} \leq \sigma \abs{p}$, this implies $\abs{y},\abs{x} \geq 2\sigma$ and (by \cref{lem:factor_shape}) $s$ has a suffix $p^{\iota\sigma} \alpha$.
Since $s^{-1}$ is a prefix of $\delta p^y$, it follows that $p^{\iota\sigma} \alpha$ is a suffix of $  p^{-y} \delta^{-1}$. Hence, there is a trace $q$ with
\begin{align}
    q p^{\iota\sigma} \alpha =_M p^{-y} \delta^{-1}.\label{qpalpha}
\end{align}

First, consider the case $\iota = \kappa$. As $\delta$ is a suffix of $p^{\kappa \sigma}$, there is some $q'$ with $\delta^{-1} q' =_M p^{-\kappa \sigma}$. Hence, by \eqref{qpalpha}, we have $q p^{\iota\sigma} \alpha q' =_M p^{-y-\kappa \sigma} =_M (p^{-\iota})^{|y|+\sigma}$.
As the number of letters from each $\cLelA \in \cL$ is the same in $p$ and $p^{-1}$,  \cref{lem:lettercount_conjugacy} implies that $p$ is conjugate to $p^{-1}$.
But since $p$ and $p^{-1}$ are both reduced this contradicts \cref{lemma-p-conjugate-p-1}. Be aware that here we rely upon the assumption $a \neq a^{-1}$ for all $a \in \Gamma$.

Now consider the case $\iota \neq \kappa$. For simplicity, assume $\iota=1$  and $\kappa = -1$ (the other case works exactly the same way), \ie $y \leq -2\sigma$.
Recall that $\alpha$ is a prefix $p^\sigma$.
With \eqref{qpalpha} we obtain $q p^\sigma \alpha =_M p^{-y} \delta^{-1} =_M p^{\sigma} p^{-y-\sigma} \delta^{-1}$, where $\alpha$ is a prefix of $p^{-y-\sigma}$.
It follows that $\alpha$ must be a suffix of
$p^{-y-\sigma} \delta^{-1}$. Hence, there exists
 $q''$ such that $p^{-y} \delta^{-1} =_M q p^\sigma \alpha =_M p^{\sigma} q''\alpha$ and therefore  $q p^\sigma  =_M
p^{\sigma} q''$. Note that $p^{-y} \delta^{-1}$ is a prefix of some
$p^z$ with $z \in \N$. In particular, there is some
$z \geq \sigma$ such that $p^{\sigma} q''\alpha$ is a prefix of $p^z$.
It follows that
$q''$ is a prefix of some $p^k$ with $k \in \N$.
Since $p \in \Omega$ is connected and primitive as a trace,
\cref{lem:p_factor} implies that $q=q''=p^\ell$ for some $\ell \in \N$.
We obtain
$p^{-y} \delta^{-1} =_M p^{\sigma+\ell} \alpha$ and hence
 $\alpha\delta =_G p^{j}$ for some $j\in \Z$ (indeed, as $p^{-1}$ is not a suffix of $\delta$ and $p$ is not
 a prefix  of $\alpha$, it follows that $\alpha\delta =_G 1$).
 Again, we obtained a contradiction.

Hence, in both cases it follows that $\abs{s} < 3 \sigma \abs{p}$.
   As above, we write $p^d \alpha =_M \alpha' r s$. By \Cref{lem:gp_alpha_beta_length}, we have $|\alpha'| < (\sigma - 1)|p|$. Moreover, as $r$ is an $I$-clique, $\abs{r}\leq \sigma$.
    It follows that
    \[d|p| \le |p^d| + |\alpha| =  |\alpha'| + |r| + |s| \le (\sigma - 1) |p| + \sigma + 3\sigma |p|.\]
    Hence, we obtain $d \le 5\sigma$.
    The bound on $|e|$ follows by symmetry.

    Now assume that $(p, \pi(\u)) \notin I$. Let $\alpha''$ be the suffix of $p^x\alpha$ and $\delta''$ be the prefix of $\delta p^y$  such that $\alpha'' \pi(u) \delta'' \rewrite{*}{\traceRS} v \pi(u)$ for an $I$-clique $v$. By \cref{lem:reducing-pqr}, $\alpha''$ and $\delta''$ must commute with $\pi(u)$.
    Thus, $p^{\iota}$ cannot be a factor of $\alpha''$ and hence, by \cref{lem:factor_shape}, $\alpha''$ must be a suffix of $p^{(\sigma - 1)\iota}\alpha$.
    Thus $|d| \le \sigma - 1 < 5\sigma$.
    The bound on $|e|$ follows by symmetry.

\medskip\noindent
{\em Rule~\eqref{eq:gp-rewrite_powerneq}}:
 By \cref{lem:reducing-pqr} we obtain factorizations
    $ p^x \alpha =_M p^{x-d} \alpha' rs$ and
    $\delta q^y =_M s^{-1}t \delta' p^{y-e}$
    such that $r$ and $t$ are $I$-cliques with $rt \rewrite{*}{\traceRS} v$ for the $I$-clique
    $v$ with $\alphabet(r) = \alphabet(t) = \alphabet (v)$ and $(rs, \pi(u)) \in I$.
    Thus, we can write $p^d\alpha = \alpha'rs$ and $\delta q^e = s^{-1} t \delta'$.
    Moreover, $s$ is a factor of $p^{\sgn(x)\nu}$ and $s^{-1}$ is a factor of $q^{\sgn(y)\nu}$  for some large enough $\nu \in \mathbb{N}$.
    We have $|s| \le 2\sigma(|p| + |q|)$.
    Otherwise, we would have $p = q$ by \Cref{th:cancel_length_factor_gp} contradicting $p \neq q$.
    By \Cref{lem:gp_alpha_beta_length}, $|\alpha'| < (\sigma - 1)|p|$ and $|\delta'| < (\sigma - 1)|q|$.
    It follows that
    \begin{align*}
        |p^d\alpha| &= |\alpha'rs| = |\alpha'| + |r| + |s| \\
        &< 2\sigma(|p| + |q|) + \sigma + (\sigma - 1)|p| \\
        &< 4\sigma(|p| + |q|) \\
        &\le 4\sigma|p| |q|\text{.}
    \end{align*}
    For the last inequality note that $p, q \in \Omega$ are composite.
    We get $|d|\cdot|p| \le |p^d\alpha| \le 4\sigma|p| |q|$, i.e., $|d| < 4\sigma|q|$.
    By symmetry, it follows that $|e| < 4\sigma |p|$.

\medskip\noindent
{\em Rule~\eqref{eq:gp-rewrite_powertoplain}}:
    By \Cref{lem:gp_alpha_beta_length}, $|\alpha| < (\sigma - 1) |p|$ and $|\beta| < (\sigma - 1) |p|$.

\medskip\noindent
{\em Rule~\eqref{eq:gp-rewrite_plainpower}}: There is a single letter prefix $b$ of $\beta p^x$ with $\alphabet(a) = \alphabet(b)$.
    Either $b$ is a prefix of $\beta$ in which case $d = 0$ or, if it is not, then $b$ must be a prefix of $p^{\sgn x}$ in which case $|d| = 1$.
    The same bound on rule~\eqref{eq:gp-rewrite_powerplain} follows by symmetry.
  \end{proof}
For a trace $w = w_1\cdots w_n \in M(\Delta, I_\Delta)$ with $w_i \in \Delta$, we define
  \begin{align*}
      \mu(w) &= \max\left\{|p| \mid w_i = \DelEl{\beta}{p^x}{\alpha} \in \Delta',\, i \in [1,n] \right\} .
  \end{align*}
For convenience we define $\mu(w)=2$ if $w$ does not contain any letter from $\Delta'$. The reason behind this is that $|p| \geq 2$ for all $\DelEl{\beta}{p^x}{\alpha} \in \Delta'$
as $p \in \Omega$ is required to be composite. Thus, in any case we have $\mu(w) \geq 2$.

  \begin{lemma}
    \label{cor:gp-steps_to_termination}
    If $w \rewrite{*}{\bigRS} v$, then $w \rewrite{\le k}{\bigRS} v$ with $k = 10\sigma^2 |w|_\Delta^2\mu(w)$.
  \end{lemma}

  \begin{proof}
    We observe the following bounds on the number of applications of the individual rules of the rewriting system (which we will prove below):
    \begin{enumerate}%[label=(\roman*)]
      \item Rules~\eqref{eq:gp-rewrite_powereq}  and~\eqref{eq:gp-rewrite_powertoplain} can be applied at most $|w|_{\Delta'}$ times in total.
      \item Rules~\eqref{eq:gp-rewrite_powereq_ncancel} and~\eqref{eq:gp-rewrite_powerneq} can be applied at most $2\sigma|w|_{\Delta'}$ times.
      \item Rule~\eqref{eq:gp-rewrite_plain} and length-reducing applications of rules~\eqref{eq:gp-rewrite_plainpower} and~\eqref{eq:gp-rewrite_powerplain} can occur at most $|w|_{\Gamma} + 2|w|_{\Delta'}(\sigma^2 + (\sigma-1)\mu(w))$ times.
      \item Length-preserving applications of rules~\eqref{eq:gp-rewrite_plainpower} and~\eqref{eq:gp-rewrite_powerplain} can occur at most $|w|_{\Gamma}|w|_{\Delta'} + 2|w|_{\Delta'}^2(\sigma^2 + (\sigma-1)\mu(w))$ times.
    \end{enumerate}
    Adding up those bounds we obtain a bound of
    \begin{align*}
        (2 \sigma + 1)|w|_{\Delta'} &+ (|w|_{\Delta'} + 1) \bigl(|w|_{\Gamma} + 2|w|_{\Delta'}(\sigma^2 + (\sigma-1)\mu(w))\bigr)\\
        &\le 10\sigma^2 |w|_\Delta^2\mu(w).
    \end{align*}
    Thus, let us prove the individual bounds 1--4.

    \begin{enumerate}%[label=(\roman*)]
      \item For an application $w \rewrite{}{\bigRS} \tilde w$ of rule~\eqref{eq:gp-rewrite_powereq}  or~\eqref{eq:gp-rewrite_powertoplain} we have $|w|_{\Delta'} > |\tilde w|_{\Delta'}$. Moreover, no rule increases $|w|_{\Delta'}$.
      \item For bounding the number of applications of  rules~\eqref{eq:gp-rewrite_powereq_ncancel} and~\eqref{eq:gp-rewrite_powerneq},
      write $w=_Mw_1 \cdots w_n$ with $w_i \in \Delta$. We say that a pair
      $(i,j)$ with $1 \leq i < j \leq n$ \emph{potentially cancels}
      if $w_i,w_j \in \Delta'$ and there is some $\cLelA \in \cL$ such that $\cLelA \in\alphabet(w_i) \cap \alphabet(w_j)$ and for all $i<k<j$ either $w_k \in \Gamma$ or $w_k \in \Delta'$ with $\cLelA\not\in\alphabet(w_k)$. Notice that the number of pairs that potentially cancels
      does not depend on the representative $w_1 \cdots w_n$ we started with (as the letters from $\Gamma_\cLelA$ are linearly ordered).
      Moreover, if a rule \eqref{eq:gp-rewrite_powereq_ncancel} or~\eqref{eq:gp-rewrite_powerneq} can be applied at positions $i < j$ in $w$, then there must be
      letters $a$ in $\pi(w_i)$ and $b$ in $\pi(w_j)$ from the same alphabet $\Gamma_\cLelA$ (in particular, $\cLelA \in \alphabet(w_i) \cap \alphabet(w_j)$) such that $(a, w_k) \in I_\Delta$ for all $i < k < j$. Therefore,
      $(i,j)$ potentially cancels~-- the converse, however, does not hold. Furthermore, for each pair that potentially cancels (at some point during the rewriting process $w \rewrite{*}{\bigRS} v$), a rule of type \eqref{eq:gp-rewrite_powereq_ncancel} or~\eqref{eq:gp-rewrite_powerneq} can be applied at most once. This is because the right-hand side of the rule is in $\pi^{-1}(\IRR(T))$ and irreducibility is not changed by the application of other rules (indeed, not by any application of rules from $T$).

      Thus, it suffices to bound the number of pairs that potentially cancel.
      Initially, for each position $i$ with $w_i\in \Delta'$ there are at most
      $\sigma$ positions $j>i$ such that $(i,j)$ potentially cancels as
      well as at most
      $\sigma$ positions $j<i$ such that $(j,i)$ potentially cancels.
      Hence, initially there are at most $\sigma|w|_{\Delta'}$ pairs that potentially cancel.
      Each removal of a letter from $\Delta'$ by rule~\eqref{eq:gp-rewrite_powertoplain} might generate up to $\sigma$ pairs that potentially cancel. All other rules do not generate pairs that potentially cancel.
        Thus, in the entire rewriting process only $2\sigma|w|_{\Delta'}$ pairs potentially cancel, which gives us an upper bound of at most $2\sigma|w|_{\Delta'}$ applications of rules \eqref{eq:gp-rewrite_powereq_ncancel} and~\eqref{eq:gp-rewrite_powerneq}.

      \item\label{point3} Initially, there are $|w|_{\Gamma}$ letters from $\Gamma$.
      Each application of rules~\eqref{eq:gp-rewrite_powereq_ncancel} or~\eqref{eq:gp-rewrite_powerneq} increases $|\cdot|_{\Gamma}$ by up to $\sigma$ (since $\abs{v} \leq \sigma$).
      Each application of rule~\eqref{eq:gp-rewrite_powertoplain} increases $|\cdot|_{\Gamma}$ by up to $2(\sigma - 1)\mu(w)$.
      Rule~\eqref{eq:gp-rewrite_plain} as well as
      length-reducing applications of rules~\eqref{eq:gp-rewrite_plainpower} and~\eqref{eq:gp-rewrite_powerplain} decrease $|\cdot|_{\Gamma}$ by at least one.
      Therefore, in total they can occur at most $|w|_{\Gamma} + \sigma \cdot 2\sigma |w|_{\Delta'} + 2(\sigma - 1)\mu(w) \cdot |w|_{\Delta'} = |w|_{\Gamma} + 2|w|_{\Delta'}(\sigma^2 + (\sigma-1)\mu(w))$  times.

      \item
      A length-preserving application of rule~\eqref{eq:gp-rewrite_plainpower} or~\eqref{eq:gp-rewrite_powerplain} involves a letter from $\Delta'$ and a letter from $\Gamma$. Moreover, for each such pair of letters, a rule \eqref{eq:gp-rewrite_plainpower} or~\eqref{eq:gp-rewrite_powerplain} can be applied at most once.
      There are $|w|_{\Gamma}$ letters from $\Gamma$ initially.
      Up to $2|w|_{\Delta'}(\sigma^2 + (\sigma-1)\mu(w))$ additional letters from $\Gamma$ are created by rules~\eqref{eq:gp-rewrite_powereq_ncancel}, \eqref{eq:gp-rewrite_powerneq} and~\eqref{eq:gp-rewrite_powertoplain} (see point \ref{point3}).
      Multiplying that with $|w|_{\Delta'}$, the number of letters from $\Delta'$, we obtain a bound of $|w|_{\Gamma}|w|_{\Delta'} + 2|w|_{\Delta'}^2(\sigma^2 + (\sigma-1)\mu(w))$ for the number of length-preserving applications of rules~\eqref{eq:gp-rewrite_plainpower} and~\eqref{eq:gp-rewrite_powerplain}. \qedhere
    \end{enumerate}
  \end{proof}

  \subsubsection{The shortened word}
  \label{subsec:gp-the-shortened-word}

  In this section we describe the shortening process.
  It is an almost the same as for free groups as we presented it in \cite{LohreyW19} (see the version on arXiv \cite{LohreyW19arxiv} for details).
  For the further consideration we fix a trace
  $u \in M(\Delta,I_\Delta)$ and some $p \in \Omega$. We consider all letters from $\Delta_p \cup \Delta_p^{-1}$ in $u$ and write $u$ as
  \begin{align}\label{p_factorization}
  u = u_0\, \DelEl{\beta_1}{p^{y_1}}{\alpha_1}\, u_1 \cdots \DelEl{\beta_m}{p^{y_m}}{\alpha_m}\, u_m
  \end{align}
  with $u_i \in M(\Delta,I_\Delta)$ not containing any letter from $\Delta_p \cup \Delta_p^{-1}$ and $\DelEl{\beta_i }{p^{y_i}}{\alpha_i} \in \Delta_p \cup \Delta_p^{-1}$.
  We define $\eta_p(u) = \sum_{j=1}^m y_j$ and $\eta_p^i(u) = \sum_{j=1}^i y_j$ for $i \in [0,m]$.
  The following lemma follows from the bounds given in \cref{lem:gp-rewrite_bounds_exponents}.

  \begin{lemma}
    \label{lem:gp-rule_bound_diff_eta}
    Let $u, v \in M(\Delta,I_\Delta)$ and $u \rewrite{}{\bigRS} v$.
    For every prefix $v'$ of $v$ there is a prefix $u'$ of $u$ such that for all $p \in \Omega$
    \begin{equation*}
        |\eta_p(u') - \eta_p(v')| \le 5\sigma\mu(u)\text{.}
    \end{equation*}
    If the applied rule is neither~\eqref{eq:gp-rewrite_powereq} nor~\eqref{eq:gp-rewrite_powertoplain}, then for all $p \in \Omega$ and $i \in [0, m]$ we have
    \begin{equation*}
        |\eta_p^i(u) - \eta_p^i(v)| \le 5\sigma\mu(u)\text{.}
    \end{equation*}
  \end{lemma}

\begin{proof}
   Let us start by proving the second statement.
    First, observe that in case the applied rule is neither~\eqref{eq:gp-rewrite_powereq} nor~\eqref{eq:gp-rewrite_powertoplain}, there is a one-to-one correspondence between letters from $\Delta_p \cup \Delta_p^{-1}$ in $u$ and $v$. Moreover, when applying a rule other than \eqref{eq:gp-rewrite_powereq} or~\eqref{eq:gp-rewrite_powertoplain}, at most two letters from $\Delta_p \cup \Delta_p^{-1}$ are modified; when also excluding rule \eqref{eq:gp-rewrite_powereq_ncancel} at most one letter from $\Delta_p \cup \Delta_p^{-1}$ is modified. Therefore, by \cref{lem:gp-rewrite_bounds_exponents}, in case of rule  \eqref{eq:gp-rewrite_powereq_ncancel}, we have $|\eta_p^i(u) - \eta_p^i(v)| \leq 10\sigma \leq 5\sigma\mu(u)$ (because $\mu(u) \geq 2$). In case of rules \eqref{eq:gp-rewrite_powerneq} or \eqref{eq:gp-rewrite_plainpower}--\eqref{eq:gp-rewrite_plain}, also by \cref{lem:gp-rewrite_bounds_exponents}, it follows that $|\eta_p^i(u) - \eta_p^i(v)| \leq 4\sigma\mu(u)$.

  Now, observe that the second statement implies the first one in the case that the applied rule is neither~\eqref{eq:gp-rewrite_powereq} nor~\eqref{eq:gp-rewrite_powertoplain}: if $v'$ is a prefix of $v$ containing exactly the first $i$ letters from $\Delta_p \cup \Delta_p^{-1}$, then we can choose $u'$ to be any prefix of $u$ containing exactly the first $i$ letters from $\Delta_p \cup \Delta_p^{-1}$.

Now, suppose that the applied rule is \eqref{eq:gp-rewrite_powertoplain}. Then we have $\eta_p(u) = \eta_p(v)$; moreover, for a prefix $v'$ of $v$ it is clear that we can find a corresponding prefix $u'$ of $u$ with $\eta_p(u') = \eta_p(v')$.

Finally, consider the case of rule~\eqref{eq:gp-rewrite_powereq} and denote the applied rule by $\ell \to r$. Note that $r$ is a word of length one. Thus, $r$ is either completely outside of $v'$ or completely inside of $v'$. In the first case,  we can choose $u' = v'$. In the second case, we can write $v' = srt$ and set $u' = s\ell t$. \cref{lem:gp-rewrite_bounds_exponents} yields
$|\eta_p(u') - \eta_p(v')| \le 2\sigma$.
\end{proof}

  In the following we define a set $\mathcal{C}$ of intervals to be carved out of the exponents from the input power word during the shortening process.
  \begin{definition} \label{def:compatible}
    Let $\mathcal{C} = \left\{ [l_j, r_j] \mid j \in [1,k] \right\}$ with $l_j,r_j \in  \Z$ be a set of finite, non-empty, and pairwise disjoint intervals of integers, where $k = |\mathcal{C}|$.
    We assume the intervals to be ordered, \ie $r_j < l_{j+1}$.
    We define the size of an interval $d_j = r_j - l_j + 1$ (which is the number of elements in $[l_j, r_j]$). An element $u \in M(\Delta,I_\Delta)$ (written in the form \eqref{p_factorization}) is said to be \emph{compatible} with the set of intervals $\mathcal{C}$, if for every prefix $u'$ of $u$ and all $j \in [1,k]$, we have $\eta_p(u') \notin [l_j, r_j]$.
  \end{definition}

  \begin{definition}
    \label{def:shortened}
    Let $\mathcal{C}  = \left\{ [l_j, r_j] \mid j \in [1,k] \right\}$ be compatible with $u \in M(\Delta,I_\Delta)$ written in the form \eqref{p_factorization}.
    The {\em shortened word} corresponding to $u$ is
    \begin{align*}
      \mathcal{S}_{\mathcal{C}}(u) = u_0 \DelEl{\beta_1}{p^{z_1}}{\alpha_1} u_1 \cdots \DelEl{\beta_m}{p^{z_m}}{\alpha_m}u_m\text{.}
    \end{align*}
    The new exponents $z_i$ are defined as
    \begin{align*}
      z_i = y_i - \sgn(y_i)\cdot \sum_{j \in C_i} d_j\text{,}
    \end{align*}
    where $C_i(u)$ is the set of indices of intervals to be removed from $y_i$, defined by
    \begin{align*}
      C_i(u) = \left\{ \begin{matrix*}[l]
                      \set{ j \in [1,k] }{ \eta_p^{i-1}(u) < l_j \le r_j < \eta_p^i(u) } &\quad\text{ if } y_i > 0,\\[2mm]
                      \set{  j \in [1,k] }{  \eta_p^{i}(u) < l_j \le r_j < \eta_p^{i-1}(u) } &\quad\text{ if } y_i < 0.
      \end{matrix*} \right.
    \end{align*}
  \end{definition}
  Note that Definitions~\ref{def:compatible} and \ref{def:shortened} depend on our fixed $p \in \Omega$.

  \begin{lemma}
    \label{lem:gp-shortened_word_irr}
    If $u \in \IRR(\bigRS)$, then $\mathcal{S}_{\mathcal{C}}(u) \in \IRR(\bigRS)$.
  \end{lemma}

\newcommand{\startIndex}{\iota}
\newcommand{\lastIndex}{\tau}

  {\allowdisplaybreaks
    \begin{proof}
    Assume that $u \in \IRR(\bigRS)$. In particular, $y_i \neq 0$ for all $i \in [1,m]$ (otherwise we could apply rule \eqref{eq:gp-rewrite_powertoplain}).
      We prove the lemma by showing that $\sgn(y_i) = \sgn(z_i)$ and $z_i \neq 0$
      for all $i \in [1,m]$.
      As the intervals in $\mathcal{C}$ are ordered, there are $\startIndex$ and $\lastIndex$ such that $C_i(u)$ consists of all indices from $\startIndex$ to $\lastIndex$.
      In case $y_i > 0$ we have
      \begin{align*}
        z_i &= y_i - \sum_{j \in C_i(u)} d_j\\
        &= y_i - \sum_{j = \startIndex}^\lastIndex d_j\\
        &= y_i - \sum_{j = \startIndex}^\lastIndex (r_j - l_j + 1)\\
        &\geq y_i - (r_{\lastIndex} - l_{\startIndex} + 1)\tag{because $r_j < l_{j+1}$}\\
        &\geq y_i - \bigl((\eta_p^i(u) - 1) - (\eta_p^{i-1}(u) + 1) + 1\bigr)\\
        &= y_i - y_i + 1\\
        &= 1 .
      \end{align*}
      The case $y_i < 0$ follows by symmetry.
    \end{proof}
  }
  \begin{definition}
    We define the distance between some $u$ and the closest interval from $\mathcal{C}$ as
    \begin{align*}
      \dist_p(u, \mathcal{C}) = \min \set{ | \eta_p^i(u) - x | }{ i \in [0,m],\, x \in [l, r] \in \mathcal{C} }\text{.}
    \end{align*}
  \end{definition}

  From that definition the following statement follows immediately.

  \begin{lemma}
    \label{cor:gp_compatoble_dist}
    $\dist_p(u, \mathcal{C}) > 0$ if and only if $u$ is compatible with $\mathcal{C}$.
  \end{lemma}

  We want to show that given some requirements are fulfilled, any rewriting step that is possible on $u$ is also possible on $\mathcal{S}_{\mathcal{C}}(u)$.

  \begin{lemma}
    \label{lem:gp-dist_one_step}
    If $\dist_p(u, \mathcal{C}) > 5\sigma \mu(u)$ and $u \rewrite{}{\bigRS} v$, then $\mathcal{S}_{\mathcal{C}}(u) \rewrite{}{\bigRS} \mathcal{S}_{\mathcal{C}}(v)$.
  \end{lemma}

  \begin{proof}
    Observe that $u$ is compatible with $\mathcal{C}$.
    By \cref{lem:gp-rule_bound_diff_eta} we have $\dist_p(v, \mathcal{C}) > 0$ and thus $v$ is compatible with $\mathcal{C}$.
    It follows that $\mathcal{S}_{\mathcal{C}}(u)$ and $\mathcal{S}_{\mathcal{C}}(v)$ are defined.

    To prove the lemma, we compare the shortened version of $u$ and $v$ and show that a rule from $\bigRS$ can be applied.
    We distinguish which rule from $\bigRS$ has been applied in the rewrite step
    $u \rewrite{}{\bigRS} v$.

\medskip\noindent
{\em Rule~\eqref{eq:gp-rewrite_powereq_ncancel},~\eqref{eq:gp-rewrite_powerneq},~\eqref{eq:gp-rewrite_plainpower},~\eqref{eq:gp-rewrite_powerplain} or~\eqref{eq:gp-rewrite_plain}}:  If one of these rules has been applied, the shortening process has the same effect on $u$ and $v$, \ie $C_i(u) = C_i(v)$ for all $i$ (this is because by \cref{lem:gp-rule_bound_diff_eta} we have $|\eta_p^i(u) - \eta_p^i(v)| \le 5\sigma\mu(u)$ and the assumption $\dist_p(u, \mathcal{C}) > 5\sigma \mu(u)$).
      The same rule that has been applied in $u \rewrite{}{\bigRS} v$ can also be used to get $\mathcal{S}_{\mathcal{C}}(u) \rewrite{}{\bigRS} \mathcal{S}_{\mathcal{C}}(v)$:
    Consider a letter $\DelEl{\beta}{p^{y_i}}{\alpha}$ in $u$ that by the shortening process is changed to $\DelEl{\beta}{p^{z_i}}{\alpha}$ with $z_i\neq y_i$. Then $C_i(u) \neq \emptyset$ and we have
    \[\abs{z_i} = \abs{y_i} - \sum_{j\in C_i(u)} d_j \geq 2\dist_p(u, \mathcal{C}) \geq 5\sigma\mu(u). \]
    Thus, by \cref{lem:gp-rewrite_bounds_exponents} the exponents in $\mathcal{S}_{\mathcal{C}}(u)$ are large enough to apply a rule of the same type as in $u \rewrite{}{\bigRS} v$.

 \medskip\noindent
{\em Rule~\eqref{eq:gp-rewrite_powertoplain}}:
      If rule~\eqref{eq:gp-rewrite_powertoplain} is applied, then $C_\ell(v) = C_\ell(u)$ for $\ell < i$ and $C_\ell(v) = C_{\ell + 1}(u)$ for $\ell \ge i$.
      We also know $y_i = 0$, which is not altered by the shortening process, \ie $C_i(u) = \emptyset$.
      Thus, rule \eqref{eq:gp-rewrite_powertoplain} can be applied to $\mathcal{S}_{\mathcal{C}}(u)$ to obtain $\mathcal{S}_{\mathcal{C}}(v)$.

 \medskip\noindent
{\em Rule~\eqref{eq:gp-rewrite_powereq}}:
      Finally, consider the case that rule~\eqref{eq:gp-rewrite_powereq} has been applied.
      Let
      \begin{align*}
        u = u_0 \DelEl{\beta_1}{p^{y_1}}{\alpha_1} u_1 \cdots \DelEl{\beta_i}{p^{y_i}}{\alpha_i}\: u_i \:\DelEl{\beta_{i+1}}{p^{y_{i+1}}}{\alpha_{i+1}}\:  u_{i+1} \cdots \DelEl{\beta_m}{p^{y_m}}{\alpha_m} u_m\text{.}
      \end{align*}
      The result of applying the rule is
      \begin{align*}
        v = u_0 \DelEl{\beta_1}{p^{y_1}}{\alpha_1} u_1 \cdots \DelEl{\beta_i}{p^{y_i + y_{i+1} + \abExp}}{ \alpha_{i+1}} \: u_i u_{i+1} \cdots \DelEl{\beta_m}{p^{y_m}}{\alpha_m} u_m
      \end{align*}
      where $\alpha_i\beta_{i+1} =_G p^\abExp$ (notice that $(u_i,p) \in I_\Delta$, for otherwise rule~\eqref{eq:gp-rewrite_powereq} cannot be applied).
      On powers not modified by the rule the shortening process behaves the same on $u$ and $v$, \ie $C_\ell(v) = C_\ell(u)$ for $\ell < i$ and $C_\ell(v) = C_{\ell + 1}(u)$ for $\ell > i$
      (because by \cref{lem:gp-rewrite_bounds_exponents}, $|\abExp| \le 2\sigma < \dist_p(u, \mathcal{C})$).
      The result of the shortening process on $v$ is
      \begin{align*}
        \mathcal{S}_{\mathcal{C}}(v) = u_0 \DelEl{\beta_1}{p^{z_1}}{\alpha_1} u_1 \cdots \DelEl{\beta_i}{p^{\tilde{z}_i}}{\alpha_{i+1}} u_i u_{i+1} \cdots \DelEl{\beta_m}{p^{z_{m}}}{\alpha_m} u_m\text{,}
      \end{align*}
      where $\tilde{z}_i = y_i + y_{i+1} + \abExp - \sgn(y_i + y_{i+1} + f) \cdot \sum_{\ell \in C_i(v)} d_\ell$.
      Rule~\eqref{eq:gp-rewrite_powereq} can be also applied to
      $\mathcal{S}_{\mathcal{C}}(u)$ (only $\alpha_i\beta_{i+1} =_G p^\abExp$
      is needed for this) and yields
      \begin{align*}
        \hat{v} = u_0 \DelEl{\beta_1}{p^{z_1}}{\alpha_1} u_1 \cdots \DelEl{\beta_i}{p^{z_i + z_{i+1} + \abExp}}{\alpha_{i+1}} u_i u_{i+1} \cdots \DelEl{\beta_m}{p^{z_{m}}}{\alpha_m} u_m\text{.}
      \end{align*}

      \noindent
      We need to show that $\tilde{z}_i = z_i + z_{i+1} + \abExp$, \ie
      $$z_i + z_{i+1} = y_i + y_{i+1} - \sgn(y_i + y_{i+1} + \abExp) \cdot \sum_{\ell \in C_i(v)} d_\ell.$$
We start by showing that, if $C_i(v) \neq \emptyset$, then  $\sgn(y_i + y_{i+1} + \abExp) = \sgn(y_i + y_{i+1})$. Indeed, if $j\in C_i(v)$, then
for all $x \in [l_j,r_j]$ we have
\[\abs{y_i + y_{i+1} + \abExp} \geq \big|\eta_p^{i-1}(v)-x\big| =
\big|\eta_p^{i-1}(u)-x\big| \geq \dist_p(u, \mathcal{C}) .\]
The last inequality follows from $[l_j,r_j ] \in \mathcal{C}$.
Since, by \cref{lem:gp-rewrite_bounds_exponents}, $|\abExp| \le 2\sigma < \dist_p(u, \mathcal{C})$, it follows that $\sgn(y_i + y_{i+1} + \abExp) = \sgn(y_i + y_{i+1})$.

 Therefore, in any case we have
$$\sgn(y_i + y_{i+1}) \cdot \!\sum_{\ell \in C_i(v)} d_\ell = \sgn(y_i + y_{i+1} + \abExp) \cdot \!\sum_{\ell \in C_i(v)} d_\ell$$
and it remains to show $z_i + z_{i+1} = y_i + y_{i+1} - \sgn(y_i + y_{i+1}) \cdot \sum_{\ell \in C_i(v)} d_\ell$.

  Now let us distinguish two cases:
      First, consider the case that $y_i$ and $y_{i+1}$ have the same sign.
      In that case we have $ C_i(u) \cap C_{i+1}(u) = \emptyset$ and (again, because by \cref{lem:gp-rewrite_bounds_exponents}, $|\abExp| \le 2\sigma < \dist_p(u, \mathcal{C})$)
      it follows that $C_i(v) = C_i(u) \cup C_{i+1}(u)$. Thus, we obtain
      \begin{align*}
        z_i + z_{i+1} &= y_i - \sgn(y_i) \cdot \!\sum_{\ell \in C_i(u)} d_\ell + y_{i+1} - \sgn(y_{i+1}) \cdot \!\!\sum_{\ell \in C_{i+1}(u)} d_\ell\\
        &= y_i + y_{i+1} - \sgn(y_i + y_{i+1}) \cdot \!\sum_{\ell \in C_i(v)} d_\ell.
      \end{align*}
      Second, we look at the case where $y_i$ and $y_{i+1}$ have opposite sign.
      We assume $|y_i| \geq |y_{i+1}|$.
      The other case is symmetric. We have
      $C_{i+1}(u) \subseteq C_i(u)$ and $C_i(v) = C_i(u) \setminus C_{i+1}(u)$. This implies
      \begin{align*}
        z_i + z_{i+1} &= y_i - \sgn(y_i) \cdot \sum_{\ell \in C_i(u)} d_\ell + y_{i+1} - \sgn(y_{i+1}) \cdot \sum_{\ell \in C_{i+1}(u)} d_\ell\\
        &= y_i + y_{i+1} - \sgn(y_i) \cdot \left( \sum_{\ell \in C_i(u)} d_\ell - \sum_{\ell \in C_{i+1}(u)} d_\ell \right)\\
        &= y_i + y_{i+1} - \sgn(y_i + y_{i+1}) \cdot \sum_{\ell \in C_i(v)} d_\ell.%\qedhere
      \end{align*}
      Note that in the case that $y_i = -y_{i+1}$ we have $C_i(u) = C_{i+1}(u) $ and $C_i(v)= \emptyset$, so the last equality also holds in this case.
      This concludes the proof of the lemma.
      %%%%
  \end{proof}

  \begin{lemma}
    \label{lem:gp-transitions_of_shortened_word}
    If $\dist_p(u, \mathcal{C}) > 5k\sigma \mu(u)$ and $u \rewrite{\le k}{\bigRS} v$, then $\mathcal{S}_{\mathcal{C}}(u) \rewrite{\le k}{\bigRS} \mathcal{S}_{\mathcal{C}}(v)$.
  \end{lemma}

  \begin{proof}
    We prove the lemma by induction.
    If $k=1$, then the statement follows from \cref{lem:gp-dist_one_step}.
    If $k>1$, then there is a $u' \in M(\Delta, I_\Delta)$ such that
    \begin{align*}
      u \rewrite{}{\bigRS} u' \rewrite{\le k-1}{\bigRS} v\text{.}
    \end{align*}
    By \cref{lem:gp-rule_bound_diff_eta} we have $\dist_p(u', \mathcal{C}) > 5(k - 1)\sigma \mu(u)$.
    As none of the rules of $\bigRS$ increases $\mu(\cdot)$, it follows that $\dist_p(u', \mathcal{C}) > 5(k - 1)\sigma \mu(u')$.
    Therefore, $\mathcal{S}_{\mathcal{C}}(u') \rewrite{\le k-1}{\bigRS} \mathcal{S}_{\mathcal{C}}(v)$ by induction.
    By \cref{lem:gp-dist_one_step} we have $\mathcal{S}_{\mathcal{C}}(u) \rewrite{}{\bigRS} \mathcal{S}_{\mathcal{C}}(u')$.
    Combining those statements we conclude $\mathcal{S}_{\mathcal{C}}(u) \rewrite{\le k}{\bigRS} \mathcal{S}_{\mathcal{C}}(v)$.
  \end{proof}

  We continue by defining a concrete set of intervals $\mathcal{C}_{u,p}^K$ based on the following intuitive idea.
  From \cref{lem:gp-correctenss_T} and \cref{cor:gp-steps_to_termination} we know that $\pi(u) =_G 1$ if and only if $u \rewrite{\le k}{\bigRS} 1$ with $k = 10\sigma^2|u|_\Delta^2\mu(u)$.
  By \cref{lem:gp-rule_bound_diff_eta}, each application of a rule changes $\eta_p(\cdot)$ by at most $5\sigma \mu(u)$.
  Thus, the partial sums of the exponents change by less than
  \begin{equation*}
      K = 50\sigma^3|u|_\Delta^2\mu(u)^2 + 1{.}
  \end{equation*}
    Let $\left\{ c_1, \dots c_\ell \right\} = \left\{ \eta_p^i(u) \mid i \in [0, m] \right\}$ be the ordered set of the $\eta_p^i(u)$, \ie $c_1 < \dots < c_\ell$.
  We define the set of intervals
  \begin{equation}
    \label{eq:gp-intervals_CupK}
    \mathcal{C}_{u,p}^K = \left\{ [c_i + K, c_{i+1} - K] \mid i \in [1, \ell -1], c_{i+1} - c_i \ge 2K \right\}\text{.}
  \end{equation}
Let us write $\mathcal{C}$ for $\mathcal{C}_{u, p}^K$ in the following.
Note that $\abs{\mathcal{C}} \leq m$.
The next lemma shows that the shortened word computed with the set $\mathcal{C}$ is the identity if and only if the original word is the identity.

  \begin{lemma}
    \label{lem:gp-short-word_confluent_one}
    $\pi(u) =_G 1$ if and only if $\pi(\mathcal{S}_{\mathcal{C}}(u)) =_G 1$.
  \end{lemma}

  \begin{proof}
    Let $k = 10\sigma^2|u|_\Delta^2\mu(u)$.
    From the definition of $\mathcal{C}$ we obtain $\dist_p(u, \mathcal{C}) > 5k\sigma\mu(u)$.

    First, let $\pi(u) =_G 1$.
    By \cref{lem:gp-correctenss_T} (point \ref{gp-correctenss_T-stmt3}) this is equivalent to $u \rewrite{*}{\bigRS} 1$, which by \cref{cor:gp-steps_to_termination} is equivalent to $u \rewrite{\le k}{\bigRS} 1$.
    By \cref{lem:gp-transitions_of_shortened_word} we have $\mathcal{S}_{\mathcal{C}}(u) \rewrite{*}{\bigRS} \mathcal{S}_{\mathcal{C}}(1) = 1$.
    Applying \cref{lem:gp-correctenss_T} (point \ref{gp-correctenss_T-stmt2}) we obtain $\pi(\mathcal{S}_{\mathcal{C}}(u)) \rewrite{*}{\traceRS} 1$ and, hence, $\pi(\mathcal{S}_{\mathcal{C}}(u)) =_G 1$.

    Second, assume that $\pi(u) \neq_G 1$. Since $\bigRS$ is terminating, there
    is some $v \in \IRR(\bigRS)$ with $u \rewrite{*}{\bigRS} {v}$.
    We cannot have $v = 1$ as this, by \cref{lem:gp-correctenss_T} (point \ref{gp-correctenss_T-stmt2}), would yield $\pi(u) =_G 1$.
    Hence, we have $v \neq 1$. Moreover, \cref{cor:gp-steps_to_termination} implies
    $u \rewrite{\le k}{\bigRS} {v}$ and with
    \cref{lem:gp-transitions_of_shortened_word} we get $\mathcal{S}_{\mathcal{C}}(u) \rewrite{*}{\bigRS} \mathcal{S}_{\mathcal{C}}(v)$.
    By \cref{lem:gp-shortened_word_irr} we have $\mathcal{S}_{\mathcal{C}}(v) \in \IRR(\bigRS)$.
    As the shortening process does not remove any letters but only replaces them, we have     $\mathcal{S}_{\mathcal{C}}({v}) \neq_{M(\Delta,I_\Delta)} 1$.
    It follows that    $\pi(\mathcal{S}_{\mathcal{C}}({v})) \neq_G 1$ (otherwise, \cref{lem:gp-correctenss_T} (point \ref{gp-correctenss_T-stmt3})
    implies $\mathcal{S}_{\mathcal{C}}(v)\notin \IRR(\bigRS)$).
    Finally, with \cref{lem:gp-correctenss_T} (point \ref{gp-correctenss_T-stmt2}) we get
    $\pi(\mathcal{S}_{\mathcal{C}}(u)) =_G \pi(\mathcal{S}_{\mathcal{C}}({v})) \neq_G 1$.
  \end{proof}

  The next lemma shows that when using the set $\mathcal{C}$ from \eqref{eq:gp-intervals_CupK}, the exponents of the shortened word are bounded by a polynomial.

  \begin{lemma}
    \label{lem:gp-short-exponent-bound}
    Let $\mathcal{S}_{\mathcal{C}}(u) = u_0 \DelEl{\beta_1}{p^{z_1}}{\alpha_1} u_1 \cdots \DelEl{\beta_m}{p^{z_m}}{\alpha_m} u_m$ for some $u \in M(\Delta, I_\Delta)$.
    Then $|z_i| \le 101m\sigma^3|u|_\Delta^2\mu(u)^2$ for all $i \in [1, m]$.
  \end{lemma}

  \begin{proof}
  We have
  { \allowdisplaybreaks\begin{align*}
                         |z_i| &= \Bigl\lvert y_i - \sgn(y_i)\cdot \sum_{j \in C_i} d_j \Bigr\rvert\\
                         &= \left\lvert y_i \right\rvert - \sum_{j \in C_i} d_j \\
                         &\overset{(\ref{short-exp-bound-fact1})}{=} \left\lvert y_i \right\rvert - \sum_{j \in C_i} \max\{0, c_{j+1} - c_{j} - 2K + 1\} \\
                         &\le \left\lvert y_i \right\rvert - \sum_{j \in C_i} (c_{j+1} - c_{j} - 2K + 1)\\
                         &= \left\lvert y_i \right\rvert - \sum_{j \in C_i} (c_{j+1} - c_{j})  + \sum_{j \in C_i} (2K - 1)\\
                         &= \sum_{j \in C_i} (2K - 1) \\
                         &\overset{(\ref{short-exp-bound-fact2})}{\le} m(2K - 1) \le 2mK,
  \end{align*}}
    where we used the following facts:
    \begin{enumerate}[(A)]%[label=(\roman*)]
      \item Definition of $\mathcal{C}$ in~\eqref{eq:gp-intervals_CupK}.\label{short-exp-bound-fact1}
      \item $|C_i| \le |\mathcal{C}| \le m$.\label{short-exp-bound-fact2}
    \end{enumerate}
    The lemma follows by plugging in $K = 50\sigma^3|u|_\Delta^2\mu(u)^2 + 1$.
  \end{proof}

\subsubsection{Solving the power word problem}\label{sec:pwp_wrap_up}

Now we are ready for the proofs of our main results for graph products.

  \begin{theorem}
    \label{th:gp-power-word-AC0}
    Let $G = \GP(\cL, I, \left(G_\cLelA\right)_{\cLelA \in \cL})$ be a graph product of f.g.~groups such that no $G_\cLelA$ contains any element $a$ with $a^2=_{G_\cLelA} 1$ and $a\neq_{G_\cLelA} 1$.
    Then the power word problem in $G$ can be decided in $\uAC^0$ with oracle gates for the word problem in $F_2$ and for the power word problems in the base groups $G_\cLelA$.
  \end{theorem}

\begin{proof}
	By \cref{lem:gp-preprocessing} the preprocessing can be done in $\uAC^0$ with oracles for the word problems in $G$ and $F_2$ (thus, by~\cite[Theorem 5.6.5, Theorem 5.6.14]{kausch2017parallel} in  $\uAC^0(\WP(F_2), (\WP(G_\cLelA))_{\cLelA \in \cL}) \subseteq \uAC^0(\WP(F_2), (\PowWP(G_\cLelA))_{\cLelA \in \cL})$).
    Let \eqref{eq-w-preproc} be the power word obtained after the preprocessing.
	The shortening procedure can be computed in parallel for each $p \in \{p_i \mid i \in [1, n]\}$.
	It requires iterated additions, which is in $\uTC^0 \subseteq \uAC^0(\WP(F_2))$.
	By \cref{lem:gp-short-exponent-bound} the exponents of the shortened word are bounded by a polynomial in the input length.
	We write the shortened word as a simple power word of polynomial length and solve the simple power word problem, which
	by \cref{lem:spowwp}, is in $\uAC^0(\WP(F_2), (\PowWP(G_\cLelA))_{\cLelA \in \cL})$.
\end{proof}

 \begin{corollary}
    \label{th:power-word-AC0}
    Let $G$ be a RAAG.
    The power word problem in $G$ is $\uAC^0$-Turing reducible to the word problem in the free group $F_2$ and, thus, in \LOGSPACE.
  \end{corollary}
The proof of the following result is analogous to the proof of Theorem~\ref{th:gp-power-word-AC0}
using the respective statements of the lemmas for the uniform case.

 \begin{theorem}
    \label{th:gp-power-word-CeqL-uniform}
    Let \(\mathcal{C}\) be a non-trivial class of f.g.~groups such that for all $G \in \cC$ and all $a \in G \setminus\{1\}$ we have $a^2\neq_{G} 1$. Then $\UPowWP(\GP(\mathcal{C}))$ belongs to
   $\uAC^0\big[\CeqL^{\UPowWP(\mathcal{C})}\big]$.
  \end{theorem}

  \begin{corollary}\label{cor:UPowWPRAAG}
    Let $\mathrm{RAAG}$ denote the class of finitely generated RAAGs given by an alphabet $X$ and an independence relation $I \subseteq X \times X$. Then $\UPowWP(\mathrm{RAAG})$ is in $\uAC^0\big[\CeqL\big] \sse \uNC^{2}$.
  \end{corollary}

  \begin{remark}
     One can consider variants of the power word problem, where the exponents are not given in binary representation but in even more
    compact forms. \emph{Power circuits} as defined in \cite{MyasnikovUW12} are such a representation that allow
    non-elementary compression for some integers. Our logspace algorithm for the power word problem in a RAAG involves
    iterated addition and comparison (for equality) of exponents.
    For arbitrary power circuits, unfortunately, comparison for \emph{less than} is \P-complete and the complexity for equality checking is unknown. However, if we restrict to certain normal forms, called reduced power circuits, both iterated addition and comparison (for equality and for less than) are in \uTCz \cite{MattesW22}. Therefore, our techniques show that the power word problem for RAAGs with exponents given by reduced power circuits is also \uACz-Turing-reducible to the word problem for the free group $F_2$.
  \end{remark}

  \section{Consequences for the knapsack problem in right-angled Artin groups}

  Recall that the knapsack problem for a  finitely generated group $G$ asks,
  whether for given group elements $g_1, \ldots, g_n,g \in G$ (represented by
  words over generators) there exist $x_1, \ldots, x_n \in \N$ such that
  $g_1^{x_1} \cdots g_n^{x_n} =_G g$ holds. Using our results on power word
  problem, we can show the following result, which solves an open problem from
  \cite{LohreyZ18}.

   \begin{corollary}
    \label{cor:knapsack}
    The uniform knapsack problem for RAAGs is $\NP$-complete:
    On  input of a RAAG $G = G(X,I)$, given by the graph $(X,I)$, and $u_1, \dots, u_n, u \in (X \cup \finv{X})^*$, it can be decided in $\NP$ whether there are $x_1, \dots, x_n \in\mathbb{N}$ with  $u_1^{x_1} \cdots u_n^{x_n} =_G u$\text{.}
  \end{corollary}

  \begin{proof}
    Let $N = |X| + |u| + \sum_{i=1}^n |u_i|$ (this is roughly the input size).
    By~\cite[Theorem 3.11]{LohreyZ18}, there is a polynomial $p(N)$ such that if there is a solution, then there is a solution $x_1, \dots, x_n$ with $x_i \le 2^{p(N)}$.
    Therefore, we can guess a potential solution in polynomial time.
    From \cref{cor:UPowWPRAAG} it follows that the uniform power word problem in RAAGs belongs to $\P$.
    Hence, the uniform knapsack problem can be decided in $\NP$.
    Finally, $\NP$-completeness follows immediately from the $\NP$-completeness of the knapsack problem for a certain fixed RAAG, which has been shown in~\cite{LohreyZ18}.
  \end{proof}
  Note that this proof even shows \NP-completeness of the slightly more general problem of uniformly solving exponent equations for RAAGs as defined in~\cite{LohreyZ18}.

\section{Open Problems}
We strongly conjecture that the requirement $a^2 \neq 1$ can be dropped in all our results (as falsely claimed in \cite{StoberW22,StoberW22arxivOld}). Indeed, we believe that our methods can be extended to cope with that case. Still, this is a highly non-trivial question for further research.

Furthermore, we conjecture that the method of \cref{ch:the-power-word-problem-in-graph-products} can similarly be applied to hyperbolic groups, and hence that the
power word problem for a hyperbolic group $G$ is \uAc{0}-Turing-reducible to the word problem for $G$.
One may also try to prove transfer results for the power word problem with respect to group theoretical
constructions other than graph products, e.g., HNN extensions and amalgamated products over finite subgroups.
For a transfer result with respect to wreath products, see \cite[Proposition 19]{FigeliusGLZ20}. However, many cases are still open.

For finitely generated linear groups, the power word problem leads to the problem of computing matrix powers
with binary encoded exponents. The complexity of this problem is open;
variants of this problem have been studied in \cite{AllenderBD14,GalbyOW15}.

Another open question is what happens if we allow nested exponents.
We conjecture that in the free group for any nesting depth bounded by a constant the problem is still in
$\uACz(\WP(F_2))$. However, for unbounded nesting depth it is not clear what happens:
we only know that it is in \P since it is a special case of the compressed word problem; but it still could be in
$\uACz(\WP(F_2))$ or it could be \P-complete or somewhere in between.

\nocite{label}
\bibliography{sn-bibliography}% common bib file
%% if required, the content of .bbl file can be included here once bbl is generated

\end{document}